\documentclass[11pt,reqno,a4paper]{amsart}

\usepackage[margin=0.92in]{geometry}
\usepackage{amsmath,amssymb,amsthm}
\usepackage[numbers,square]{natbib}
\usepackage{mathtools}
\usepackage{graphicx}
\usepackage{tikz}
\usepackage{tikz-cd}
\usepackage{booktabs}
\usepackage{array}
\usepackage[colorlinks=true,linkcolor=blue,citecolor=blue]{hyperref}
\usepackage[all]{hypcap}

\usepackage{enumitem}

\newtheorem{theorem}{Theorem}[section]
\newtheorem{lemma}[theorem]{Lemma}
\newtheorem{proposition}[theorem]{Proposition}

\newtheorem*{maintheorem}{Main Theorem}
\theoremstyle{definition}
\newtheorem{definition}[theorem]{Definition}
\newtheorem{remark}[theorem]{Remark}

\usetikzlibrary{arrows.meta}

\newcommand{\C}{\ensuremath{\mathbb{C}}}
\newcommand{\R}{\ensuremath{\mathbb{R}}}

\newcommand{\g}[1]{\ensuremath{\mathfrak{#1}}}
\newcommand{\s}[1]{\ensuremath{\mathsf{#1}}}

\DeclareMathOperator{\tr}{tr}

\DeclareMathOperator{\Id}{Id}
\DeclareMathOperator{\Ad}{Ad}

\DeclareMathOperator{\ad}{ad}
\DeclareMathOperator{\Exp}{Exp}
\DeclareMathOperator{\Int}{Int}
\DeclareMathOperator{\spann}{span}
\DeclareMathOperator{\diag}{diag}
\DeclareMathOperator{\Ric}{Ric}
\DeclareMathOperator{\ric}{ric}

\DeclareMathOperator{\Isom}{Isom}

\renewcommand{\H}{\ensuremath{\mathbb{H}}}

\title[Inhomogeneous Einstein metrics on complex projective spaces]{Inhomogeneous Einstein metrics\\ on complex projective spaces}
\author[G.~Cao-Labora]{Gonzalo Cao-Labora}
	\address{Institute of Mathematics, EPFL, Station 8, Lausanne, VD 1015, Switzerland}
	\email{gonzalo.caolabora@epfl.ch}
	\author[A.~Rodr\'iguez-V\'azquez]{Alberto~Rodr\'iguez-V\'azquez}
	\address{Department of Mathematics, Universidade de Santiago de Compostela, Spain}
	\email{a.rodriguez@usc.es}
	
	\thanks{The first author G.~Cao-Labora was supported by the Swiss State Secretariat for Education, Research and Innovation (SERI) under contract number MB22.00034 through the project TENSE. The second author A.~Rodr\'iguez-V\'azquez was supported by the Horizon Europe research and innovation program under Marie Sklodowska-Curie Actions with grant agreement 101149711 - HOLYFLOW and the program Ram\'on y Cajal with grant number RYC2024-050605-I funded by MICIU/AEI/10.13039/501100011033 and by ESF+.}
\date{}

\begin{document}

\begin{abstract}
The only Einstein metrics currently known on complex projective spaces are  homogeneous: the Fubini-Study metric, arising via the Hopf fibration; and, in odd complex dimensions, Ziller's metric, obtained as a canonical variation along the twistor fibration over the quaternionic projective space. In 1965, Berger proved that the Fubini-Study metric is the unique K\"ahler--Einstein metric on the complex projective space, and posed the question of whether other Einstein metrics exist. We construct the first inhomogeneous Einstein metrics on $\C\s P^n$ for $3\leq n \leq 7$, answering Berger's question affirmatively for the complex even-dimensional $\C\s P^4$ and $\C\s P^6$. 
\end{abstract}

\maketitle
\vspace{-0.7cm}

\noindent


\section{Introduction}
A Riemannian manifold $(M, g)$ is Einstein if its Ricci tensor satisfies $\Ric(g)=\lambda g$ for some constant $\lambda\in\R$.  We now review some known results on Einstein manifolds.

When $\lambda < 0$, Böhm and Lafuente~\cite{BL23} proved that any homogeneous such $M$ is diffeomorphic to~$\R^n$.  Recall that a Riemannian manifold $(M,g)$ is Kähler if it has an almost Hermitian structure $J$ such that $\nabla J=0$. From now on we further assume that $M$ is compact.  If $(M, g, J)$ is Kähler with first Chern class satisfying $c_1(M) < 0$, or equivalently, the canonical bundle $K_M$ is ample, so that $-c_1(M)$ is a Kähler class, then, by independent results of Aubin~\cite{A78} and Yau~\cite{Y78}, $M$ admits a Kähler-Einstein metric with $\lambda<0$. The only known compact examples with $\lambda<0$ that are both non-K\"ahler and not locally homogeneous are constructed using Dehn fillings of hyperbolic cusps~\cite{An06, Bam12}, and certain examples with negative sectional curvature due to~\cite{FP20} in dimension $4$ and~\cite{HJ24} in every dimension $\geq 4$.

In the case that $\lambda=0$,  it remains an open question whether every simply connected compact Ricci-flat manifold has special holonomy. It is relevant to notice that under the special holonomy assumption the second order equation $\Ric(g)=\lambda g$ reduces to a first order equation. The case of holonomy $\s{SU}_n$ corresponds to Ricci-flat Kähler metrics, for which Yau's resolution of the Calabi conjecture~\cite{Y78} yields a vast class of examples on every compact Kähler manifold with vanishing first Chern class; see also~\cite{Ba94, CDLS88} for explicit constructions.  A non-exhaustive list of constructions of compact manifolds for the remaining special holonomies is: \cite{Be83, OG99} (with holonomy $\s{Sp}_n$), \cite{J96_1,J96_2, CHNT15, JK21}  (with holonomy $\s G_2$ or $\s{Spin}_7$).

Let us now assume $\lambda > 0$. In the homogeneous case, the restriction to $\lambda > 0$ is natural, as $\lambda = 0$ forces $g$ to be flat~\cite{AK75}, and Bochner's theorem implies that there is no compact homogeneous Einstein manifold with $\lambda < 0$. The homogeneous setting is by now well understood, with abundant general constructions~\cite{WZ85, WZ86} as well as obstructions~\cite{B04}, though interesting questions remain unanswered (see~\cite{BK23} for a list of open problems in this setting). In the K\"ahler case, a major source of compact Einstein metrics comes from the theorem of Chen, Donaldson and Sun: a compact K\"ahler manifold $(M,g,J)$ with $c_1(M) > 0$ admits a Kähler-Einstein metric if and only if it is $K$-polystable~\cite{CDS15_1, CDS15_2, CDS15_3}. The simplest examples of $K$-polystable complex manifolds with $c_1(M)>0$ are the complex projective spaces $\C\s P^n$, which play a fundamental role in complex, Kähler and algebraic geometry. Indeed, the fundamental example of a Kähler-Einstein metric is provided by the Fubini-Study metric on~$\C\s P^n$, which is homogeneous, $\s{SU}_{n+1}$-invariant, and is known to be the unique K\"ahler--Einstein metric on~$\C\s P^n$~\cite{Ber65, Ma72}.

Berger~(see~\cite[Remark~7.4]{Ber65}) raised the question whether  $\C\s P^n$ admits a second Einstein metric. In odd complex dimension, such a metric does
exist. This rests on a higher-dimensional analogue of the twistor fibration
$\C\s P^3 \to \s S^4$, which gives an example of a twistor space, a notion
introduced by Penrose~\cite{Pen67} in the setting of
relativistic physics and later given a Riemannian formulation by
Atiyah, Hitchin, and Singer~\cite{AHS78}. Realizing $\C\s P^{2n+1}$ as the twistor space
over the quaternionic projective space $\H\s P^n$, Ziller obtained in~\cite{Z82} a homogeneous
$\s{Sp}_{n+1}$-invariant Einstein metric, by taking the Fubini-Study metric of minimal sectional curvature~$1$ and
shrinking the $\s S^2$-fibers of the twistor fibration by a factor of~$\tfrac{1}{n+1}$, see also~\cite[Example~9.83]{Be08}. Ziller~\cite{Z82} further showed that the
Fubini-Study and the $\s{Sp}_{n+1}$-invariant metric exhaust all homogeneous Einstein metrics
on $\C\s P^{2n+1}$.
\smallskip

In this article we provide the first examples of inhomogeneous Einstein metrics on~$\C\s P^n$, answering Berger's question for the even complex dimensions $n=4,6$ affirmatively.

\phantomsection\label{th:main}
\begin{maintheorem}
There is an inhomogeneous Einstein metric $g$ on $\C\s P^n$ for each $n\in\{3,4,5,6,7\}$.
\end{maintheorem}

The metric $g$ on $\C\s P^n$ is not Hermitian (see Proposition~\ref{prop:notHermitian}), and thus, it is not conformally Kähler, see Remark~\ref{rem:Hermitian}. Moreover, it is  not nearly K\"ahler with respect to the standard non-integrable almost complex structure on the twistor space $\C\s P^{n}$ with $n$ odd, see Remark~\ref{rem:nonNK}. It turns out that $g$ does not have non-negative sectional curvature~(see~Remark~\ref{rem:sec>0}), and it has smaller Yamabe energy than the homogeneous examples, see Table~\ref{table:yamabe} in Appendix~\ref{sec:appendix_A}. Furthermore, the metric $g$ is $\s{SO}_{n+1}$-invariant and of cohomogeneity one. The two singular orbits of the cohomogeneity-one action of $\s{SO}_{n+1}$ are a totally geodesic Lagrangian $\R\s P^n$ and the complex quadric $\s Q_{n-1}$, which is Kähler; while the principal orbits are diffeomorphic to the unit tangent bundle $T_1 \R\s P^n$. We observe numerically no $\s{SO}_{n+1}$-invariant examples for $n \ge 8$ (see Appendix~\ref{sec:appendix_A}); moreover, the dimensions $n = 3$ and $n = 7$ seem special for topological reasons, as the principal orbits are diffeomorphic to products in that case  (see~Remark~\ref{rem:n3n7}), and $n = 2$ is special for analytical reasons (see Remark~\ref{rem:nge3}). Numerical plots of the ansatz functions defining the metric $g$ for each $n\in\{3,4,5,6,7\}$ are shown in Figure~\ref{fig:plots}.

Our examples have real dimension strictly greater than $4$. It must be pointed out that Einstein $4$-manifolds are rather special, as dimension four is the first in which the Einstein condition does not force constant sectional curvature. We refer
to~\cite{An10} for a survey on Einstein $4$-manifolds. Exploiting the particular structure of
the curvature operator in dimension~$4$, Broder~\cite{Br26} classified
the Einstein metrics of cohomogeneity one on closed $4$-manifolds. This,
combined with the \hyperref[th:main]{Main Theorem}, implies that
$n = 3$ is the smallest value for which $\C\s P^n$ admits an
inhomogeneous cohomogeneity-one Einstein metric. It is also worth noting that a Hermitian--Einstein metric on $\C\s P^2$ must be isometric to the Fubini-Study by~\cite[Theorem~A]{L12} and thus homogeneous. To our knowledge, it remains an open question whether every Hermitian--Einstein metric on $\C\s P^n$ with $n > 2$ is homogeneous.

The first inhomogeneous Einstein metric on a compact manifold was found by
Page~\cite{Pa78} on the Del Pezzo surface $\C\s P^2 \# \overline{\C\s P^2}$, the blow-up of $\C\s P^2$
at a point. Page's metric is of cohomogeneity one, and it inspired a family of cohomogeneity-one Einstein metrics in dimensions greater than four~\cite{BB82, PP, B98, Ch24}. For other Del Pezzo surfaces,
Tian~\cite{T90} proved that a Kähler-Einstein metric exists if and only if the surface is not the blow-up of $\C\s P^2$ at one or two points. In contrast with the metric we construct on $\C\s P^n$ for $n \in \{3, 4, 5, 6, 7\}$, the two exceptional
 Del Pezzo surfaces carry Hermitian, conformally-K\"ahler (non-Kähler) Einstein metrics: Page's metric on
$\C\s P^2 \# \overline{\C\s P^2}$, and, on $\C\s P^2 \# 2\overline{\C\s P^2}$, an Einstein metric of cohomogeneity two constructed by Chen-LeBrun-Weber~\cite{CLW08}.

Much of the attention in the search for Einstein metrics has been
directed towards spheres: firstly, B\"ohm~\cite{B98} found inhomogeneous Einstein metrics on $\s S^n$ for $n \in \{5, 6, 7, 8, 9\}$ with cohomogeneity one, and later Boyer, Galicki, Ghigi and Koll\'ar~\cite{BGK05,GK} produced inhomogeneous
Einstein metrics on every odd-dimensional sphere $\s S^{2n+1}$ with $n\ge 2$. Foscolo and Haskins~\cite{FH17} constructed a
cohomogeneity-one nearly K\"ahler metric on $\s{S}^6$, and thus, Einstein. More recent constructions yield cohomogeneity-one Einstein metrics on $\s S^{10}$ and $\s S^{12}$~\cite{NW23, BH, W26}.

In this article, we develop a new strategy for constructing cohomogeneity-one Einstein metrics. Most prior work has focused on $\s{SO}_k \times \s{SO}_{n+1-k}$-invariant metrics on $\s S^n$~\cite{B98, NW23, BH, W26},
satisfying $1 \leq k \leq n-1$. From the algebraic viewpoint, this is a
particularly simple action: the simplest cohomogeneity-one action on the sphere after the isotropy action itself, whose associated ansatz consists of the
radially symmetric metrics on $\s S^n$ (the only Einstein representative
being the round metric). The ansatz corresponding to the $\s{SO}_k \times \s{SO}_{n+1-k}$-action depends on only two functions, whereas the $\s{SO}_{n+1}$-invariant metrics on $\C\s P^n$ that we consider depend on three, see~\eqref{eq:ansatz}. This more complicated ansatz brings additional difficulties. The smoothness conditions at each singular orbit imply that the family of smooth local solutions around each singular orbit is parametrized by two parameters, rather than one. As a consequence, obtaining global smooth Einstein metrics amounts to a four-dimensional matching problem in place of a two-dimensional one. The two-dimensional setting is better suited to shooting arguments, as in B\"ohm~\cite{B98} and Nienhaus--Wink~\cite{NW23},
although even there computer assistance becomes highly relevant when the sphere $\s S^n$ has dimension $n>10$, as in~\cite{BH} and~\cite{W26}.

Our approach is also computer-assisted, but in a fundamentally different way. Rather than using a Schauder fixed-point theorem on a Banach space to produce an exact solution near a highly accurate numerical one as in~\cite{BH, W26}, we use computer-assistance via a fixed-point in a finite dimensional space (concretely, of dimension $4$). The computer enters only through an argument to show that a matching map $\mathcal G  \colon U \subset \mathbb R^4 \to \mathbb R^4$ has a zero. Consequently, Lemma~\ref{lem:matchingzero} is the only place in the paper where the computer is invoked. We believe that our methods can be extended to ansätze depending on more functions, opening the way to the discovery of further inhomogeneous Einstein metrics. We can state our strategy as follows:
\begin{enumerate}
	\item[\textup{(1)}] We formulate the problem as the existence of a global solution to an 	ODE on an interval $[0, T]$ satisfying prescribed smoothness conditions 	at both endpoints. The resulting ODE is autonomous and six-dimensional (corresponding to three second order ODEs).

	\item[\textup{(2)}] We employ a conserved quantity~(see~Remark~\ref{rem:firstintegral}), and reparametrize using the mean curvature $H \in (-\infty, \infty)$, obtaining a four-dimensional non-autonomous system in $H$ which is equivalent to the original one.
	
	\item[\textup{(3)}] We construct the solution for $H \in (-\infty, H_-]$ and $H \in [H_+, \infty)$ with $-H_-, H_+$ sufficiently large, each of the solutions depending on two parameters. We do this by giving explicit recurrences for the coefficients in a series, showing its convergence, and obtaining explicit bounds on the remainders.
	
	\item[\textup{(4)}] We define a four-dimensional matching function $\mathcal G$ by constructing our local solutions at $\pm \infty$ (each depending on two parameters) and testing the difference between a forward-propagated ODE from $H_-$ to $0$ and a backward-propagated one from $H_+$ to $0$.
	
	\item[\textup{(5)}] We prove that $\mathcal G$ has a zero by applying Krawczyk's fixed point theorem to a small hypercube in $\R^4$.
\end{enumerate}

The organization of the paper is as follows. Section \ref{sec:prelim} introduces the necessary preliminaries. Step~(1) of our strategy is developed in Section \ref{sec:so_n+1-action}, culminating in the ODE system \eqref{eq:E1}--\eqref{eq:E4} that $\s{SO}_{n+1}$-invariant Einstein metrics satisfy, and in the identification of the corresponding smoothness conditions at the singular orbits. In Section \ref{sec:H-rep}, we carry out Step~(2), obtaining an equivalent four-dimensional ODE system parametrized by~$H$ (Subsection \ref{subsec:rewriting} translates the ODE and Subsection~\ref{subsec:smoothH} the smoothness conditions). In Section~\ref{sec:Hexpansions} we construct the local solutions around $H = \pm \infty$ by a convergent series expansion, with explicit remainder bounds completing Step~(3). In Section~\ref{sec:krawczyk} we define our matching function (Step~(4)), and combine Krawczyk's fixed-point theorem with computer-assistance to show it has a zero~(Step~(5)), concluding the proof of the \hyperref[th:main]{Main Theorem}. The only result with a computer-assisted proof is Lemma \ref{lem:matchingzero}, and the code can be found in file \texttt{CAP\_Einstein\_CPn\_SubmissionVersion.ipynb} in the TeX source of our arXiv version. Finally, in Appendix~\ref{sec:appendix_A}, we include some numerical plots of the solutions we found and we compare the Yamabe energies of all the currently known Einstein metrics on complex projective spaces. 

\subsection*{Acknowledgments:} We wish to thank Prof.~Fine for very helpful discussions and Prof.~Böhm for insightful comments that helped us sharpen the exposition at several points. The second author is grateful to the geometry group at the Universidade de Santiago de Compostela for their hospitality in hosting a course he gave on cohomogeneity-one manifolds in January 2026 and the discussions around which sparked this project.

\subsection*{AI usage statement:} LLMs, in the form of Claude Opus 4.6, 4.7 and 4.8 were used in early stages of the research process to code simulations. LLMs were also used in helping to produce the code which gave Figure~\ref{fig:c1_action}, and to detect typos and minor editing in earlier versions of this manuscript. The authors take full responsibility for all content.

\section{Preliminaries}
\label{sec:prelim}
\subsection{Cohomogeneity-one invariant metrics and smoothness conditions}
\label{subsec:c1theory}
In this subsection, we introduce and explain some well-known facts about cohomogeneity-one actions. For more details one can consult~\cite[\S 6.3]{ABbook}.
 
Let $\s{G}$ be a compact Lie group acting with cohomogeneity one on $M$, that is, the principal $\s{G}$-orbits have codimension one in $M$. Let us denote by $M^\mathrm{reg}$ the union of all principal orbits, by  $M/\s{G}$ the orbit space, and assume that $M/\s{G}$ is  homeomorphic to $[0,T]$ for some $T>0$. Let $\gamma\colon [0,T]\rightarrow M$ be a normal geodesic with $\gamma(t)\in M^{\mathrm{reg}}$ for $t\in(0,T)$ where $\gamma(0)=p_{-}$ and $\gamma(T)=p_+$  lie in the two distinct singular orbits. We denote by $\s{K}_{\pm}$ the isotropy group of $p_{\pm}$ and by $\s{H}$ the principal isotropy group, that is, the isotropy group of $\gamma(t)$ for any $t\in(0,T)$. Let us denote by $D_{\pm}$ a normal slice at $p_{\pm}$, respectively, see~\cite[\S 3.2]{ABbook} for more information on slices and the slice theorem. As the $\s{G}$-action is of cohomogeneity one, $\s K_{\pm}$ acts with cohomogeneity one on the normal slice $D_{\pm}$, and thus $\s{K}_{\pm}$ acts transitively on the unit spheres of $D_{\pm}$ implying that $\s{S}^{\ell_{\pm}}=\s{K}_{\pm}/\s{H}$, where $\ell_{\pm}=\dim(D_{\pm})-1$. With this notation, we define the group diagram
\begin{equation*}
	\begin{tikzcd}
		& \mathsf{G} & \\
		\mathsf{K}_- \arrow[ur] & & \mathsf{K}_+ \arrow[ul] \\
		& \mathsf{H} \arrow[ul] \arrow[ur] &
	\end{tikzcd}
\end{equation*}
where the arrows denote the natural inclusions. We also denote the above diagram by $\s{H}\leq\{\s K_{-}, \s{K}_{+}\} \leq \s{G}$. Notice that this diagram determines two homogeneous sphere bundles, one over each singular orbit $\s{G}/\s{K}_{\pm}$, namely
\begin{equation}
\label{eq:spherebundles}
\s{K}_{\pm}/\s{H}\rightarrow \s{G}/\s{H}\xrightarrow{\pi_{\pm}}\s{G}/\s{K_{\pm}}, \enspace \text{where $\pi_{\pm}(g\s{H})=g \s{K}_{\pm}$.}  
\end{equation}
Conversely, given compact Lie groups $\s{H}\leq\{ \s{K}_{-}, \s{K}_{+}\} \leq \s{G}$ such that $\s{K}_{\pm}/\s{H}$ are spheres, we can glue the disk bundles with fibers closed disks $\bar D_{\pm}$ of dimensions $\ell_{\pm}+1$ induced by the spherical bundles in \eqref{eq:spherebundles} along their boundaries to get a smooth manifold $M$ for which $\s G$ acts with cohomogeneity one.

 By continuity, any $\s{G}$-invariant metric on $M$ is determined by its restriction to $M^\mathrm{reg}$, and in our case $M^\mathrm{reg}=M\setminus(\s{G}/\s{K}_{+} \cup \s{G}/\s{K}_{-})$. Now we can choose a geodesic $\gamma$ that intersects all $\s{G}$-orbits orthogonally. The choice of this geodesic on $M$ determines a $\s{G}$-equivariant diffeomorphism $M^\mathrm{reg}\cong(0,T)\times \s{G}/\s{H}$ for which $\gamma(t)=(t,e\s{H})$. Thus, we can write any $\s{G}$-invariant metric $g$ on the regular part as
\begin{equation}
	\label{eq:ansatzcohomo1metric}
	g=dt^2 + g_t, \quad t\in(0,T); 
\end{equation}
where $g_t$ is a smooth $1$-parameter family of $\s G$-invariant metrics on $\s{G}/\s{H}$ satisfying some smoothness conditions as $t\to 0$ and $t\to T$, which we explain below.

Let $\mathcal{Q}$ denote a bi-invariant metric on $\s{G}$ and let $\g{n}$ be the orthogonal complement of $\g{h}$ in $\g{g}$. Thus, for each $t\in(0,T)$, we have the orthogonal splitting
\[T_{\gamma(t)} M= \R \dot{\gamma}(t) \oplus T_{\gamma(t)} \Sigma_t, \]
where $\dot{\gamma}(t)$ is the velocity of $\gamma$ and $\Sigma_t$ is the hypersurface $\s{G}\cdot \gamma(t)$. It is a standard fact in the theory of homogeneous spaces that we can identify $\g{n}$ and $T_{\gamma(t)}\Sigma_t$ for each $t\in(0,T)$ by considering $ X\in\g{n}\rightarrow X^*(\gamma(t))$,
where $X^*$ is the Killing vector field induced by $X\in\g{n}$, that is, $X_p^*:=\frac{d}{ds}\rvert_{ s=0}\Exp(s X)\cdot p$ for each $p\in M$.  Let us consider a smooth $1$-parameter family of $\mathcal{Q}$-symmetric $\Ad(\s{H})$-equivariant endomorphisms $\mathcal{P}_t\colon \g{n}\rightarrow \g{n}$ such that
\[g(X_{\gamma(t)}^*,Y_{\gamma(t)}^*)=g_t(X,Y)=\mathcal{Q}(\mathcal{P}_t X, Y), \quad \text{where $X,Y\in \g{n}$ and $t\in(0,T)$}. \]

The product structure $M^{\mathrm{reg}} \cong (0,T) \times \s G/\s H$ together with the warped form $g = dt^2 + g_t$ gives the block-diagonal volume form $d\mathrm{vol}_g = dt \wedge d\mathrm{vol}_{g_t}$. By Fubini's theorem,
\begin{equation}
	\label{eq:volformula}
	\mathrm{Vol}(M, g) = \int_{M^{\mathrm{reg}}} d\mathrm{vol}_g = \int_0^T   \int_{\Sigma_t} d\mathrm{vol}_{g_t}  dt = \int_0^T \mathrm{Vol}(\Sigma_t, g_t) dt=\mathrm{Vol}(\s G/\s H, \mathcal{Q}) \int_0^T \sqrt{\det \mathcal{P}_t} dt,
\end{equation}
where we have used that by $\s G$-invariance, the volume density of $g_t$ relative to that of $\mathcal{Q}$ equals the constant $\sqrt{\det \mathcal{P}_t}$, so $\mathrm{Vol}(\Sigma_t, g_t) = \sqrt{\det \mathcal{P}_t}\, \mathrm{Vol}(\s G/\s H, \mathcal{Q})$.

We also consider the shape operator $\mathcal{S}_t$ of the orbit $\Sigma_t$, which is given by $\mathcal{S}_t X^*=-\nabla_{X^*}\dot{\gamma}$, where $X\in\g{n}$, and $\nabla$ denotes the Levi-Civita connection associated with the cohomogeneity-one metric $g$. Notice that we follow the opposite sign convention for the shape operators from~\cite{B98}. However, we follow the same convention as in~\cite{GZ02}, and thus, by~\cite[Equation (1.6)]{GZ02}, we have
\begin{equation}
\label{eq:shape}
\mathcal{S}_t(X)=-\frac{1}{2}\, \mathcal{P}_t^{-1} \mathcal{P}_t' X, \quad \text{for each $X\in \g{n}$}.
\end{equation} 
Let us choose an $\Ad(\s{H})$-invariant splitting of $\g{n}$, namely:
\begin{equation}
\label{eq:ndescomposition}
\g{n}=\g{n}_0\oplus \g{n}_1\oplus \cdots \oplus \g{n}_r,
\end{equation}
where $\s{H}$ acts trivially on $\g{n}_0$, and irreducibly on $\g{n}_i$ for $i\ge 1$. In this work we will consider diagonal metrics with respect to a splitting of $\g{n}$, that is,  we will choose $\mathcal{P}_t$ of the form:
\begin{equation}
\label{eq:Pt}
\mathcal{P}_t=\diag\left(h^2 \Id_{\g{n}_0}, f^2_1 \Id_{\g{n}_1}, \cdots,f^2_r \Id_{\g{n}_r} \right)
\end{equation}
where $h, f_i\colon(0,T)\rightarrow \R$ are certain smooth positive functions with $i\in\{1,\ldots,r\}$.
Thus, by~\eqref{eq:shape}, the shape operator is diagonal and has the expression
\begin{equation}
\label{eq:eigshape}
\mathcal{S}_t=-\diag\left(\frac{h'}{h} \Id_{\g{n}_0}, \frac{f'_1}{f_1} \Id_{\g{n}_1}, \cdots,\frac{f'_r}{f_r}  \Id_{\g{n}_r} \right).
\end{equation}
Let $X, Y, Z, W$ be vector fields tangent to $ \Sigma_t$. Recall the Gauss equation along the hypersurface~$\Sigma_t$ and the Riccati equation along the geodesic $\gamma$, which read 
\begin{align}
	\label{eq:gauss}
	R(X, Y, Z, W) &= R^{t}(X, Y, Z, W) + g_t(\mathcal{S}_t X, Z) g_t(\mathcal{S}_t Y, W) - g_t(\mathcal{S}_t X, W) g_t(\mathcal{S}_t Y, Z),\\[2ex]
	\label{eq:riccati}
	\mathcal{S}_t' &= \mathcal{S}_t^2 + R_{\dot\gamma},
\end{align}
where $R$ is the curvature tensor of $(M,g)$,  $R_{\dot\gamma}$ is the Jacobi operator along $\gamma$, i.e. $R_{\dot\gamma}(X):= R(X, \dot\gamma) \dot\gamma$,  and $R^t$ denotes the curvature tensor of $\Sigma_t$. Notice that this implies
\begin{equation}
\label{eq:jacobit}
R_{\dot\gamma}=-\diag\left(\frac{h''}{h} \Id_{\g{n}_0}, \frac{f''_1}{f_1} \Id_{\g{n}_1}, \cdots,\frac{f''_r}{f_r}  \Id_{\g{n}_r} \right).
\end{equation}
Moreover, combining~\eqref{eq:gauss} and \eqref{eq:riccati}, and denoting by $H:=\tr\mathcal{S}_t$ the mean curvature of $\Sigma_t$, we have
\begin{equation}
\label{eq:rictan}
\Ric(X,X)=\Ric^t(X,X) - H g_t(\mathcal{S}_t X,X) + g_t(\mathcal{S}'_t X,X), \quad \text{where $X$ is tangent to $\Sigma_t$}.
\end{equation}
Now if $(M,g)$ is Einstein with constant $\lambda$, that is, $\mathrm{Ric}^g = \lambda g$, we have, setting $f_0:=h$,
for each $i \in \{0, 1, \ldots, r\}$,
\begin{equation}
	\label{eq:logform}
	(\log f_i)'' - H (\log f_i)' - \mathrm{ric}_i(g_t) =- \lambda,
\end{equation}
where $\ric_i(g_t):=\Ric_{g_t}(\hat X^*_i,\hat X^*_i)$ with $\hat X_i\in\g{n}_i$ a $g_t$-unit vector.
\medskip

Now we turn our attention to the problem of determining the smoothness conditions of the metric in \eqref{eq:ansatzcohomo1metric}. This problem has been solved by Verdiani and Ziller in~\cite{VZ22}.

A first observation is that, since this is a local problem, it suffices to study the smoothness conditions for cohomogeneity-one actions with one singular orbit. If our action has two singular orbits, we only need to repeat this process on the other singular orbit. Thus, our input data are compact Lie groups $\s{H}<\s{K}<\s{G}$ with $\s{K}/\s{H}\cong\s{S}^\ell$ and $\ell>0$. The action of $\s{K}$ on $\s{S}^\ell$ extends to a linear action on $\R^{\ell+1}$. Moreover, the slice theorem implies the existence of a slice for the $\s G$-action at a point $p_0\in M$ of the form $\exp_{p_0} B_r(0)$, where  $V=\nu_{p_0}(\s G\cdot p_0)$ denotes the normal space to the $\s G$-orbit at $p_0\in M$ and $B_r(0)\subset V$ is a small ball centered at $0$ in $V$. Thus, for our purposes, we may regard our cohomogeneity-one manifold $M$ to be equal to the homogeneous vector bundle $\s{G}\times_{\s{K}} V$, with fiber $V\cong \R^{\ell+1}$ over the singular orbit $\s{G}/\s{K}$. Let us consider a unit speed geodesic $\gamma\colon[0,\infty)\rightarrow M$ with $\gamma(0)=p_0$ intersecting all $\s{G}$-orbits orthogonally, and write the $\s{G}$-invariant metric on $M$ as $g=dt^2 + g_t$ with $t\in (0,\infty)$,
where $g_t$ is a $\s{G}$-invariant metric on the homogeneous space $\s{G}/\s{H}$.
A first relevant result in this discussion is the following, see~\cite{VZ22} for a proof.
\begin{theorem}[{\cite[Theorem~A]{VZ22}}]
	\label{th:vzsmoothness}
	Let $g=dt^2 + g_t$ with $t\in (0,\infty)$ be a Riemannian metric on $\s{G}\times_{\s{K}} (V\setminus\{0\})$,
	where $g_t$ is a $\s{G}$-invariant metric on the homogeneous space $\s{G}/\s{H}$. Then $g$ admits a smooth extension to the singular orbit $\s{G}/\s{K}$ if and only if it is smooth when restricted to the exponential image in $M$ of $P\cap B_r(0)$ for every $2$-plane $P$ of $V$ containing~$\dot{\gamma}(0)$.
\end{theorem}
In view of this, Verdiani and Ziller also showed that, using the classification of transitive actions on spheres, it is possible to require these conditions only on a finite set of $2$-planes containing the $\dot{\gamma}(0)$-direction. In our case (see discussion in~\S\ref{subsec:smoothnesscpn}), the situation further simplifies because $\s{H}$ acts irreducibly on the orthogonal complement of $\dot{\gamma}(0)$ in $V$, and thus, we only need to require these conditions in a single $2$-plane.

Let us consider  the homogeneous spherical bundle induced by the triple $\s H< \s K< \s G$ and the following subspaces of the Lie algebra $\g{g}$:
\[ \s{K}/\s{H}\rightarrow \s{G}/\s{H}\rightarrow \s{G}/\s{K}, \quad T_{e\s{H}}\s{K}/\s{H}\simeq \g{p}, \enspace  T_{e\s{K}}\s{G}/\s{K}\simeq \g{m}, \enspace T_{e\s{H}}\s{G}/\s{H}\simeq \g{n}= \g{p}\oplus\g{m},\]
where we use the standard identifications and $\g{n}$ is an $\Ad(\s{H})$-invariant complement of $\g{h}$ in $\g{g}$,  $\g{p}$ is an $\Ad(\s H)$-invariant complement of $\g{h}$ in $\g{k}$, and  $\g{m}$ is an $\Ad(\s{K})$-invariant complement of $\g{k}$ in $\g{g}$.  Set an $\Ad(\s{H})$-invariant splitting of $\g{n}$, namely $\g{n}=\g{n}_0\oplus \g{n}_1\oplus \cdots \oplus \g{n}_r$ as in~\eqref{eq:ndescomposition}.
Geometrically, $\g{p}$ can be identified with a subspace of $V$ and $\g{m}$ as the tangent space of the singular orbit $\s G/\s K$ at~$p_0$.

 Fix a vector $X\in\g{p}$ normalized in such a way that the map $\theta\in[0,2\pi)\mapsto \Exp(\theta X)$ is injective. Define $\s{L}:=\{\Exp(\theta X): \theta\in[0,2\pi]\}\cong\s S^1$ and for the choice of this fixed $X\in\g{p}$ set $a$ as the positive integer given by
\begin{equation}
	\label{eq:adef}
	a:=\frac{2\pi}{\theta_0}, \quad \text{where $\theta_0>0$ is the smallest number such that $\Exp(\theta_0 X)\in\s{H}$}, 
\end{equation}
or equivalently $a=|\s L\cap \s H|$. Notice that in our situation (see~\eqref{eq:splittingtuberpn} and \eqref{eq:splittingtubequadric}), $\g{p}$ is an $\s H$-irreducible module.  Hence, the definition of $a$ does not depend on the choice of elements $X\in\g{p}$ subject to the normalization above. Now observe that the group $\s{L}$ induces a decomposition into $\s{L}$-submodules of the subspace $\g{m}$, that is,
\begin{equation*}
	\label{eq:Ldec}
 \g{m}=\ell_0\oplus \ell_1\oplus \cdots \oplus\ell_s,
\end{equation*}
where $\s{L}_{\rvert \ell_0}=\Id$, $\s{L}_{\rvert \ell_k}=R(d_k\theta)$ for each $k\ge1$, with $R(\alpha\theta)$ denoting the standard anticlockwise rotation in $\R^2$ of weight $\alpha\in \mathbb Z_{>0}$, and $d_k$ positive integers. In order to verify that $\s{L}$ acts on a $2$-dimensional module with weight $\alpha\in\mathbb{Z}$, it suffices to find two linearly independent vectors $Y_1, Y_2$ in the $\s{L}$-module such that $[X,Y_1]=\alpha Y_2$ and  $[X,Y_2]=-\alpha Y_1$.

Now by Theorem~\ref{th:vzsmoothness}, it suffices to derive smoothness conditions for the metric $g$  restricted to the image under  $\exp_{p_0}$ of a small ball in any $2$-plane $\spann\{ X,\dot{\gamma}(0)\}$, with $X\in\g{p}\subset V$. Thus,   for our purposes, it suffices to derive smoothness conditions for inner products of the form $\langle Z, Z\rangle$ where $Z$ lies in $\g{p}$ or $\g{m}$ since $\g{m}$ is orthogonal to $\g{p}$ in the diagonal metrics we consider, see \S \ref{subsec:smoothnesscpn} and \eqref{eq:splittingtuberpn}, \eqref{eq:splittingtubequadric}. Thus, we will not discuss the mixed products between $\g{m}$ and $\g{p}$ although they can be found in~\cite[Table~C]{VZ22}. The smoothness conditions for $\langle Z_1, Z_2\rangle$ where $Z_1, Z_2\in\g{p}$ or $Z_1, Z_2\in\g{m}$  were derived by Verdiani and Ziller in~\cite{VZ22} and are reproduced below for the reader's convenience.

For $X\in\g{p}$, we have $g_{\gamma(t)}(X^*,X^*)=a^2 t^2 + t^4 \phi_0(t^2)$,
where $\phi_0$ is some smooth function and $a$ is defined as in~\eqref{eq:adef}. 
Let $\{Y_1,Y_2\}$ be a basis of $\ell_{i}$. We define the coefficients $\widetilde{g}_{km}:=g_{\gamma(t)}(Y^*_k, Y^*_m)$, where $k,m\in\{1,2\}$. 
Then $\widetilde g_{11}+\widetilde g_{22}=\phi_1(t^2)$, $\widetilde g_{11}- \widetilde g_{22}=t^{\frac{2d_i}{a}}\phi_2(t^2)$, and $\widetilde g_{12}=t^{\frac{2 d_i}{a}}\phi_3(t^2)$, for some smooth functions $\phi_1, \phi_2, \phi_3$. Now we study the inner products between two different $\s L$-submodules $\ell_i$ and $\ell_j$ of $\g{m}$, with $i\neq j$ and $i,j\ge 1$. Let $\{Y_1,Y_2\}$ and  $\{Z_1,Z_2\}$ be bases of $\ell_i$ and $\ell_j$ respectively, as above. Define the coefficients $\widehat{g}_{km}:=g_{\gamma(t)}(Y^*_k, Z^*_m)$ with $k,m\in\{1,2\}$. Then the smoothness conditions are:
\begin{align} \begin{split} \label{eq:pre_smooth}
	\widehat{g}_{11}+\widehat{g}_{22} &= t^{\frac{|d_i-d_j|}{a}}\psi_1(t^2), \qquad
	\widehat{g}_{11}-\widehat{g}_{22} = t^{\frac{d_i+d_j}{a}}\psi_2(t^2), \\
	\widehat{g}_{12}-\widehat{g}_{21} &= t^{\frac{|d_i-d_j|}{a}}\psi_3(t^2), \qquad
	\widehat{g}_{12}+\widehat{g}_{21} = t^{\frac{d_i+d_j}{a}}\psi_4(t^2),
  \end{split}  \end{align}
for some smooth functions $\psi_1,\psi_2,\psi_3,\psi_4$. Finally, for the inner products between $\ell_0$ and $\ell_i$ with $i\ge 1$, if $Y\in\ell_0$ and $\{Z_1,Z_2\}$ is a basis of $\ell_i$ as above, then
\[
g_{\gamma(t)}(Y^*,Z^*_k) = t^{\frac{d_i}{a}}\eta_k(t^2), \quad k\in\{1,2\},
\]
for some smooth functions $\eta_1,\eta_2$, while $g_{\gamma(t)}(Y^*,Y^*)$ is an even function of $t$.

\subsection{CAP and interval arithmetics}
\label{subsec:cap}

We will use analytical tools to reduce the proof of existence of our Einstein metrics to the existence of a zero of a certain matching function $\mathcal{G}\colon U \subset \mathbb R^4 \to \mathbb R^4$, and a computer-assisted proof (CAP) to prove the existence of a zero of $\mathcal G$. Loosely speaking, the smoothness conditions \eqref{eq:pre_smooth} at each singular orbit (either $t = 0$ or $t = T$) each generate a two-parameter family of possible solutions, locally defined in a neighbourhood of $t = 0$ or $t = T$, respectively. Hence, the domain of $\mathcal{G}$ should be understood as an open subset of $\mathbb R^4 = \mathbb R^2 \times \mathbb R^2$, the first $\mathbb R^2$ determining a local solution around $t = 0$ and the second $\mathbb R^2$ determining a smooth solution around $t = T$. The open set $U$ is taken so that both local solutions can be continued until a certain smooth orbit (at $t = t_0$, corresponding to the smooth orbit of zero mean curvature), and then the existence of a global smooth Einstein metric corresponds to the question of whether those local solutions can be smoothly glued at $t = t_0$. We will see that this is possible if four matching conditions are satisfied and those are precisely encoded in the codomain of $\mathcal{G}$, in such a way that $\mathcal{G} = 0$ determines local solutions at $t = 0$ and $t = T$ that can be glued smoothly at $t = t_0$ (see Section \ref{sec:krawczyk} for a precise definition of $\mathcal{G}$). 

 The type of computer-assistance that we will employ is interval arithmetic, which goes back to the pioneering work of Moore \cite{Moore}. We refer to the survey of G\'omez-Serrano \cite{GS19} for a more recent treatment of the subject, with explicit applications. The main idea is to replace numbers by intervals and operate in such a way that we have rigorous guarantee that the number to be computed is in the corresponding interval. Let $x^I = [x^L, x^R]$ and $y^I = [y^L, y^R]$ be intervals such that $x\in x^I$ and $y \in y^I$. The set theoretic definition of an operation $\ast$ acting on those is $x^I \ast y^I = \{ x \ast y \; : \: x \in x^I, \; y \in y^I \}$. In contrast, when working with computer-assisted proofs we simply compute $x^I \ast y^I$ ensuring $x^I \ast y^I \supset \{ x \ast y \; : \: x \in x^I, \; y \in y^I \}$. On the one hand, this relaxation is still enough to rigorously establish the desired result: $x \ast y \in x^I \ast y^I$. On the other hand, this definition is relaxed enough that we can represent $x^I \ast y^I$ with the computer. For example, if $x^I = [x^L, x^R]$ and $y^I = [y^L, y^R]$, the sum is computed as
\begin{equation} \label{eq:CAP_sum}
x^I + y^I = [(x^L + y^L)_-, (x^R +y^R)_+]
\end{equation}
where $z_-$ (resp.\ $ z_+$) denotes the downward (resp.\ upward) truncation of $z$. That is, when $z$ is not representable by the computer, we truncate it to the lower nearest (resp.\ upper nearest) representable number. The same can be done with other elementary operations. More complicated functions such as $\sin, \cos, \exp, \ldots$ can be computed, for example, by using the Taylor formula up to a very high order (which is just elementary operations) and absorbing the remainder via an explicit bound on the Taylor series. 

Similarly, if we have a function $F\colon \mathbb R^n \to \mathbb R^m$, we can consider its interval-arithmetic implementation $F^I\colon \mathcal I (\mathbb R)^n \to \mathcal I ( \mathbb R )^m$ to be a function sending real boxes $x^I_1 \times x^I_2 \times \cdots \times x^I_n$  to real boxes $y^I_1 \times y^I_2 \times \ldots \times y^I_m$ in such a way that if $x_i \in x^I_i = [x^L_i, x^R_i]$ for every $i = 1, 2, \ldots, n$, then it must hold that $y = F(x) \in y_1^I \times y_2^I \times \cdots \times y_m^I$. Notice that the same $F\colon \mathbb R^n \to \mathbb R^m$ admits infinitely many valid implementations. One practical example of this is the following: suppose that $F(x) = \sum_{n = 0}^\infty c_n x^n$ with $| c_n | \leq C_0 C_{\mathrm{rad}}^n$, with $F\colon D = (0, 1/C_{\mathrm{rad}}) \to \mathbb R$. Here $1/C_{\mathrm{rad}}$ bounds the radius of convergence of the series. 

Then, for any $x^I = [x^L, x^R] \subset D$ we split $F$ into the sum of its first $m$ terms plus an enclosure of the remainder, the latter bounded by a geometric series. Thus we can implement:
\begin{equation} \label{eq:CAP_taylor}
F(x^I) = \sum_{n = 0}^{m-1} c_n \cdot (x^I)^n + \left[- \frac{C_0 C_{\mathrm{rad}}^m (x^R)^m}{1 - x^R C_{\mathrm{rad}}} , \frac{C_0 C_{\mathrm{rad}}^m (x^R)^m}{1 - x^R C_{\mathrm{rad}}} \right]
\end{equation}
which just involves the elementary operations $+$ and $\cdot$. In practice, we can also perform truncations as in \eqref{eq:CAP_sum} that preserve rigor, for instance, computing an upper truncation of $\frac{C_0 C_{\mathrm{rad}}^m (x^R)^m}{1 - x^R C_{\mathrm{rad}}}$. That is done by simply enclosing $C_0$, $C_{\mathrm{rad}}$ and $1$ by intervals $C_0^I$, $C_{\mathrm{rad}}^I$ and $1^I$, using interval arithmetic to compute $E^I = \frac{C_0^I (C_{\mathrm{rad}}^I)^m (x^I)^m}{1^I - x^I C_{\mathrm{rad}}^I}$.

An estimate of the type \eqref{eq:CAP_taylor} will be crucial for our interval-arithmetic implementation of $\mathcal{G}$. Concretely, in Section \ref{sec:Hexpansions} we will derive series expansions of the local solutions close to $t = 0$ and $t = T$ respectively (see Proposition~\ref{prop:RPseries} and Proposition~\ref{prop:QCseries}), with explicit bounds on the coefficients of the series. Since those series expansions are not enough to arrive at $t = t_0$, they will be combined with an ODE solver from the CAPD library \cite{KMWZ21}. This tool allows us to rigorously enclose the solution of a non-singular ODE. This ODE will correspond to the Einstein equation in the cohomogeneity-one framework. While the ODE becomes singular at the endpoints $t \in \{ 0, T \}$, we treat local neighbourhoods of $\{ 0, T \}$ with the series approach described above, so that the ODE solver is used on a compact interval of $(0, T)$ where the ODE is non-singular. Finally, that allows us to obtain an interval-arithmetic implementation of $\mathcal{G}$, which for every $x^I = x_1^I \times x_2^I \times x_3^I \times x_4^I$ returns a box $y_1^I \times y_2^I \times y_3^I \times y_4^I$ in such a way that we have the rigorous guarantee that $\mathcal{G}(x) \in y^I$ for all $x \in x^I$. That will allow us to prove the existence of a zero of $\mathcal G$ via Krawczyk's fixed point theorem (which we recall in Theorem \ref{th:krawczyk}).

In the context of finding solutions to ODEs, computer assistance has proven very powerful in recent years. In the context of Einstein metrics, we refer to the works of Buttsworth--Hodgkinson \cite{BH} and Wang \cite{W26}, where the solution is found by formalizing a numerical guess via a Schauder fixed point theorem. Our approach is different, since we directly show zeros of a finite-dimensional matching function rather than dealing with the infinite dimensional problem. While we are not aware of this approach being employed in geometric problems, we point to the work of Dahne and Figueras \cite{DF24}, where they applied a similar technique to solve a matching problem for an ODE determining the existence of self-similar solutions for the focusing nonlinear Schr\"odinger equation (and the closely related Ginzburg--Landau equations). The combination of series expansions with computer-assisted approaches to bound ODEs is also inspired by \cite{BCG25, CGSS26_1,CGSS26_2}.

\section{The cohomogeneity-one $\s{SO}_{n+1}$-action on $\C\s{P}^n$}
\label{sec:so_n+1-action}

In this section we study $\s{SO}_{n+1}$-invariant metrics on $\C\s P^n$. To this end we first explain the geometry underlying this cohomogeneity-one action. Then, we study cohomogeneity-one invariant metrics and determine smoothness conditions using the theory of \cite{VZ22}. Finally, we compute the Ricci tensor to get the Einstein equations.

Let us consider $\C^{n+1}$ and denote by $\{e_1,\ldots,e_{n+1}\}$ its canonical basis. Define an $\s{SO}_{n+1}$-invariant inner product given by $\langle z,w\rangle:=\sum_{i=1}^{n+1}z_iw_i$, for $z,w\in\C^{n+1}$, that is, the complex-bilinear extension of the Euclidean inner product of $\R^{n+1}$. The complex projective space is $\C\s P^n=\s{S}^{2n+1}/\s S^1$, where $\s{S}^{2n+1}:=\{z\in\C^{n+1}: \langle z,\bar z\rangle=1\}$. Let us equip $\C\s P^n$ with $g_{\mathrm{FS}}$, that is, the Fubini-Study metric scaled such that its minimum sectional curvature equals $1$, or equivalently, its diameter is equal to~$\pi/2$.

From now on we assume $n\ge 2$. Let us consider the standard action $\s{SO}_{n+1}\curvearrowright\C^{n+1}$. This action descends to a cohomogeneity-one action on $\C\s P^n$, see~\cite{Ta73} for the classification of isometric cohomogeneity-one actions on $\C\s P^n$ with respect to $g_{\mathrm{FS}}$. In what follows, we will describe the orbit structure of the action. To this end, let us introduce the $\s{SO}_{n+1}$-invariant function
\[
\chi\colon\C\s P^n\to[0,1], \qquad \chi([z])=|\langle z,z\rangle|,\quad\text{for }z\in\s{S}^{2n+1},
\]
which is well defined as $|\langle e^{i\theta}z,e^{i\theta}z\rangle|=|e^{2i\theta}||\langle z,z\rangle|=|\langle z,z\rangle|$ for every $\theta\in[0,2\pi)$. Let $p=[e_1]\in\R\s P^n$ and consider
\begin{equation*}
	\gamma(r)=[\cos(r)e_1+\sin(r)ie_2],\quad r\in[0,\pi/4].
\end{equation*}
The lift $\widetilde\gamma(r)=\cos(r)e_1+\sin(r)ie_2$ to the unit sphere $\s S^{2n+1}\subset\C^{n+1}$ is a unit-speed great circle and is horizontal with respect to the Hopf fibration $\s S^{2n+1}\to\C\s P^n$. Since this is a Riemannian submersion, its projection $\gamma$ is a unit-speed geodesic of $\C\s P^n$. A direct computation gives $\chi(\gamma(r))=|\cos(2r)|$, and thus $\chi$ is surjective. Indeed, let $[z]\in\C\s P^n$ with $z=u+iv\in\s S^{2n+1}$, with $u,v\in\R^{n+1}$. Multiplying $z$ by a unit complex number, we can assume without loss of generality that $u$ and $v$ are orthogonal with $|u|\ge|v|$. Let us write $|u|=\cos r$, $|v|=\sin r$ for $r\in[0,\pi/4]$. If $v\ne 0$, choose $A\in\s{SO}_{n+1}$ with $Ae_1=u/|u|$, $Ae_2=v/|v|$; then $A\widetilde\gamma(r)=z$, so $[z]$ lies on the orbit of $\gamma(r)$. If $v=0$, then $[z]=[u]\in\R\s P^n$ is on the orbit of $\gamma(0)$.   Moreover, $\chi\circ\gamma\colon[0,\pi/4]\longrightarrow[0,1]$ is injective, and thus $\chi$ parametrizes the $\s{SO}_{n+1}$-orbits yielding a homeomorphism between the orbit space $\C\s P^n/\s{SO}_{n+1}$ and $[0,1]$.

\begin{figure}[ht]
	\centering
\begin{tikzpicture}[
	>={Stealth[length=2.5mm]},
	geo/.style={thick},                
	orbit/.style={thin},               
	hidden/.style={thin, dashed},      
	perp/.style={line width=0.4pt}    
	]
	
	\coordinate (N) at (0, 3);
	\coordinate (S) at (0,-3);
	
	\draw[thick] (N) to[bend left=45]  (S);
	\draw[thick] (N) to[bend right=45] (S);
	
	\foreach \z/\r in {2.0/0.74, 1.0/1.12, 0/1.24, -1.0/1.12, -2.0/0.74} {
		\draw[hidden] (-\r, \z) arc [start angle=180, end angle=0,   x radius=\r, y radius=0.18];
		\draw[orbit]  (-\r, \z) arc [start angle=180, end angle=360, x radius=\r, y radius=0.18];
	}
	
	\draw[geo] (N) -- (S);
	
--- Orthogonality sqare---
	\draw[perp] (0, 1.94) -- (0.14, 1.94) -- (0.14, 1.82);
	\draw[perp] (0, 0.94) -- (0.14, 0.94) -- (0.14, 0.82);
	\draw[perp] (0, -0.07) -- (0.14, -0.07) -- (0.14, -0.17);
	\draw[perp] (0, -1.07) -- (0.14, -1.07) -- (0.14, -1.17);
	\draw[perp] (0, -2.07) -- (0.14, -2.07) -- (0.14, -2.17);
	
	\fill (N) circle (1.4pt);
	\fill (S) circle (1.4pt);
	
	\node[above=3pt] at (N) {$\s G/\s K_+ \cong \s Q_{n-1}$};
	\node[below=3pt] at (S) {$\s G/\s K_- \cong \R\s P^n$};
	\node[left]      at (-1.5, 0)   {$\s G/\s H \cong T_1\R\s P^n$};
	\node[right]     at (0.15, 0.5) {$\gamma$};
	
	\draw[->, thick] (2.7, 0.3) to[bend left=15] (4.7, 0.3);
	\node at (3.7, 0.95) {$\chi$};
	
	\coordinate (T) at (6, 3);
	\coordinate (B) at (6,-3);
	\draw[geo] (T) -- (B);
	\fill (T) circle (1.4pt);
	\fill (B) circle (1.4pt);
	\node[right=2pt] at (T) {\small $0$};
	\node[right=2pt] at (B) {\small $1$};
	\node[below=3pt] at (B) {$\C\s P^n / \s{SO}_{n+1} \cong [0,1]$};
\end{tikzpicture}
\caption{The orbit structure of the cohomogeneity-one action of $\s{SO}_{n+1}$ on $\C\s P^n$.}
\label{fig:c1_action}
\end{figure}
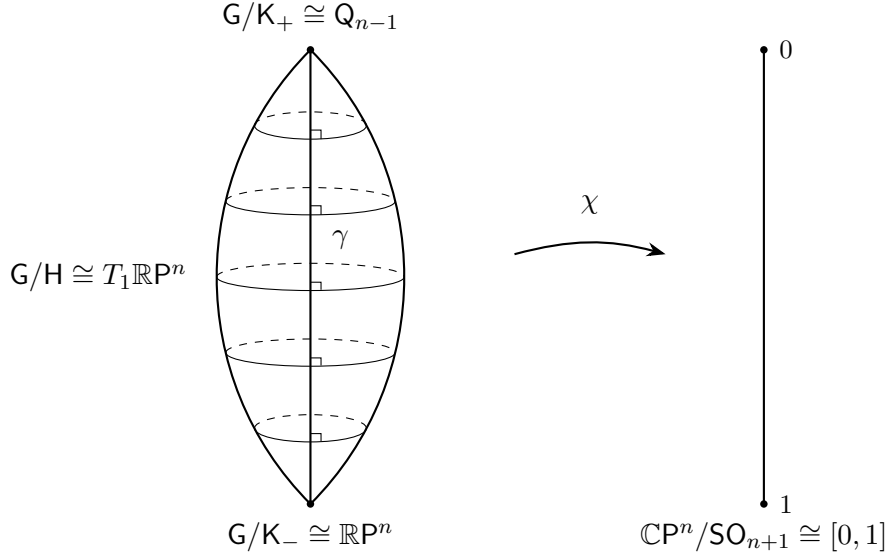

Now we will state the isotropies for the $\s{SO}_{n+1}$-action. Observe that it suffices to compute the isotropy of the points lying on the geodesic $\gamma$. It turns out that the isotropy at $\gamma(0)$, $\gamma(\pi/4)$ and  $\gamma(r)$ for $r\in(0,\pi/4)$ is equal to $\s{K}_-=\s{S}(\s O_1\times \s O_n)\cong \s O_n$,  $\s{K}_{+}=\s{SO}_2\times \s{SO}_{n-1}$, and $\s{H}=\mathbb{Z}_2\times\s{SO}_{n-1}$, respectively. Thus, the singular orbits are a real projective space $\R\s{P}^n=\s{G}/\s{K}_{-}$, which is given by those $[z]\in\C\s P^n$ such that $\chi([z])=1$; and the complex quadric $\s Q_{n-1}=\s{G}/\s{K}_{+}$, which is given by those $[z]\in\C\s P^n$ such that $\chi([z])=0$ and is diffeomorphic to the oriented Grassmannian of two planes in~$\R^{n+1}$.  Consequently, we have the following cohomogeneity-one diagram:
\[
\begin{tikzcd}[column sep=3em]
	& \s G = \s {SO}_{n+1} & \\
	\s K_- = \s O_n \arrow[ur,hook] & &
	\s K_+ = \s{SO}_2\times \s{SO}_{n-1} \arrow[ul,hook'] \\
	& \s H = \mathbb{Z}_2\times \s{SO}_{n-1}  \arrow[ul,hook] \arrow[ur,hook'] &
\end{tikzcd}
\]
\[
\s  K_-/\s H \cong \s{SO}_n/ \s{SO}_{n-1} = \s{S}^{n-1},\quad \ell_- = n-1;\qquad
\s  K_+/\s H \cong \s{S}^1,\quad\ell_+ = 1.
\]
Furthermore, we claim that the geodesic $\gamma$ meets every orbit orthogonally. The conjugation $c\colon[z]\mapsto[\bar z]$ is an isometry of $(\C\s P^n,g_{\mathrm{FS}})$ with fixed-point set the totally geodesic $\R\s P^n=\s{SO}_{n+1}\cdot[e_1]=\s{SO}_{n+1}\cdot \gamma(0)=\chi^{-1}(1)$. The involution $dc_p$ splits $T_p\C\s P^n=T_p\R\s P^n\oplus\nu_p\R\s P^n$ into its $\pm1$-eigenspaces. Since $c(\gamma(r))=\gamma(-r)$, we get $dc_p(\dot\gamma(0))=-\dot\gamma(0)$, hence $\dot\gamma(0)\perp T_p\R\s P^n$. As $\s{SO}_{n+1}$ acts by isometries, $g_{\mathrm{FS}}(\dot\gamma,X^*)$ is constant along $\gamma$ for every Killing field $X^*$ induced by $X\in\mathfrak{so}_{n+1}$, and thus $\gamma$ intersects orthogonally all $\s{SO}_{n+1}$-orbits. Also, the principal orbits $\s G/\s H$ can be regarded as tubes around the totally geodesic $\R\s P^n$ and  $\gamma$ realizes the distance from $\R\s P^n$ to each orbit it intersects, that is,  $d_{\mathrm{FS}}(\gamma(r),\R\s P^n)=r$ for every $r\in[0,\pi/4]$.

\subsection{$\s{SO}_{n+1}$-invariant metrics on $\C\s{P}^n$}
Let us start by studying the moduli space of $\s{G}$-invariant metrics on the principal orbit $\s{G}/\s{H}=\s{SO}_{n+1}/(\mathbb{Z}_2\times\s{SO}_{n-1})$.  Fix a background metric $\mathcal{Q}$ given by the $\Ad(\s{G})$-invariant inner product on the Lie algebra $\g{g}= \g{so}_{n+1}$ defined via $\mathcal{Q}(X,Y)=-\tfrac{1}{2}\tr(X\cdot Y)$, for $X,Y\in\g{g}$. A reductive decomposition of $\s{G}/\s{H}$ is given by $\g{g}=\g{h}\oplus\g{n}$, where
\[
\g{h}=\left\{\begin{pmatrix} 0 & 0 \\ 0 & A \end{pmatrix} : A\in\g{so}_{n-1}\right\},\qquad
\g{n}=\left\{\begin{pmatrix} aJ & B \\ -B^\top & 0 \end{pmatrix} : a\in\mathbb{R},\ B\in M_{ 2\times(n-1)}(\mathbb{R})\right\},
\]
where $J=\left(\begin{smallmatrix} 0 & 1 \\ -1 & 0\end{smallmatrix}\right)$, and $\g{n}=\g{h}^{\perp}$ is the orthogonal complement of $\g{h}$ with respect to $\mathcal{Q}$. Moreover, the isotropy representation has the following $\s{H}$-submodules:
\begin{equation}
\label{eq:nsplitting}
\mathfrak{n}_0:=\spann\{E_{1,2}\}, \quad \mathfrak{n}_1:=\spann\{E_{1,3}, E_{1,4}, \ldots, E_{1,n+1}\},\quad \mathfrak{n}_2:=\spann\{E_{2,3}, E_{2,4}, \ldots, E_{2,n+1}\},
\end{equation}
where $E_{i,j}$, with $i<j$ denotes the matrix with a $1$ in the position $(i,j)$ and a $-1$ in the position~$(j,i)$. Hence $\g{n}=\g{n}_0\oplus\g{n}_1\oplus\g{n}_2$, with $\g{n}_0$ trivial and $\g{n}_1\cong\g{n}_2$ the standard representation of $\s{SO}_{n-1}$ twisted by the sign action of $\mathbb{Z}_2=\diag(\pm 1, \pm 1, 1,\cdots,1)$ on each factor $\g{n}_i$ with $i\in\{1,2\}$. Although $\g{n}_1$ and $\g{n}_2$ are equivalent, by using the action of the normalizer $\s N_{\s G}(\s H)$, every $\s{G}$-invariant metric~(see~\cite[p.~120]{Ke98}) is induced by an $\Ad(\s H)$-invariant inner product on $\g{n}$ of the form
\begin{equation*}
	\label{eq:metricstiefel}
	\langle X, Y\rangle_{\lambda_0, \lambda_1, \lambda_2}= \lambda_0 \mathcal{Q}(X_{\mathfrak{n}_0},Y_{\mathfrak{n}_0}) + \lambda_1  \mathcal{Q}(X_{\mathfrak{n}_1},Y_{\mathfrak{n}_1}) + \lambda_2 \mathcal{Q}(X_{\mathfrak{n}_2},Y_{\mathfrak{n}_2}),
\end{equation*}
where $\lambda_i>0$ and $(\cdot)_{\g{n}_i}$ denotes the orthogonal projection to $\g{n}_i$ for each $i\in\{0,1,2\}$.
From now on we will only consider  $\s{SO}_{n+1}$-invariant metrics of the form:
\begin{equation}\label{eq:ansatz}
	g(X,Y) = dt^2
	+ h(t)^2 \mathcal{Q}(X_{\g{n}_0},Y_{\g{n}_0}) + f_1(t)^2 \mathcal{Q}(X_{\g{n}_1},Y_{\g{n}_1}) + f_2(t)^2 \mathcal{Q}(X_{\g{n}_2},Y_{\g{n}_2}),
\end{equation}
where $h$, $f_1$ and $f_2$ are positive smooth functions defined on $[0,T]$ and the variable $t$ is the time parameter of a unit normal geodesic $\gamma\colon[0,T]\rightarrow \C\s{P}^n$ such that $\gamma(0)\in \s G/\s K_{-}$, $\gamma(T)\in \s G/\s K_{+}$ and $\gamma(t)$ lies in a principal orbit when $t\in(0,T)$. 
\begin{remark}
\label{rem:n3n7}
	By~\cite[Theorem~5.1]{Da09}, see also~\cite[Example~5.2]{Da09}, every $\s{SO}_{n+1}$-invariant Einstein metric on $\C\s P^n$ with $n\ge 4$ (notice that \cite[Theorem~5.1]{Da09} requires the irreducible isotropy submodules to be of real type) can be globally diagonalized in the splitting $\g{n}=\g{n}_0\oplus\g{n}_1\oplus\g{n}_2$ implying that every $\s{SO}_{n+1}$-invariant Einstein metric on $\C\s P^n$ is of the form \eqref{eq:ansatz} when $n\ge 4$. Moreover, one can check that $\s O_{n+1}$-invariant metrics correspond to the metrics on the ansatz~\eqref{eq:ansatz}. Also it is interesting to point out that the principal orbits, which are diffeomorphic to $T_1 \R\s P^n$, $n\ge2$, are diffeomorphic to the product $\s{S}^{n-1}\times\R\s P^n$ if and only if $n\in\{3,7\}$, which correspond to the lower and upper bound of valid complex dimensions in the existence result in \hyperref[th:main]{Main Theorem}.
\end{remark}

\subsection{The smoothness conditions of diagonal $\s{SO}_{n+1}$-invariant metrics on $\C\s P^n$}
\label{subsec:smoothnesscpn}
Now we derive the smoothness conditions at the singular orbits $\s{G}/\s{K}_-=\R\s P^n$ and $\s G/\s K_{+}=\s Q_{n-1}$. Near $\s{G}/\s{K}_-$ the action is described by the homogeneous vector bundle induced by the triple $\s{H}= \mathbb{Z}_2\times \s{SO}_{n-1} < \s{K}_- = \s{O}_n < \s{G}=\s{SO}_{n+1}$, with $\s{K}_-/\s{H}\cong\s{S}^{n-1}$, so the slice has dimension $\ell_-+1=n$.  Following the notation established in \S\ref{subsec:c1theory}, we write
\[\g{n}=\g{p}\oplus\g{m}, \qquad \g{p}=T_{e\s{H}}\s{K}_-/\s{H},\quad \g{m}=T_{e\s{K}_-}\s{G}/\s{K}_-.\]
In terms of the splitting $\g{n}=\g{n}_0\oplus\g{n}_1\oplus\g{n}_2$ described in~\eqref{eq:nsplitting}, we have
\begin{equation}
\label{eq:splittingtuberpn}
\g{p}=\g{n}_2, \qquad \g{m}=\g{n}_0\oplus\g{n}_1.
\end{equation}
Indeed, $\g{m}$ can be identified with the tangent space to the singular orbit $\R\s P^n$ at $\gamma(0)=p$, while $\g{p}\oplus\R\dot{\gamma}$ can be identified with its normal space. The principal isotropy $\s{H}$ acts on $\g{p}\cong\R^{n-1}$ as the standard $\s O_{n-1}$-representation, which is irreducible. As $\s H$ acts transitively on the unit sphere of $\g{p}$, we can fix $X=E_{2,3}$, and define $\s{L}=\{\Exp(\theta X):\theta\in[0,2\pi]\}$. The smallest $\theta_0>0$ with $\Exp(\theta_0 X)\in\s{H}=\mathbb{Z}_2\times\s{SO}_{n-1} $ is $\theta_0=2\pi$. Therefore $a=\frac{2\pi}{\theta_0}=1$. We decompose $\g{m}$ into irreducible $\s{L}$-submodules. Now we have $\ad_X E_{1,2}=[E_{2,3},E_{1,2}]=-E_{1,3}$ and $\ad_X E_{1,3}=[E_{2,3},E_{1,3}]=E_{1,2}$. Hence $\spann\{E_{1,2},E_{1,3}\}\subset\g{m}$ is an irreducible $\s{L}$-submodule of weight $d=1$. For $j\ge 4$, $\ad_X E_{1,j}=[E_{2,3},E_{1,j}]=0$, so each $E_{1,j}\in\g{n}_1$ spans a trivial $\s L$-submodule, contributing to $\ell_0\subset\g{m}$.

Therefore, $\g{m}=\ell_0\oplus\ell_1$, where $\ell_0=\spann_{\R}\{E_{1,j}\}^{n+1}_{j=4}$ and $\ell_1=\spann_{\R}\{E_{1,2}, E_{1,3}\}$. Since the matrices $E_{i,j}$ are pairwise $\mathcal{Q}$-orthogonal and the metric $g$ in \eqref{eq:ansatz} is diagonal with respect to the splitting $\g{n}=\g{n}_0\oplus\g{n}_1\oplus\g{n}_2$, the inner product $\widetilde{g}_{12}$ within $\ell_1$ and the cross inner products between $\ell_0$ and $\ell_1$ (as in \S\ref{subsec:c1theory}) vanish identically. The corresponding smoothness conditions are therefore automatically satisfied, and it suffices to check the smoothness conditions for inner products of elements within each $\g{n}_i$ for $i\in\{0,1,2\}$. Thus, by  \S\ref{subsec:c1theory}, the metric $g$ extends smoothly to $\s G/\s K_{-}$ if and only if for a unit $Z\in\g{n}$, we have for some smooth functions $\phi, \psi$ and $\eta$:
\begin{align*}
	h(t)^2&=g_{\gamma(t)}(Z^*,Z^*)=\phi(t^2)+t^2 \psi(t^2), \quad \text{when $Z\in\g{n}_0$}\\
	f_1(t)^2&=g_{\gamma(t)}(Z^*,Z^*)=\phi(t^2)-t^2 \psi(t^2), \quad \text{when $Z\in\g{n}_1$}\\
	f_2(t)^2&=g_{\gamma(t)}(Z^*,Z^*)=t^2+t^4 \eta(t^2), \quad \text{when $Z\in\g{n}_2$},
\end{align*}
 Equivalently, we have for each $k\ge 1$ with $\xi>0$ the following:
\begin{equation}
	\label{eq:smoothproj}
	\begin{aligned}
		h(0)=\xi,& &h'(0)&=0,&  & h^{(2k+1)}(0)=0,\\
		f_1(0)=\xi,& &f_1'(0)&=0,&  &f_1^{(2k+1)}(0)=0,\\
		f_2(0)=0,& &f_2'(0)&=1,&  &f_2^{(2k)}(0)=0.
	\end{aligned}
\end{equation}
We now derive the smoothness conditions at the singular orbit $\s{G}/\s{K}_+$ (corresponding to $t=T$) via the change of variable $s=T-t$. Near $\s{G}/\s{K}_+$ the action is described by the homogeneous vector bundle induced by the triple $\s{H}=\mathbb{Z}_2\times\s{SO}_{n-1} < \s{K}_+ = \s{SO}_2\times\s{SO}_{n-1} < \s{G}=\s{SO}_{n+1}$, with $\s{K}_+/\s{H}\cong\s{S}^1$, so the slice has dimension $\ell_++1=2$. Following the notation of \S\ref{subsec:c1theory} (see also \cite{VZ22}),
\[\g{n}=\g{p}\oplus\g{m}, \qquad \g{p}=T_{e\s{H}}\s{K}_+/\s{H},\quad \g{m}=T_{e\s{K}_+}\s{G}/\s{K}_+.\]
In terms of the splitting $\g{n}=\g{n}_0\oplus\g{n}_1\oplus\g{n}_2$ described in~\eqref{eq:nsplitting}, we have
\begin{equation}
\label{eq:splittingtubequadric}
\g{p}=\g{n}_0, \qquad \g{m}=\g{n}_1\oplus\g{n}_2.
\end{equation}
Indeed, $\g{m}$ can be identified with the tangent space to the singular orbit $\s{G}/\s{K}_+$ at $\gamma(T)$, while $\g{p}\oplus\R\dot\gamma$ can be identified with its normal space.

Fix $X\in\g{p}$ of unit length, namely $X=E_{1,2}$, and define $\s{L}=\{\Exp(\theta E_{1,2}):\theta\in[0,2\pi]\}$. The smallest $\theta_0>0$ with $\Exp(\theta_0 E_{1,2})\in\s{H}=\mathbb{Z}_2\times\s{SO}_{n-1}$ is $\theta_0=\pi$, since $\Exp(\pi E_{1,2})=\mathrm{diag}(-1,-1,1,\ldots,1)$ generates the $\mathbb{Z}_2$-factor of $\s{H}$. Therefore $a=\frac{2\pi}{\theta_0}=2$.

\medskip
We now decompose $\g{m}$ into irreducible $\s{L}$-submodules.  For each $j\ge 3$: $\ad_X E_{1,j}=[E_{1,2},E_{1,j}]=-E_{2,j}$ and $\ad_X E_{2,j}=[E_{1,2},E_{2,j}]=E_{1,j}$. Hence $\spann\{E_{1,j},E_{2,j}\}\subset\g{m}$ is an irreducible $\s{L}$-submodule of weight $d=1$.
Therefore, $\g{m}=\bigoplus_{j=3}^{n+1}\ell_j$, where each $\ell_j=\spann_{\R}\{E_{1,j},E_{2,j}\}$ is an irreducible $\s{L}$-submodule of weight $d=1$.

In the explicit bases above, each $\ell_j$ is spanned by a pair $\{E_{1,j}, E_{2,j}\}$ for some $j\in\{3,\ldots,n+1\}$, with the first vector lying in $\g{n}_1$ and the second in $\g{n}_2$. Since the matrices $E_{i,j}$ are pairwise $\mathcal{Q}$-orthogonal and the metric~\eqref{eq:ansatz} is diagonal with respect to the splitting $\g{n}=\g{n}_0\oplus\g{n}_1\oplus\g{n}_2$, the inner products  $\widetilde{g}_{12}$ within each $\ell_i$ and all cross inner products $\widehat{g}_{km}$   between different submodules $\ell_i$ and $\ell_j$ (as in \S\ref{subsec:c1theory}) vanish identically. The corresponding smoothness conditions are therefore automatically satisfied, and it suffices to check the smoothness conditions for inner products of elements within each $\g{n}_i$ for $i\in\{0,1,2\}$.   Thus, by \S\ref{subsec:c1theory}, the metric $g$ in \eqref{eq:ansatz} extends smoothly to the singular orbit $\s G/\s K_{+}$ if and only if for $Z$ of unit length the following is satisfied:
\begin{align*}
	h(T-s)^2&=g_{\gamma(T-s)}(Z^*,Z^*)=4s^2+s^4\,\eta(s^2), && \text{when $Z\in\g{n}_0$},\\	f_1(T-s)^2&=g_{\gamma(T-s)}(Z^*,Z^*)=\phi(s^2)+s\,\psi(s^2), && \text{when $Z\in\g{n}_1$},\\
	f_2(T-s)^2&=g_{\gamma(T-s)}(Z^*,Z^*)=\phi(s^2)-s\,\psi(s^2), && \text{when $Z\in\g{n}_2$},
\end{align*}
where $\phi,\psi$ and $\eta$ are smooth functions. Equivalently,  for each $k\ge 1$, we have
\begin{equation}
	\label{eq:smoothquadric}
	\begin{aligned}
		h(T)&=0, &h'(T)&=-2, &h^{(2k)}(T)&=0,\\
		f_1(T)&=\zeta_0, &f_1'(T)&=\zeta_1, &f_1^{(k)}(T)&=\zeta_k,\\
		f_2(T)&=\zeta_0, &f_2'(T)&=-\zeta_1, &f_2^{(k)}(T)&=(-1)^k\zeta_k,
	\end{aligned}
\end{equation}
where $'$ denotes derivatives with respect to $t$, $\zeta_0>0$ and $\zeta_k\in\R$ for each $k\ge1$.
Consequently, by the discussion in \S\ref{subsec:c1theory}, the metric $g$ is smooth in $[0,T]$ if and only if it satisfies the conditions provided in \eqref{eq:smoothproj} and \eqref{eq:smoothquadric}.
\begin{remark}
	We remark that the functions $h^{\mathrm{FS}}(t)=\cos(2t)$, $f^{\mathrm{FS}}_1(t)=\cos(t)$ and $f^{\mathrm{FS}}_2(t)=\sin(t)$ satisfy the smoothness conditions at the ends $t=0$ and $t=T=\tfrac{\pi}{4}$. 
\end{remark}

\subsection{The Einstein equations of an $\s{SO}_{n+1}$-invariant metric on $\C\s{P}^n$}
The Einstein equations for a metric of the form \eqref{eq:ansatz} were derived in \cite[\S 4.3]{AGG26}. However, we present here an independent derivation using \cite[Proposition~1.14]{GZ02} for the sake of completeness. We note that \cite[Proposition 4.8]{AGG26} states that an $\s{SO}_{n+1}$-invariant Einstein metric on $\C\s P^n$ cannot have a totally geodesic singular orbit. This statement appears to overlook the Fubini-Study metric or the metric constructed in \hyperref[th:main]{Main Theorem}. In both cases, the singular orbit of the $\s{SO}_{n+1}$-action, namely, $\R\s P^n$, is totally geodesic.

Let $\{e_\alpha\}^{2n-1}_{\alpha=1}$ be a $\mathcal{Q}$-orthonormal basis of $\g{n}$ such that each $e_\alpha$ belongs to a subspace $\g{n}_i$ for some $i\in\{0,1,2\}$.  For each unit $X\in\g{n}_i$ with $i\in\{0,1,2\}$, we define the constants
\begin{equation}
\label{eq:constantsricci}
\Big[{}^{\,s}_{i\,r}\Big]
:=\sum_{e_\alpha\in\g{n}_r}\bigl\|[X,e_\alpha]_{\g{n}_s}\bigr\|_\mathcal{Q}^2,
\qquad
\omega_i:=\sum_{e_\alpha\in\g{n}_i}\bigl\|[X,e_\alpha]_{\g{h}}\bigr\|_\mathcal{Q}^2, \quad \text{where $i,r,s\in\{0,1,2\}$}.
\end{equation}
Notice that the constants in \eqref{eq:constantsricci} are well-defined by Schur's lemma as $\g{n}_i$ is an irreducible $\s H$-module for each $i\in\{0,1,2\}$.
\begin{lemma}\label{lem:structurect}
We have the following identities:
	\[
	\Big[{}^{2}_{0\,1}\Big]=\Big[{}^{1}_{0\,2}\Big]=n-1,\qquad
	\Big[{}^{2}_{1\,0}\Big]=\Big[{}^{0}_{1\,2}\Big]=\Big[{}^{1}_{2\,0}\Big]=\Big[{}^{0}_{2\,1}\Big]=1,
	\]
	and the rest of the coefficients $\Big[{}^{\,s}_{i\,r}\Big]$ are zero. Moreover,  $\omega_0=0$, $\omega_1=\omega_2=n-2$.
\end{lemma}

\begin{proof}
First of all, notice that $\Big[{}^{\,s}_{i\,r}\Big]=0$ for any triple $(i,r,s)$, which is not a permutation of $(0,1,2)$ because of the relations:
\begin{equation*}
	[\g{n}_0,\g{n}_1]\subset \g{n}_2,\qquad
	[\g{n}_1,\g{n}_2]\subset \g{n}_0,\qquad
	[\g{n}_2,\g{n}_0]\subset \g{n}_1.
\end{equation*}
Take $X=E_{1,2}\in\g{n}_0$. We have $[X,E_{1,k}]=-E_{2,k}\in\g{n}_2$ and $[X,E_{2,k}]=E_{1,k}\in\g{n}_1$ 	for $k=3,\dots,n+1$, while $[X,E_{1,2}]=0$. Summing the squared $\mathcal{Q}$-norms over each $\g{n}_r$ gives $	\Big[{}^{2}_{0\,1}\Big]=\Big[{}^{1}_{0\,2}\Big]=n-1$. Also $\omega_0=0$, since
	$[\g{n}_0,\g{n}_0]=0$. 	Take $X=E_{1,3}\in\g{n}_1$. Then $[X,E_{1,2}]=E_{2,3}\in\g{n}_2$
	(contributes $1$ to $\big[{}^{2}_{1\,0}\big]$), $[X,E_{2,3}]=-E_{1,2}\in\g{n}_0$
	(contributes $1$ to $\big[{}^{0}_{1\,2}\big]$), $[X,E_{1,\ell}]=-E_{3,\ell}\in\g{h}$
	for $\ell\geq 4$ (contributing $n-2$ to $\omega_1$), and $[X,E_{2,\ell}]=0$
	for $\ell\neq 3$. The case $X=E_{2,3}\in\g{n}_2$ is identical.
\end{proof}
 Now we are ready to compute the Ricci tensor. Let $X_i\in\g{n}_i$ be a unit vector for each $i\in\{0,1,2\}$ and denote $h$ by $f_0$ for convenience. By \cite[Proposition~1.14 \textup{(a)}]{GZ02}, we have 
\begin{equation}
\label{eq:ric_ni}
\Ric(X_i^*,X_i^*)=\omega_i	+
\sum^2_{r,s=0}
\frac{f_i^{4}-2f_r^{4}+2f_r^{2}f_s^{2}}
{4f_r^{2}f_s^{2}}
\Big[{}^{\,s}_{i\,r}\Big]
+
\left(
-\frac{f_i'}{f_i}
\sum^2_{r=0} \frac{f_r'}{f_r}\dim\mathfrak n_r
+\frac{(f_i')^2}{f_i^2}
-\frac{f_i''}{f_i}
\right)
f_i^2.
\end{equation}
 Moreover, we have that for $i\neq j$,  the brackets $[X_i,e_\alpha]$ and $[X_j,e_\alpha]$ lie in distinct modules
$\g{n}_k$. Thus, again by \cite[Proposition~1.14 \textup{(b)} and \textup{(c)}]{GZ02}, we have
\begin{equation}
\label{eq:ricmix}
\Ric(X_i,X_j)=\Ric(X_i,\dot{\gamma})=0,
\end{equation}
where $i\neq j$ and $\dot{\gamma}$ is the velocity of the normal geodesic $\gamma$. Finally, by \cite[Proposition~1.14 \textup{(d)}]{GZ02},
\begin{equation}
\label{eq:rictt}
\Ric(\dot{\gamma},\dot{\gamma})=-\frac{h''}{h} - (n-1)\frac{f''_1}{f_1} - (n-1) \frac{f''_2}{f_2}.
\end{equation}
Since each $\g{n}_i$ is $\Ad(\s H)$-irreducible, Schur's lemma implies that $\Ric|_{\g{n}_i\times\g{n}_i}$ is a multiple of $g|_{\g{n}_i\times\g{n}_i}$. Thus, combining Lemma~\ref{lem:structurect} and~\eqref{eq:ric_ni}, \eqref{eq:ricmix}, \eqref{eq:rictt} we have:
\begin{proposition}\label{prop:Einstein}
	The metric $g$ in \eqref{eq:ansatz} is an Einstein metric with Einstein constant $\lambda$, i.e., $\Ric=\lambda g$, if and only if the functions $h,f_1,f_2$ satisfy the following system of ODEs:
		\begin{align}
		\label{eq:E1}
		&\frac{h''}{h}
		+(n-1)\frac{h'}{h} \left(\frac{f_1'}{f_1}+\frac{f_2'}{f_2}\right)-\frac{(n-1)}{2} \left( \frac{h^2}{f_1^2 f_2^2} - \frac{f_1^2}{h^2 f_2^2}-\frac{f_2^2}{h^2 f_1^2}\right)
		+\frac{1-n}{h^2}
		=-\lambda,\\[2pt]
		\label{eq:E2}
		& \frac{f_1''}{f_1}
		+\frac{f_1'}{f_1} \left(\frac{h'}{h}+(n-2)\frac{f_1'}{f_1}+(n-1)\frac{f_2'}{f_2}\right)-
		\frac{1}{2} \left(\frac{f_1^2}{ h^2 f_2^2} - \frac{h^2}{f_1^2 f_2^2} - \frac{f_2^2}{h^2 f_1^2}\right)+\frac{1-n}{f_1^2}
		=-\lambda,\\[2pt]
		\label{eq:E3}
		&\frac{f_2''}{f_2}
		+\frac{f_2'}{f_2} \left(\frac{h'}{h}+(n-1)\frac{f_1'}{f_1}+(n-2)\frac{f_2'}{f_2}\right)-\frac{1}{2} \left(\frac{f_2^2}{ h^2 f_1^2}-\frac{h^2}{f_1^2 f_2^2}-\frac{f_1^2}{h^2 f_2^2}\right) + \frac{1-n}{f_2^2} 
		=-\lambda,\\[2pt]
		\label{eq:E4}
		&\frac{h''}{h}+(n-1)\frac{f_1''}{f_1}+(n-1)\frac{f_2''}{f_2}=-\lambda,
	\end{align}
subject to the smoothness conditions described in \eqref{eq:smoothproj} and \eqref{eq:smoothquadric}.	
\end{proposition}

\begin{remark}
Observe that the equations are satisfied by $h^{\mathrm{FS}}(t)=\cos(2t)$, $f^{\mathrm{FS}}_1(t)=\cos(t)$ and $f^{\mathrm{FS}}_2(t)=\sin(t)$ with $\lambda=2(n+1)$. This is precisely the Einstein constant for the Fubini-Study metric $g_{\mathrm{FS}}$ on $\C\s P^n$, where $g_{\mathrm{FS}}$ is scaled such that its minimum sectional curvature is equal to~$1$. An alternative check comes from submanifold theory and the study of the geometry of tubes. By~\eqref{eq:eigshape}, the principal curvatures of a tube of radius $t>0$ around the totally geodesic $\R \s P^n$ in ($\C \s{P}^n$, $g_{\mathrm{FS}}$) are: $2\tan(2t)$, $\tan(t)$, and $-\cot(t)$, with multiplicities $1$, $n-1$ and $n-1$, respectively, matching the known literature, see~e.g.~\cite[p.~352]{CR15}.

\end{remark}

\begin{remark}\label{rem:firstintegral}
	The Einstein system \eqref{eq:E1}--\eqref{eq:E4} consists of four equations
	for three unknowns, but it is not overdetermined: the four equations are
	dependent. Indeed, in the combination
	$\eqref{eq:E1}+(n-1)\eqref{eq:E2}+(n-1)\eqref{eq:E3}-\eqref{eq:E4}$ the
	second-derivative terms cancel and, after dividing by $n-1$, we obtain
    \begin{equation}\label{eq:conservation}
	\begin{split}
		2\,\frac{h'f_1'}{hf_1}+2\,\frac{h'f_2'}{hf_2}
		+2(n-1)\frac{f_1'f_2'}{f_1f_2}
		+(n-2)\Big(\frac{f_1'^2}{f_1^2}+\frac{f_2'^2}{f_2^2}\Big)
		&=\frac{1}{h^2}+\frac{n-1}{f_1^2}+\frac{n-1}{f_2^2}\\[2pt]
		&\quad-\frac{h^2}{2f_1^2f_2^2}-\frac{f_1^2}{2h^2f_2^2}
		-\frac{f_2^2}{2h^2f_1^2}-2\lambda ,
	\end{split}
\end{equation}
	which has no second derivative present. Therefore, solving the system \eqref{eq:E1}--\eqref{eq:E4} is equivalent to solving any three of the equations, together with the solution $(f_i, f_i')$ living on the 5-dimensional invariant manifold in phase space fixed by the constraint \eqref{eq:conservation}. This first integral is well known. Writing $\mathrm{scal}(g_t)$ for the scalar curvature of the principal orbit and $H=\tr\mathcal S_t$ for its mean curvature, Eschenburg and Wang obtain it in \cite[Remark~2.2]{EW00} in the form $\mathrm{scal}(g_t)-H^{2}+\tr(\mathcal S_t^{2})=(2n-2)\lambda$. Since it involves only $H^{2}$ and $\tr(\mathcal S_t^{2})$, it is independent of the sign convention for $\mathcal S_t$, and dividing by $n-1$ recovers \eqref{eq:conservation}. In particular, we can check that our system \eqref{eq:E1}--\eqref{eq:E4} agrees with \cite[Proposition~2.1]{EW00}.
\end{remark}

\section{The $H$-reparametrization and the matching problem}
\label{sec:H-rep}
Let us start this section by recalling some well-known facts about the behaviour of the mean curvature $H$ of the principal orbits of a cohomogeneity-one Einstein metric with constant $\lambda>0$.

We follow the notation from \S\ref{subsec:c1theory}. Let $M$ be a simply connected compact cohomogeneity-one manifold equipped with a $\s G$-invariant metric as in \eqref{eq:ansatzcohomo1metric}. Recall that $\mathcal S_t$ denotes the
shape operator of the principal orbit $\Sigma_t=\s G\cdot\gamma(t)$ with
respect to $\dot\gamma(t)$, and define the mean curvature of $\Sigma_t$ as $H(t):=\tr\mathcal S_t$. Consider the smooth function $\mathrm{vol}\colon(0,T)\to\R$, $\mathrm{vol}(t):=\mathrm{Vol}(\Sigma_t)$, where $\mathrm{Vol}$ denotes the volume of the hypersurface $\Sigma_t$. The first
variation of hypersurface volume  yields
\begin{equation*}
	H(t)=-\frac{\mathrm{vol}'(t)}{\mathrm{vol}(t)}=-\frac{d}{dt}\log \mathrm{vol}(t),
\end{equation*}
so $H$ is the opposite of the logarithmic derivative of the orbit
volume.  Notice that if $M$ is simply connected, by~\cite[Proposition 6.39 (ii)]{ABbook}, the principal orbits collapse to a  lower-dimensional singular orbit at the endpoints of the interval $[0,T]$. Thus,  $\mathrm{vol}(t)\to 0$ both as $t\to 0^+$ and as $t\to T^-$, and the rate of vanishing is governed by the number of collapsing dimensions.

Indeed, recall that $\ell_{\pm}:=\dim(\s K_{\pm}/\s H)=\dim\Sigma_t-\dim \s G/\s K_{\pm}$, and thus, $\ell_{\pm}$ is the number of collapsing
dimensions at $t=0$ and $t=T$, respectively. The hypersurface
$\Sigma_t$ is a  tube of radius $t$ around the singular orbit $\s G/\s K_-$. By \cite[Theorem~3.11]{Gr04}, we have
\[
\frac{\vartheta_u'(t)}{\vartheta_u(t)}
= -\Big(\frac{\ell_-}{t}+H(t)\Big),
\]
along the normal geodesic $\gamma$, where $\vartheta_u$ is the normal volume density at $\gamma(0)=p_-$
in direction $u=\dot{\gamma}(0)$. As $\vartheta_u(0)=1$ (see \cite[Lemma~3.7]{Gr04}), then $\frac{\vartheta_u'(t)}{\vartheta_u(t)}$ is bounded as $t\to 0^+$, and hence
\[
H(t)=-\frac{\ell_-}{t}+O(1)
\quad\text{as }t\to 0^+.
\]
Analogously, when $t\to T^-$, we have $H(t)=\frac{\ell_+}{T-t}+O(1)$.
Now let us assume that $g$ is an Einstein metric with Einstein constant equal to $\lambda>0$. Notice that
\[T=d(\gamma(0),\gamma(T))\leq \mathrm{diam}(M)\leq \pi\sqrt{\frac{\dim(M)-1}{\lambda}},\]
where we have used Bonnet--Myers theorem. Notice that if $g$ is not the round metric on the sphere, by Cheng's diameter rigidity theorem we have the strict inequality $T<\pi\sqrt{\frac{\dim(M)-1}{\lambda}}$.
Finally, by taking traces on Riccati's equation (see~\eqref{eq:riccati}), we have
\[H'(t)=\lambda + \tr(\mathcal{S}^2_t)\ge \lambda + \frac{H^2(t)}{\dim(M) -1}>0, \]
where we have used Cauchy--Schwarz inequality to bound $\tr(\mathcal{S}^2_t)$.
Notice that this implies that $H$ has a unique zero in $t\in(0,T)$, which corresponds to the unique minimal principal orbit of the $\s G$-action on $M$. Consequently, by the discussion above we have:
\begin{lemma}\label{lem:focusing}
	Along any solution of \eqref{eq:E1}--\eqref{eq:E4},
	\begin{equation}\label{eq:Qdef}
		H'(t)=Q(t),
		\qquad\text{where}\quad
		Q(t):=\lambda+\tr(\mathcal S_t^2)\;\ge\;\lambda+\frac{H^2(t)}{2n-1}.
	\end{equation}
	Moreover, assuming the smoothness conditions \eqref{eq:smoothproj} and
	\eqref{eq:smoothquadric}, the mapping $t\in(0,T)\mapsto H(t)$ is an increasing
	diffeomorphism from $(0,T)$ onto $\R$, and the following is satisfied:
	\begin{enumerate}
		\item[\textup{(i)}] $T<\pi\sqrt{\tfrac{2n-1}{\lambda}}$,
		\item[\textup{(ii)}] $H$  blows up at $t=0$ and $t=T$ as follows:
		\[
		H(t)=\frac{1-n}{t}+O(1)
		\quad\text{as }t\to 0^+, \quad H(t)=\frac{1}{T-t}+O(1)
		\quad\text{as }t\to T^-.
		\]
		\item[\textup{(iii)}] $H$ has a unique zero.
	\end{enumerate}
\end{lemma}

\subsection{The $H$-reparametrization of the system}\label{subsec:rewriting}
The goal of this subsection is to rewrite the system \eqref{eq:E1}--\eqref{eq:E4} as a first-order system in scale-invariant variables. The system \eqref{eq:E1}--\eqref{eq:E4} is an autonomous six-dimensional system (three ODEs of second order). Restricting to the five-dimensional manifold of $(h, f_1, f_2, h', f_1', f_2')$ that satisfies \eqref{eq:conservation}, we obtain a five-dimensional system. In this subsection we will replace the $t \in (0, T)$ coordinate by $H \in (-\infty, \infty)$, which we can do thanks to the monotonicity of Lemma \ref{lem:focusing}. With that reduction we will obtain a four-dimensional \textit{non-autonomous} system, which we state in scale-invariant variables. We will derive the new ODE system \eqref{eq:abUV} and prove that it is equivalent to \eqref{eq:E1}--\eqref{eq:E4} (Proposition \ref{prop:equivalence}). The corresponding smoothness conditions in the $H$-formulation will be discussed in the next subsection.
We now define the scale-invariant variables:
\begin{equation}\label{eq:uvdef}
	a:=\frac{f_1^2}{h^2},
	\qquad
	b:=\frac{f_2^2}{h^2},
	\qquad
	u:=\Big(\log\frac{f_1}{h}\Big)'=\frac{f_1'}{f_1}-\frac{h'}{h},
	\qquad
	v:=\Big(\log\frac{f_2}{h}\Big)'=\frac{f_2'}{f_2}-\frac{h'}{h}.
\end{equation}
\begin{remark}
\label{rem:variablesgeom}
The variables $(a,b)$ are scale-invariant under the simultaneous rescaling given by
	$(h(t),f_1(t),f_2(t))\mapsto(\kappa\, h(t/\kappa), \kappa\, f_1(t/\kappa), \kappa\, f_2(t/\kappa))$, which leaves the Einstein
	system \eqref{eq:E1}--\eqref{eq:E4} invariant up to $\lambda\mapsto\kappa^{-2}\lambda$. 	We note that the use of logarithmic ratios as a change of variables in this cohomogeneity-one setting goes back at least to B\"ohm~\cite[\S4]{B98} in the case that $\mathcal{P}_t$ defined in \eqref{eq:Pt} involves only two functions,	and since then it has been employed in other works such as~\cite{Br26}.

   Moreover, we can give a geometric meaning to the variables $(u,v)$. By~\eqref{eq:eigshape}, the condition $u(t_0)=v(t_0)=0$ is
	equivalent to $\frac{h'}{h}(t_0)=\frac{f_1'}{f_1}(t_0)=\frac{f_2'}{f_2}(t_0)$,
	which means that the shape operator $\mathcal{S}_{t_0}$ is a multiple of the
	identity, that is, the principal orbit $\Sigma_{t_0}$ is totally umbilical
	in $(M,g)$. If moreover $u\equiv v\equiv 0$ on $(0,T)$,
	then $f_1=c_1 h$ and $f_2=c_2 h$ on $(0,T)$ for some constants $c_1,c_2>0$, and
	the metric \eqref{eq:ansatz} takes the warped-product form
	$g|_{(0,T)\times \s G/\s H}=dt^2+h(t)^2\, g_{\s G/\s H}$, where $g_{\s G/\s H}$
	is the homogeneous metric on the principal orbit induced by $\mathcal{Q}|_{\g{n}_0}
	+c_1^2 \mathcal{Q}|_{\g{n}_1}+c_2^2 \mathcal{Q}|_{\g{n}_2}$.
\end{remark}
Recall that $\ric_i(g_t)=\Ric_{g_t}(\hat X^*_i,\hat X^*_i)$, with $\hat X_i\in\g{n}_i$ a $g_t$-unit vector. It can be extracted from~\eqref{eq:E1}, \eqref{eq:E2}, \eqref{eq:E3} by taking the summands that do not contain derivatives. In terms of the variables $a,b$ defined in \eqref{eq:uvdef} one obtains
\begin{align*}\begin{split}\label{eq:ricformulas}
		\ric_0(g_t)&=\frac{(n-1)(1-(a-b)^2)}{2abh^2},\qquad
		\ric_1(g_t)=\frac{a^2 - (b+1)^2+2bn}{2abh^2},\\
		\ric_2(g_t)&=\frac{b^2-1-a(2+a-2n)}{2abh^2}.
\end{split}\end{align*}
In particular, it is convenient to write the difference of the Ricci eigenvalues $\ric_i(g_t)-\ric_0(g_t)$, with $i\in\{1,2\}$, and the scalar
curvature
$\mathrm{scal}(g_t)=\ric_0(g_t)+(n-1)\big(\ric_1(g_t)+\ric_2(g_t)\big)$
as
\begin{equation}\label{eq:ricdiffs}
	\begin{aligned}
	\ric_1(g_t)-\ric_0(g_t)&=\frac{N_1}{2abh^2}, &
\ric_2(g_t)-\ric_0(g_t)&=\frac{N_2}{2abh^2},\\ \ric_2(g_t)-\ric_1(g_t)&=\frac{N_3}{a b h^2},&
\mathrm{scal}(g_t)&=\frac{(n-1)D}{2abh^2},
	\end{aligned}
\end{equation}
where
\begin{equation}
\label{eq:abuvquantities}
\begin{aligned}
N_1&:=(n-2) b^2 + na^2 - 2(n-1) ab  + 2(n-1) b - n, \\
N_2&:= (n-2) a^2 + n b^2 - 2 (n-1) a b + 2(n-1)a - n, \\
N_3&:=	(b-a) (a+b-n+1)\\
D&:= 2(n-1)(a+b) + 2ab -a^2 - b^2 - 1.
\end{aligned}
\end{equation}
In~\eqref{eq:logform}, if we subtract the equation for $i=0$ from those for $i=1$ and $i=2$ we obtain
\begin{equation*}
	u'-Hu=\ric_1(g_t)-\ric_0(g_t),
	\qquad
	v'-Hv=\ric_2(g_t)-\ric_0(g_t).
\end{equation*}
Now, since $(\log a)'=2\big(\log(f_1/h)\big)'=2u$ and likewise $(\log b)'=2v$, by using~\eqref{eq:ricdiffs} and~\eqref{eq:abuvquantities} we conclude the system of equations:
\begin{align*}
	a' = 2 a u, \quad
	b' = 2 b v, \quad
	u' = Hu + \frac{N_1}{(n-1)D}\,\mathrm{scal}(g_t), \quad
	v' = Hv + \frac{N_2}{(n-1)D}\,\mathrm{scal}(g_t),
\end{align*}
where we observe that all $N_i$, $D$ and $\mathrm{scal}(g_t)$ just depend on $a, b$, so we regard this as a first-order system in $(a, b, u, v)$.
Since our objective is to use $H$ as a reparametrization of $t$, we also notice that the quantity $Q$ of
Lemma~\ref{lem:focusing} depends only on $u, v$ and $H$. The eigenvalues of $\mathcal S_t$ are $-f_i'/f_i$ with
multiplicities $(1,n-1,n-1)$, and their pairwise differences are exactly
$u$, $v$ and $u-v$; hence, by the Lagrange identity
$\big(\sum_id_i\big)\big(\sum_id_iL_i^2\big)-\big(\sum_id_iL_i\big)^2
=\sum_{i<j}d_id_j(L_i-L_j)^2$ applied to $L_i=-f_i'/f_i$ with weights
$(d_0,d_1,d_2)=(1,n-1,n-1)$, we express $\tr(\mathcal{S}^2_t)$ in terms of $u,v$ and $H$, and thus
\begin{equation}\label{eq:Qcoord}
	Q=\lambda+\tr(\mathcal S_t^2)
	=\lambda+\frac{H^2+(n-1)\big[nu^2+nv^2-2(n-1)uv\big]}{2n-1}.
\end{equation}
We also notice that the quantity $\mathrm{scal}(g_t)$ is explicit in terms of $H$ and $Q$ as by tracing~\eqref{eq:rictan}, 
\begin{equation}\label{eq:constraintH}
	\mathrm{scal}(g_t) = H^2-Q+(2n-1)\lambda .
\end{equation}
All in all, recalling from Lemma~\ref{lem:focusing} that $\tfrac{d}{dH}=\tfrac{1}{Q}\tfrac{d}{dt}$, we arrive at the non-autonomous system:
\begin{equation}\label{eq:abUV}
	\begin{aligned}
		\frac{da}{dH} &= \frac{2 a u}{Q}, &\qquad
		\frac{du}{dH} &= \frac{(n-1)DHu+\big(H^2-Q+(2n-1)\lambda\big)N_1}{(n-1)DQ},\\[2pt]
		\frac{db}{dH} &= \frac{2 b v}{Q}, &\qquad
		\frac{dv}{dH} &= \frac{(n-1)DHv+\big(H^2-Q+(2n-1)\lambda\big)N_2}{(n-1)DQ},
	\end{aligned}
\end{equation}
where the quantities $N_1$, $N_2$, $D$, and $Q$ are given by~\eqref{eq:abuvquantities} and \eqref{eq:Qcoord}, and only depend on $a,b,u,v$ and $H$. In conclusion, the derivation above shows how solutions of \eqref{eq:E1}--\eqref{eq:E4} give rise to solutions of \eqref{eq:abUV}. The following proposition provides the converse, which is the direction needed for the matching argument.
\begin{proposition}\label{prop:equivalence}
	Let $(a,b,u,v)\colon\R\to\R^4$ be a solution of \eqref{eq:abUV} (with $D, N_i, Q$ defined in \eqref{eq:abuvquantities} and \eqref{eq:Qcoord}) satisfying $	a,b>0$ and $\tfrac{H^2-Q+(2n-1)\lambda}{D}>0$. Then $T:=\int_\R dH/Q$ is finite, $t:=\int_{-\infty}^{H}d\varsigma/Q(\varsigma)$ defines
	an increasing diffeomorphism $\R\to(0,T)$, and the functions of $t$
	\begin{equation}
	\label{eq:hf1f2}
	h:=\Big(\frac{(n-1)D}{2ab\big(H^2-Q+(2n-1)\lambda\big)}\Big)^{1/2},
	\qquad
	f_1:=h\sqrt a,\qquad f_2:=h\sqrt b
	\end{equation}
	give a positive smooth solution of \eqref{eq:E1}--\eqref{eq:E4} on $(0,T)$
	with $H\to-\infty$ as $t\to0^+$ and $H\to+\infty$ as $t\to T^-$, whose
	variables \eqref{eq:uvdef} and mean curvature are the given
	$(a,b,u,v)$ and $H$.
\end{proposition}
\begin{proof}
The idea is to invert the substitution \eqref{eq:hf1f2} and show that the four Einstein equations \eqref{eq:E1}--\eqref{eq:E4} reduce to \eqref{eq:E4}, the two Ricci-difference equations $\eqref{eq:E2}-\eqref{eq:E1}$ and $\eqref{eq:E3}-\eqref{eq:E1}$, and the conservation law \eqref{eq:conservation}. By \eqref{eq:Qcoord} we have $Q\ge\lambda+H^2/(2n-1)>0$, so it is clear that
	$\int_{\mathbb{R}} dH / Q < + \infty$ and $H \mapsto t(H) := \int_{-\infty}^H \frac{d\varsigma}{Q(\varsigma)}$ is an increasing diffeomorphism $\mathbb R  \to (0, T)$. Now, changing variables and writing $a,b,u,v$ as functions of $t$, \eqref{eq:abUV} yields the system:
	\begin{equation}
    \begin{aligned}
    \label{eq:abuvt}
		a' =2au, \qquad u' - H u = \frac{(H^2  - Q + (2n-1)\lambda )N_1}{(n-1)D}  \\
		b' = 2 b v, \qquad v' - H v = \frac{(H^2  - Q + (2n-1)\lambda ) N_2}{(n-1)D} \qquad H' = Q.
    \end{aligned}
	\end{equation}
	 Now, we proceed to take a derivative on $\frac{(n-1)D}{2ab(H^2-Q + (2n-1)\lambda)}$. First notice that, using \eqref{eq:constraintH} and substituting $(2n-1)\lambda$ in the equation below using \eqref{eq:Qcoord}, we get
\begin{align}
\mathrm{scal}(g_t)&= (n-1)(2Q - \underbrace{ (nu^2 + nv^2 - 2(n-1)uv) }_{=:W} ) = (n-1) (2Q - W) \label{eq:scalgt} \\
W' - 2HW& =\frac{2\mathrm{scal}(g_t)}{(n-1)D}\left(n(u-v)(N_1-N_2)+uN_2+vN_1\right), \label{eq:Wprime}
\end{align}
where we have used \eqref{eq:abuvt} to derive  \eqref{eq:Wprime}.
Taking a derivative with respect to $t$ on \eqref{eq:scalgt} and adding $2(n-1)\frac{d}{dt}\mathrm{scal}(g_t)= -(2n-2)Q'+4(n-1)HQ$ (from \eqref{eq:constraintH}), the terms with $Q'$ cancel, and we obtain that
\begin{align*}
(2n-1)\frac{d}{dt}\mathrm{scal}(g_t) &= 4(n-1) H Q - (n-1) W' \\
&= (n-1)[4HQ - 2HW - 2\frac{\mathrm{scal}(g_t)}{(n-1)D}\bigl(n(u-v)(N_1-N_2)+uN_2+vN_1\bigr)]\\
&= 2H\mathrm{scal}(g_t) - 2\frac{\mathrm{scal}(g_t)}{D}\bigl(n(u-v)(N_1-N_2)+uN_2+vN_1\bigr)
\end{align*}
where in the second line we used \eqref{eq:Wprime} and in the third one we used again \eqref{eq:scalgt}. Now, taking the derivative of $\log( \frac{(n-1)D}{2ab \mathrm{scal}(g_t)} )$, we get
\begin{align} \begin{split} \label{eq:preH}
 \frac{d}{dt}&\log\Big(\frac{(n-1)D}{2ab\,\mathrm{scal}(g_t)}\Big) = \frac{D'}{D} - \frac{a'}{a} - \frac{b'}{b} - \frac{\tfrac{d}{dt}\mathrm{scal}(g_t)}{\mathrm{scal}(g_t)} \\
&= \frac{(2n-1)D' + 2nuN_1 + 2nvN_2 - 2(n-1)uN_2 - 2(n-1)vN_1}{(2n-1)D} - 2u -2v - \frac{2}{2n-1} H \\
&= \frac{2n(u+v)D}{(2n-1)D} - 2u - 2v - \frac{2}{2n-1}H = - \frac{(2n-2)}{(2n-1)}(u+v) - \frac{2}{2n-1}H.
\end{split} \end{align}
The computation from the second to the third line is a heavy but straightforward computation that one obtains from $D' = 4(n-1)(au+bv) + 4abv+4abu-4a^2u-4b^2v$ and expanding $N_1, N_2$ according to \eqref{eq:abuvquantities}.
By hypothesis, $\frac{ H^2 - Q + (2n-1)\lambda }{D} > 0$, and this allows us to define $h(t) > 0$ via \eqref{eq:hf1f2}, and therefore we see that \eqref{eq:preH} is simply $\frac{2h'}{h}$. Now, we also define $f_i$ according to \eqref{eq:hf1f2}. Using our ODE for $a'$, we have that $f_1' = h'\sqrt{a} + \frac12 \frac{h a'}{\sqrt{a}} = h'\sqrt{a} + h u \sqrt{a} = \frac{h' f_1}{h} + u f_1. $
Similarly, we obtain that $f_2' = \tfrac{h' f_2}{h} + v f_2$. Therefore, we have that
\begin{align} \begin{split} \label{eq:tuvalu}
	u = \frac{f_1'}{f_1} - \frac{h'}{h}, \quad v = \frac{f_2'}{f_2} - \frac{h'}{h}, \quad a = \frac{f_1^2}{h^2}, \quad b = \frac{f_2^2}{h^2}, \quad H = - \frac{h'}{h} - (n-1)\left( \frac{f_1'}{f_1} + \frac{f_2'}{f_2} \right).
\end{split} \end{align}
where the last formula for $H$ comes from \eqref{eq:preH} after substituting $u$ and $v$ by their formulas.  Notice that this coincides with the definition of our variables in \eqref{eq:uvdef}. Moreover, we get
\begin{align} \begin{split}\label{eq:uvH_prime}
u' &= \frac{f_1''}{f_1}- \frac{h''}{h} - \frac{(f_1')^2}{f_1^2} + \frac{(h')^2}{h^2}, \qquad v' = \frac{f_2''}{f_2}- \frac{h''}{h} - \frac{(f_2')^2}{f_2^2} + \frac{(h')^2}{h^2} \\
H' &= -\frac{h''}{h}-(n-1)\frac{f_1''}{f_1}-(n-1)\frac{f_2''}{f_2}
      +\frac{(h')^2}{h^2}+(n-1)\frac{(f_1')^2}{f_1^2}+(n-1)\frac{(f_2')^2}{f_2^2}
\end{split}\end{align}
Now, we take \eqref{eq:abuvt} and using our definition of $h$, we reexpress the last three equations as
\begin{equation*} \label{eq:abuvt2}
u' - Hu = \frac{N_1}{2abh^2}, \qquad v' - Hv = \frac{N_2}{2abh^2}, \qquad H' = \lambda + \frac{H^2 + (n-1)(nu^2 + nv^2 - 2(n-1)uv)}{2n-1}
\end{equation*}
Finally, using the last three equations together with $2ab h^2(H^2 - Q + (2n-1)\lambda) = (n-1)D$ (from $h$ definition) and substitute $a, b, u, v, u', v', H, H'$ by their corresponding expressions given in \eqref{eq:tuvalu}--\eqref{eq:uvH_prime}. We get
\begin{align*}\label{eq:abuvt3}
\frac{f_1''}{f_1}&-\frac{h''}{h}-\frac{(f_1')^2}{f_1^2}+\frac{(h')^2}{h^2}
  -H\Big(\frac{f_1'}{f_1}-\frac{h'}{h}\Big)
=\frac{n}{2}\Big(\frac{f_1^2}{h^2f_2^2}-\frac{h^2}{f_1^2f_2^2}\Big)
  +\frac{n-2}{2}\frac{f_2^2}{h^2f_1^2}
  +(n-1)\Big(\frac{1}{f_1^2}-\frac{1}{h^2}\Big),\\
\frac{f_2''}{f_2}&-\frac{h''}{h}-\frac{(f_2')^2}{f_2^2}+\frac{(h')^2}{h^2}
  -H\Big(\frac{f_2'}{f_2}-\frac{h'}{h}\Big)
=\frac{n}{2}\Big(\frac{f_2^2}{h^2f_1^2}-\frac{h^2}{f_1^2f_2^2}\Big)
  +\frac{n-2}{2}\frac{f_1^2}{h^2f_2^2}
  +(n-1)\Big(\frac{1}{f_2^2}-\frac{1}{h^2}\Big),\\
\frac{h''}{h}&+(n-1)\frac{f_1''}{f_1}+(n-1)\frac{f_2''}{f_2}=-\lambda, \\
2\frac{h'}{h}&\Big(\frac{f_1'}{f_1}+\frac{f_2'}{f_2}\Big)
  +2(n-1)\frac{f_1'f_2'}{f_1f_2}
  +(n-2)\Big(\frac{(f_1')^2}{f_1^2}+\frac{(f_2')^2}{f_2^2}\Big) \\
 &\qquad =\frac{1}{h^2}+(n-1)\Big(\frac{1}{f_1^2}+\frac{1}{f_2^2}\Big)
   -\frac12\Big(\frac{h^2}{f_1^2f_2^2}+\frac{f_1^2}{h^2f_2^2}+\frac{f_2^2}{h^2f_1^2}\Big)-2\lambda .
\end{align*}
The third equation is exactly \eqref{eq:E4}, the last equation is \eqref{eq:conservation}, the first is $\eqref{eq:E2}-\eqref{eq:E1}$, and the second is $\eqref{eq:E3}-\eqref{eq:E1}$. The four combinations \eqref{eq:E4}, $\eqref{eq:E2}-\eqref{eq:E1}$, $\eqref{eq:E3}-\eqref{eq:E1}$ and \eqref{eq:conservation} are linearly independent, hence equivalent to \eqref{eq:E1}--\eqref{eq:E4} (Remark~\ref{rem:firstintegral}). 
\end{proof}
\begin{remark}
\label{rem:wpeinsteinsol}
Notice that the solutions to the algebraic equation $(N_1,N_2,N_3)=(0,0,0)$, where each $N_i$ with $i\in\{1,2,3\}$ is defined in~\eqref{eq:abuvquantities}, provide exactly the Einstein $\s{SO}_{n+1}$-invariant metrics on a principal orbit $\Sigma_t=\s G/\s H=T_1\R\s{P}^n$. By $N_3=0$, we either have $a+b=n-1$ or $a=b$. In the first case, we can deduce that there are no valid solutions (one obtains
$N_1=N_2=4(n-1)\big(a-\tfrac{n-1}{2}\big)^2+(n-1)^2-n$, which is always positive for $n\ge3$ and only zero for $n = 2$ and $(a, b) \in \{ (1, 0), (0, 1) \}$). However, in the second case $a = b$ there is a unique solution for each $n\ge 2$. Indeed, we deduce from $N_1=N_2=0$ that $a=b=\tfrac{n}{2(n-1)}$. The Einstein  $\s{SO}_{n+1}$-invariant metrics on $ T_1\s{S}^n$, which is the universal cover of $\Sigma_t=T_1\R\s{P}^n$ for $n\ge3$, were classified in \cite[\S4]{Ke98}, and the metric found here coincides with the one obtained there.
Now it is clear that $(a,b,u,v)=\tfrac{n}{2(n-1)}(1,1,0,0)$ is a fixed point of the system given by  \eqref{eq:abUV}. Thus, by Proposition~\ref{prop:equivalence}, it provides a smooth Einstein metric on $(0,T)\times\Sigma_t$. This metric is a warped product metric by Remark~\ref{rem:variablesgeom}, and it does not extend smoothly to $\C\s P^n$. Indeed, by \eqref{eq:uvdef}, we have $1=\tfrac{a}{b}=\tfrac{f_1^2}{f^2_2}$, implying that $f_1=f_2$. However, by \eqref{eq:smoothproj}, this cannot happen at $t=0$ when the metric extends smoothly to $t=0$.
\end{remark}
\begin{remark}\label{rem:quadricvars}
	When working close to the complex quadric orbit, we will instead work with
	the variables
	\begin{equation}
    \label{eq:greekvariables}
		\alpha=\frac{a+b}{2},\qquad
		\beta=\frac{a-b}{2},\qquad
		\mu=\frac{u+v}{2},\qquad
		\eta=\frac{u-v}{2}.
	\end{equation}
	In that case, the system reads
	\begin{equation}\label{eq:abUVquadric}
		\begin{aligned}
			\frac{d\alpha}{dH} &= \frac{2(\alpha\mu+\beta\eta)}{Q}, &\qquad
			\frac{d\mu}{dH} &= \frac{(n-1)DH\mu+\big(H^2-Q+(2n-1)\lambda\big)N_+}{(n-1)DQ},\\[2pt]
			\frac{d\beta}{dH} &= \frac{2(\alpha\eta+\beta\mu)}{Q}, &\qquad
			\frac{d\eta}{dH} &= \frac{(n-1)DH\eta+\big(H^2-Q+(2n-1)\lambda\big)N_-}{(n-1)DQ},
		\end{aligned}
	\end{equation}
	where $N_\pm:=\tfrac12(N_1\pm N_2)$, and $D$, $Q$, $N_\pm$ are expressed in terms of $(\alpha,\beta,\mu,\eta)$ as follows:
	\begin{align}
	D&=  4(n-1)\alpha - 4\beta^2 - 1, &
	Q&=\lambda+\frac{H^2+2(n-1)\mu^2+2(n-1)(2n-1)\,\eta^2}{2n-1}, \nonumber\\
	N_+&=4(n-1)\beta^2+2(n-1)\alpha-n, &
	\qquad
	N_-&=2\beta\,(2\alpha-n+1). \label{eq:NDQquadric}
\end{align}
	Notice that Proposition~\ref{prop:equivalence} applies to this system as well, since we
	can first transform it to
	$(a,b,u,v)=(\alpha+\beta,\,\alpha-\beta,\,\mu+\eta,\,\mu-\eta)$ and
	then apply Proposition~\ref{prop:equivalence}.
\end{remark}

\subsection{The smoothness conditions at $H=\pm\infty$}\label{subsec:smoothH}

The objective of this subsection is to express the smoothness conditions
\eqref{eq:smoothproj} and \eqref{eq:smoothquadric} in the variables of
\S\ref{subsec:rewriting}. Since we will construct the solutions of \eqref{eq:abUV} locally around $H = \pm \infty$ via a series expansion, it is convenient to express the necessary conditions for \eqref{eq:smoothproj} and \eqref{eq:smoothquadric} in terms of the $H$-expansion at those singular orbits.

\begin{proposition}\label{prop:smoothH}
Let $(a,b,u,v)$ be a solution of \eqref{eq:abUV} as in
Proposition~\ref{prop:equivalence}, and let $(h,f_1,f_2)$ be the
corresponding solution of \eqref{eq:E1}--\eqref{eq:E4} on $(0,T)$.
\begin{enumerate}
\item[\textup{(i)}] If, for $-H$ sufficiently large, the solution is given by
convergent series
\begin{equation}\label{eq:seriesRP}
a=\sum_{\substack{k\ge0\\ k\,\mathrm{even}}} a_kH^{-k},\qquad
b=\sum_{\substack{k\ge2\\ k\,\mathrm{even}}} b_kH^{-k},\qquad
u=\sum_{\substack{k\ge1\\ k\,\mathrm{odd}}} u_kH^{-k},\qquad
v=\sum_{\substack{k\ge-1\\ k\,\mathrm{odd}}} v_kH^{-k},
\end{equation}
with $a_0=1$,  $v_{-1}=\tfrac{-1}{n-1}$, and $b_2>0$, then $(h,f_1,f_2)$
satisfies \eqref{eq:smoothproj}.
\item[\textup{(ii)}] If, for $H$ sufficiently large, the functions
$(\alpha,\beta,\mu,\eta)$ of Remark~\ref{rem:quadricvars} are given by convergent series
		\begin{align}\label{eq:seriesquadric}
		\begin{split}
		\alpha=\sum_{\substack{k\ge -2\\ k\,\mathrm{even}}} \alpha_k H^{-k},\qquad
		\beta=\sum_{\substack{k\ge-1\\ k\, \mathrm{odd}}} \beta_k H^{-k},\qquad
		\mu=\sum_{\substack{k\ge-1\\ k\,\mathrm{odd}}} \mu_k H^{-k},\qquad
		\eta=\sum_{\substack{k\ge 0\\ k\,\mathrm{even}}} \eta_k H^{-k},
		\end{split}
		\end{align}
		with $\mu_{-1}=1$, $\alpha_{-2}>0$, and $(n-1)\alpha_{-2}\neq\beta_{-1}^2$, then $(h,f_1,f_2)$ satisfies
		\eqref{eq:smoothquadric}.
	\end{enumerate}
\end{proposition}

\begin{proof}
We start by setting $\rho = -1/H$, so that \eqref{eq:seriesRP} can be written as 
\begin{equation}
\label{eq:2seriesrpn}
        a=1+\rho^2 \bar a(\rho^2),\qquad
        b=\rho^2 \bar b(\rho^2),\qquad
        u=\rho \bar u(\rho^2),\qquad
        v=\frac{1}{(n-1)\rho}+\rho \bar v(\rho^2),
\end{equation}
where $\bar a,\bar b,\bar u,\bar v$ are analytic in a neighbourhood of $0$, and $\bar b(0)=b_2>0$. 
Similarly, from the definition of $Q$ (see~\eqref{eq:Qcoord}), we get
\begin{equation*}
        Q(-1/\rho)=\rho^{-2}\bar Q(\rho^2),
        \qquad \bar Q (0)=\frac1{n-1}.
\end{equation*}
In particular $\bar Q$ is analytic and nonzero for $\rho$ small. Therefore,
\begin{equation*}
        t(\rho) =\int_{-\infty}^{-1/\rho}\frac{dH}{Q(H)}
        =\int_0^\rho \frac{d\sigma}{\bar Q(\sigma^2)}
        =\rho \bar t(\rho^2),
\end{equation*}
where $\bar t$ is analytic and $\bar t(0)=n-1$. Hence, by the analytic
inverse function theorem, we can invert as $\rho = t \bar \rho(t^2)$ for some analytic function $\bar \rho$ near $0$, with $\bar \rho (0) = 1/ (n-1)$. Now, composing this with \eqref{eq:2seriesrpn}, we obtain that $a, b$ are even analytic at $t = 0$, while $u, v - \frac{1}{t}, H + \frac{n-1}{t}$ are odd analytic at $t = 0$, with 
\begin{align*}
	a(t)&=1+t^2 \widehat a(t^2),& 
	b(t)&=\frac{b_2}{(n-1)^2}t^2+t^4 \widehat b(t^2),\\
	u(t)&=t\widehat u(t^2),&
	v(t)&=\frac{1}{t}+t\widehat v(t^2),&
	H(t)&=-\frac{n-1}{t}+t\widehat H(t^2),
\end{align*}
for some analytic functions $\widehat a,\widehat b,\widehat u,\widehat v,\widehat H$.
Using formulas \eqref{eq:hf1f2}, we obtain that $h(t)^2, f_1(t)^2, f_2(t)^2$ are all even analytic, with $h(0) = f_1(0) = \frac{n-1}{\sqrt{b_2}} > 0$ and $f_2(0) = 0$. Also, taking two derivatives on \eqref{eq:hf1f2}, we have $2 f_2'(0)^2 = h(0)^2 \frac{2 b_2}{(n-1)^2}$, so that $f_2'(0) = 1$. In particular, we immediately obtain all conditions from \eqref{eq:smoothproj} with $\eta, \psi, \phi$ analytic in a neighbourhood of zero, by taking $\phi(t^2) = \tfrac{1}{2}\big(h(t)^2 + f_1(t)^2\big)$, with $\psi(t^2) = \tfrac{1}{2t^2}\big(h(t)^2 - f_1(t)^2\big)$ and $\eta(t^2) = \frac{f_2(t)^2-t^2}{t^4}$. This completes the proof of \textup{(i)}.

The proof of \textup{(ii)} is completely analogous to the proof of \textup{(i)}, so we just describe the main steps. In this case, we set $\rho = 1/H$. Following steps analogous to those in \textup{(i)}, we have that
\begin{align}
\alpha &= \alpha_{-2} \rho^{-2} + \bar \alpha (\rho^2), \quad \beta = \beta_{-1} \rho^{-1} + \rho \bar \beta (\rho^2), \quad \mu = \mu_{-1} \rho^{-1} + \rho \bar \mu (\rho^2), \quad \eta = \bar \eta (\rho^2) \label{eq:guanajuato} \\[1ex]
Q(1/\rho) &= \rho^{-2}\bar Q(\rho^2), \hspace{0.12cm}\qquad \bar Q(0)=1, \notag \\
T - t(\rho) &= T - \int_{-\infty}^{1/\rho} \frac{dH}{Q(H)} = \int_{1/\rho}^{\infty} \frac{dH}{Q(H)}
= \int_0^\rho \frac{d\sigma}{\bar Q(\sigma^2)}
= \rho \bar t (\rho^2), \notag 
\end{align}
with $\bar \alpha, \bar \beta, \bar \mu, \bar \eta, \bar Q$ analytic around $\rho = 0$. We take $s = T - t$ and use the analytic inverse function theorem to obtain $\rho = s \bar \rho (s^2)$, with $\bar \rho$ analytic around $s = 0$. Since $D$ is a linear combination of $1$, $\alpha$ and $\beta^2$ it admits an even series expansion in $s$. The same holds for $ab = \alpha^2 - \beta^2$. Therefore, from \eqref{eq:hf1f2}, we obtain that $h(T-s)^2$ is even analytic at $s = 0$. Expressing $h(t)^2$ from \eqref{eq:hf1f2} in the variables of Remark \ref{rem:quadricvars}, we obtain $h(t)^2 = \frac{4(n-1)\alpha - 4\beta^2 - 1}{4(\alpha^2-\beta^2)(\lambda - \eta^2 + (H^2-\mu^2)/(2n-1))}$, with leading order term $\frac{(n-1)\alpha_{-2} - \beta_{-1}^2}{\alpha_{-2}^2 (\lambda - \eta_0^2 - 2 \mu_1 / (2n-1))}H^{-2}$. As a consequence $h(T) = 0$. Evaluating \eqref{eq:abUVquadric} also allows us to obtain $\eta_0$ and $\mu_1$ (a precise computation will be done in \eqref{eq:firstcoef-quad}), concluding also that $h'(T)^2 = 4$. Therefore, $h(T-s)^2 = 4s^2 + s^4 \eta (s^2) > 0$ for some analytic $\eta$, which is the first condition in \eqref{eq:smoothquadric}. Now, using \eqref{eq:hf1f2}, 
\begin{equation} \label{eq:helsinki}
f_i(T-s)^2 = h(T-s)^2 (\alpha (T-s) + \epsilon_i \beta (T-s)) = (4 + s^2 \eta(s^2)) (s^2 \alpha (T-s) + \epsilon_i s^2 \beta (T-s)),
\end{equation}
with $\epsilon_i = 1, -1$ for $i = 1, 2$, respectively. Combining \eqref{eq:guanajuato} with $\rho = s \bar \rho (s^2)$ we see that $s^2 \alpha (T-s) = \phi (s^2)$ and $s^2 \beta (T-s) = s \psi (s^2)$, with $\phi, \psi$ analytic in a neighbourhood of zero. Plugging that in~\eqref{eq:helsinki}, we conclude the proof of \textup{(ii)}.\qedhere

\end{proof}

\section{The $H$-expansions at the singular orbits}\label{sec:Hexpansions}

The purpose of this section is to construct a series such as the one in Proposition \ref{prop:smoothH}, where moreover $(a, b, u, v)$ (resp.\ $(\alpha, \beta, \mu, \eta)$) satisfy the ODE \eqref{eq:abUV} (resp.\ \eqref{eq:abUVquadric}). We will see that we can construct such converging series (and therefore a local solution) by fixing just two parameters in each case: $(b_2, u_1)$ when close to $\R\s P^n$ ($H\to -\infty$) and $(\alpha_{-2},\beta_{-1})$ when close to the complex quadric ($H \to +\infty$). This is the content of Lemma~\ref{lem:recRPn} and Lemma~\ref{lem:recQn-1}.

Moreover, when proving convergence of the series it will be useful to obtain explicit bounds on the series coefficients, see Lemma~\ref{lem:inductive_RPn} and Lemma~\ref{lem:inductive_CQ}. In order to do this we will make use of Catalan numbers. The $k$-th Catalan number is defined by $\mathrm{Cat}(k):=\tfrac1{k+1}\binom{2k}{k}$ (see \cite{S15} for more information on Catalan numbers). They will be useful to propagate inductive bounds on the Taylor coefficients of our ODEs (see~\cite[Prop. 2.3]{BCG25} or \cite{CGSS26_2} for the use of such technique in a different context). That will allow us to prove absolute convergence on an explicit interval $(-\infty, H_-]$ (resp.~$[H_+,+\infty)$), with an explicit tail bound, as well as analytic dependence on the parameters $(b_2, u_1)$ (resp.~$(\alpha_{-2},\beta_{-1})$). This is the content of Proposition~\ref{prop:RPseries} and Proposition~\ref{prop:QCseries}.

For simplicity and to allow an easier comparison with the Fubini-Study metric, from now on we will normalize our metric taking $\lambda=2(n+1)$. 

\subsection{The $H$-expansion around $\R\s P^n$}\label{subsec:expRP}

By Proposition~\ref{prop:smoothH}\,(i), we look for expansions of the type \eqref{eq:seriesRP}. We derive explicit recurrences for $a_k, b_k, u_k, v_k$, depending just on the two free parameters $(b_2, u_1)$. Clearing the denominators in \eqref{eq:abUV} and setting $\lambda=2(n+1)$, the system becomes the set of polynomial identities
\begin{equation}\label{eq:rp-ode}
	\begin{aligned}
		&Q\frac{da}{dH}-2au=0,&(n-1)DQ\,\frac{du}{dH}-(n-1)DHu-N_1\left(H^2-Q+2(2n- 1)(n+1)\right)=0, \\
		&Q\frac{db}{dH}-2bv=0,
		&(n- 1)DQ\,\frac{dv}{dH}-(n- 1)DHv-N_2\left(H^2-Q+2(2n- 1)(n+1)\right)=0,
	\end{aligned}
\end{equation}
with $N_1$, $N_2$, $D$ as in \eqref{eq:abuvquantities} and $Q$ as in \eqref{eq:Qcoord}. First, we give the formulas for the series coefficients of $N_1$, $N_2$, $D$ and $Q$, writing $(\cdot)_k$ for the coefficient of $H^{-k}$ in each of them. We start by noting that, by the parities in \eqref{eq:seriesRP}, all of them are series in $H^{-k}$ with $k \geq 0$ even, except for $Q$, in which case $k$ is also even but a $k = -2$ term appears. For $k \geq 2$, we have
\begin{align} \begin{split} \label{eq:NDQk}
    Q_k &=\frac{n-1}{2n-1}\Big[2 v_{-1}\left(n v_{k+1} - (n-1)u_{k+1}\right)+  \sum_{\substack{i+j=k\\ 1\le i\le k-1\\ 1\le j\le k-1}} n u_iu_j-2(n-1)u_i v_j + n v_i v_j \Big], \\
    D_k &= 2a_k(-a_0+(n-1))+ 2b_k(a_0+(n-1))+  \sum_{\substack{i+j=k\\ 2\le i\le k-2\\ 2\le j\le k-2}}  2a_ib_j -  a_ia_j -   b_ib_j, \\
    N_{1,k}&= 2na_k  + \sum_{\substack{i+j=k\\ 2\le i\le k-2\\ 2\le j\le k-2}}  n a_ia_j + (n- 2) b_ib_j - 2(n-1) a_i b_j,\qquad N_{1,k}^{\rm des} :=N_{1,k}-2n a_k, \\
    N_{2,k} &=2 a_k \left( (n-2)a_0 + n-1\right)  - 2(n-1)b_k + \sum_{\substack{i+j=k\\ 2\le i\le k-2\\ 2\le j\le k-2}} (n-2) a_i a_j + n b_i b_j -2(n-1)a_i b_j 
\end{split} \end{align}
where the definition of $N_{1, k}^{\rm des}$ is simply $N_{1,k}$ with its top-order term $a_k$ removed. Notice that we have moved out of the convolution-type sums the terms with $a_0$ and $v_{-1}$; this will be convenient later.  The coefficients with $k = 0$ (and $Q_{-2}$) are
\begin{equation}
\label{eq:first_coefs_NDQ}
\begin{aligned} 
	Q_{-2}&=\tfrac1{n-1},&  Q_0&=2(n+1)-\tfrac{2nv_1-2(n-1)u_1}{2n-1}, & D_0&=2(n-2),\\
    N_{1,0}&=0,&  N_{2,0}&=2(n-2).
\end{aligned}
\end{equation}
By~\eqref{eq:NDQk} and \eqref{eq:first_coefs_NDQ} we can compute $N_{1,k},N_{2,k}$ and $D_k$ knowing the coefficients $a_j,b_j$ for $j\le k$, and $Q_k$ knowing $u_j,v_j$ for $j\le k+1$. In order to compute $N_{1, k}^{\rm des}$ we only need $a_j, b_j$ for $j \leq k-2$.

Now recall from Proposition~\ref{prop:smoothH}~(i) that we impose $v_{-1} = -\frac{1}{n-1}$, $a_0 = 1$ to have smoothness at the singular orbit $\R\s P^n$. The coefficients $b_2, u_1$ are our free inputs determining the series. The coefficients $a_2$ and $v_1$ can be obtained by looking at the $H^{-1}$ term in the $\frac{da}{dH}$ equation in \eqref{eq:rp-ode} and at the $H^{0}$ term in the $\frac{dv}{dH}$ equation in \eqref{eq:rp-ode}. We obtain:
\begin{equation}\label{eq:a2u1}
a_2 = -\frac{a_0 u_1}{Q_{-2}}  = -(n-1) u_1, \qquad v_1 = \frac{2n-1}{n(n-1)}\,b_2+\frac{n-1}{n}\,u_1-\frac{2(n+1)(2n-1)}{n}.
\end{equation}

For the higher order coefficients, we define our recurrence iteratively. Assume we know the coefficients $a_j,b_j,u_j,v_j$ for $j<k$, with $k\geq 3$ odd, and we show how to compute $a_{k+1},b_{k+1},u_k,v_k$. In particular we then know $N_{i,j},D_j$ for all $j\le k-1$ and $Q_j$ for all $j\le k-3$. Taking the coefficient of $H^{-k}$ in the $db/dH$-equation in \eqref{eq:rp-ode} isolates $b_{k+1}$,
\begin{equation}\label{eq:rec-b}
	\tfrac{k-1}{n-1}\,b_{k+1}= R_{b, k}:= -2\sum_{\substack{i+j=k\\ 2\le i\le k-1\\ 1\le j\le k-2}} b_iv_j- \sum_{\substack{i+j=k-1\\ 0\le i\le k-3\\ 2\le j\le k-1}} j\,Q_i\,b_j.
\end{equation}
We notice that this completely determines $b_{k+1}$ in terms of known quantities in the right-hand side. We are also using $k \geq 3$ so that $k-1 > 0$ (in contrast, the case $k = 1$ gives a vacuous equation for $b_2$, and that is why $b_2$ is a free parameter).  Next, the $da/dH$-equation at order $k$ and the $du/dH$-identity at order $k-1$ give a $2\times2$ system for $(a_{k+1},u_k)$. Namely, we have the linear system given by $M_k (a_{k+1}\enspace u_k)^\top =  (R_{a, k}\enspace R_{u, k})^\top$, where:
\begin{align}
	M_k&=\begin{pmatrix}-\dfrac{k+1}{n-1} & -2\\[6pt] -\dfrac{2n(n-2)}{n-1} & -2(n-2)(k+n- 1)\end{pmatrix},
	\enspace
	M_k^{-1}=\frac{\begin{pmatrix}-(n-1)(k+n-1) & \frac{n-1}{n-2}\\[4pt] n & -\frac{k+1}{2(n-2)}\end{pmatrix}}{(k-1)(k+n+1)},\nonumber\\[2ex]
 R_{a, k}&:= 2 \sum_{\substack{i+j=k\\ 2\le i\le k-1\\ 1\le j\le k-2}} a_iu_j+  \sum_{\substack{i+j=k-1\\ 0\le i\le k-3\\ 2\le j\le k-1}}  j Q_i a_j, \nonumber\\[2ex]
    \frac{R_{u, k}}{n-1} &:=  \tfrac{n-2}{(n-1)^2}N_{1,k+1}^{\rm des} + \frac{2(2n- 1)(n+1)}{n-1} N_{1,k-1}-\frac{1}{n-1}  \sum_{\substack{i+j=k-1\\ 0\le i\le k-3 \\ 
        2 \leq j \leq k-1}}  Q_i\,N_{1,j}\label{eq:rec-aU}\\ &+ Q_{-2}  \sum_{\substack{i+j=k\\ 1\le j\le k-2 \\ 2 \leq i \leq k-1}}  j D_i u_{j}  + \sum_{\substack{i+j+l=k-2\\ 2\le i\le k-3, 0\le l\le k-3\\ 1\le j\le k-2}}  j  D_i\,Q_l u_{j} + D_0  \sum_{\substack{j+l=k-2\\ 0\le l\le k-3 \\ 1\le j\le k-2}}  j  Q_l u_{j}
        + \sum_{\substack{i+j=k\\ 2\le i\le k-1 \\ 1 \leq j \leq k-2}}  D_i u_{j}\nonumber. 
\end{align}
In particular, we see that for $k \geq 3$ and $n \geq 3$, the system \eqref{eq:rec-aU} allows us to solve for $a_{k+1}$ and $u_k$. This also allows us to compute $N_{1,k+1},N_{2,k+1}$ and $D_{k+1}$ via \eqref{eq:NDQk}.

\begin{remark}
\label{rem:nge3}
Notice that the behaviour of the system $M_k (a_{k+1}\enspace u_k)^\top =  (R_{a, k}\enspace R_{u, k})^\top$ is very different for $n=2$ and for $n\ge 3$, as for $n=2$, $\det(M_k)=0$, and for $n\ge 3$, we have $\det(M_k)\neq0$. 
\end{remark}

Finally, the $dv/dH$-identity of \eqref{eq:rp-ode} at order $k-1$ allows us to compute $v_k$. More precisely, we have the equation  $v_k=\frac{R_{v,k}}{2(n-2)(k+n- 1)}$ where
\begin{equation}\label{eq:rec-v}
	\begin{aligned}
		 R_{v, k}&:= -(n- 1) \Big[ D_{k+1} v_{-1} + \sum_{\substack{i+j=k\\ 2\le i\le k-1 \\ 1\leq j \leq k-2 }}  D_i  v_{j}\Big] -2(2n- 1)(n+1) N_{2,k-1}  -N_{2,k+1} \\
		& + \Big[ Q_{-2}N_{2,k+1} + \sum_{\substack{i+j=k-1\\ 0\le i\le k-3 \\ 2\le j\le k-1}}  Q_i\,N_{2,j} \Big] + (n- 1) \Big[ D_{k+1} Q_{-2} v_{-1} + v_{-1}  \sum_{\substack{i+j=k-1\\ 2\le i\le k-1, \\ 0\le j\le k-3}}  D_i\,Q_j\\
        & - Q_{-2} \sum_{\substack{i+l=k\\ 2\le i \le k-1 \\ 1 \leq l \leq k-2}}  l D_i v_{l} - D_0 \sum_{\substack{j+l=k-2\\ 0\le j\le k-3 \\ 1\le l\le k-2}}  l Q_j v_{l} - \sum_{\substack{i+j+l=k-2\\ 2\le i\le k-3,\,0\le j\le k-3\\ 1\le l\le k-2}}  l D_i Q_j v_{l}  \Big].
	\end{aligned}
\end{equation}
To sum up, the discussion above yields the following:

\begin{lemma}
\label{lem:recRPn} 
Fix an integer $n\ge3$ and $(b_2,u_1)\in\C^2$, with $\mathrm{Re}(b_2) > 0$. Let $(a_k,b_k,u_k,v_k)_k$ be the coefficients generated from \eqref{eq:a2u1}, \eqref{eq:rec-b}--\eqref{eq:rec-v}, with $N_{1, k}, N_{2, k}, D_k$, and $Q_k$ given by \eqref{eq:first_coefs_NDQ}, \eqref{eq:NDQk}.  

We consider $a, b, u, v$ as formal power series given from \eqref{eq:seriesRP} and $N_i, D, Q$ given by the formal power series $\sum x_k H^{-k}$ for $x\in \{ N_i, D, Q \}$. Then, the ODE \eqref{eq:rp-ode} and the definitions \eqref{eq:abuvquantities} and \eqref{eq:Qcoord} are satisfied as equalities of formal power series.
\end{lemma}

\begin{remark} While the series \eqref{eq:seriesRP} only make sense for real coefficients, we will consider their complex extension by allowing the coefficients to be complex. The reason for that is that at some point in the next section we will be interested in bounding the first order derivatives of \eqref{eq:seriesRP} with respect to the free parameters $(b_2, u_1)$. For that, it will be useful to use the Cauchy bound for holomorphic functions, and for that we need to allow $(b_2, u_1)$ to be complex. Since all our bounds on the series will be by propagating absolute values through triangle inequality, this distinction is not really relevant for the proof.
\end{remark}

 \subsection{Catalan estimates and convergence of the $H$-expansion at $\R\s P^n$} 
 \label{subsec:catRPn}
We start this subsection by proving Lemma~\ref{lem:inductive_RPn}. That Lemma will propagate inductive bounds on the coefficients of the $H$-expansions of $a,b,u,v$.

In order to do that, we will estimates based on Catalan numbers. Recall that $\mathrm{Cat}(k):=\tfrac1{k+1}\binom{2k}{k}$ is the $k$-th Catalan number. Two properties make Catalan numbers especially useful for bounding series coefficients in our $H$-expansions: the convolution identity
\[
\mathrm{Cat}(k+1) = \sum_{i+j = k,\ i,j\geq 0} \mathrm{Cat}(i)\, \mathrm{Cat}(j),
\]
and the exponential bound $\mathrm{Cat}(k) \leq 4^k$. Since our recursions produce sums of the type $\sum_{i+j = k-1} x_i y_{j}$, these two facts will provide exponential bounds on the coefficients. We state our main estimate in an abstract way here, so that it can be applied to $a_{m+1}, b_{m+1}, u_m,$ and $v_m$.  Let $x,y,z$ be sequences of complex numbers satisfying:
	\[
	|x_p|\le C_x\,\mathrm{Cat}(p+\sigma_x)\,C_{\mathrm{rad}}^{\,p+\sigma_x},\qquad
	|y_p|\le C_y\,\mathrm{Cat}(p+\sigma_y)\,C_{\mathrm{rad}}^{\,p+\sigma_y},\qquad
	|z_p|\le C_z\,\mathrm{Cat}(p+\sigma_z)\,C_{\mathrm{rad}}^{\,p+\sigma_z},
	\]
    for every $0 \le p \le m$, where the shifts satisfy $\sigma_x, \sigma_y, \sigma_z \geq -p$. Then, we have 
	\begin{equation}\label{eq:catconv}
		\begin{gathered}
			\Big|\sum_{\substack{ i+j=m \\ i\geq 0, \;\; i\geq -\sigma_x \\ j \geq 0, \;\; j \geq - \sigma_y}} x_iy_j\Big|\le C_xC_y\,\mathrm{Cat}(m + \sigma +1)\,C_{\mathrm{rad}}^{m + \sigma},\\
			\Big|\sum_{\substack { i+j+l=m \\ i,j,l \geq 0 \;\; i \geq -\sigma_x \\ j\geq -\sigma_y, \;\; l \geq -\sigma_z }}x_i y_j z_l\Big|\le C_xC_yC_z\,\mathrm{Cat}(m+\sigma +2)\,C_{\mathrm{rad}}^{ m + \sigma},
		\end{gathered}
	\end{equation}
	with $\sigma:=\sigma_x+\sigma_y$ in the first case and $\sigma : = \sigma_x + \sigma_y +\sigma_z$ in the second one. We will apply \eqref{eq:catconv} with shift $\sigma_a = \sigma_b = -1$ for the sequences $a_k, b_k$ (by \eqref{eq:catbounds}) and also $\sigma_D = \sigma_N = -1$ for $D_k, N_{1,k}, N_{2,k}$, as we will see in \eqref{eq:auxC-def}. We also set $\sigma_u = \sigma_v =0$ for $u_k, v_k$, and finally we set $\sigma_Q = 1$ for $Q_k$. 
\begin{lemma}\label{lem:inductive_RPn}
	Fix an integer $n\ge3$, let $(b_2,u_1)\in\C^2$ with $\mathrm{Re}(b_2)>0$, and consider $(a_k,b_k,u_k,v_k)_k$ the coefficients of the $H$-expansions as in Lemma~\ref{lem:recRPn}. Let $C_a, C_b, C_u, C_v>0$ and $C_{\mathrm{rad}}>1$, and define the auxiliary constants $C_D, C_{N_1}, C_{N_2}, C_Q$ according to \eqref{eq:auxC-def} below. Assume that the closure inequalities \eqref{eq:close-b}, \eqref{eq:close-au}, and \eqref{eq:close-v} hold. If, for some odd $m\ge3$, the bounds
	\begin{equation}\label{eq:catbounds}
		\begin{aligned}
			|a_{k+1}|&\le C_a\,\mathrm{Cat}(k)\,C_{\mathrm{rad}}^{\,k}, &
			|b_{k+1}|&\le C_b\,\mathrm{Cat}(k)\,C_{\mathrm{rad}}^{\,k},\\
			|u_k|&\le C_u\,\mathrm{Cat}(k)\,C_{\mathrm{rad}}^{\,k}, &
			|v_k|&\le C_v\,\mathrm{Cat}(k)\,C_{\mathrm{rad}}^{\,k}
		\end{aligned}
	\end{equation}
	hold for every odd $k$ with $1\le k\le m-2$, then they also hold at the next odd level $k=m$.
\end{lemma}
\begin{proof}
 We structure the proof as follows. First, we obtain bounds for the coefficients in the $H$-expansions of $D,N_1,N_2$ and $Q$.  Second, we obtain bounds for the coefficients in the $H$-expansions of $b$, then of $a$ and $u$, and finally of $v$.

Now, we go back to \eqref{eq:NDQk} and use our hypotheses \eqref{eq:catbounds} on $a_k$, $b_k$, $u_k$, $v_k$ to obtain bounds on $D_k, N_{1, k}$, $N_{2, k}$, $Q_k$. Notice that the $\sum$ terms from \eqref{eq:NDQk} involve only coefficients with indices $i,j \geq 2$ (resp.\ $i,j\geq1$ for $Q_k$), meaning we can apply \eqref{eq:catbounds} to bound the inner terms and then use \eqref{eq:catconv} directly (this is the reason why we separated the terms involving $a_0$ and $v_{-1}$ on \eqref{eq:NDQk}). For example, in the case of $D_k$, we obtain
    \begin{align*}
    |D_k| &\leq 2(n-1)(|a_k| + |b_k|) + 2|b_k| + 2|a_k| + \sum_{\substack{ i + j = k \\ i,j \geq 2 \\ i, j \; \text{even}}} (2|a_ib_j| + |a_ia_j| + |b_ib_j|)  \\
    &\leq 2n (C_a + C_b) C_{\mathrm{rad}}^{k-1} \mathrm{Cat}(k-1) + \sum_{\substack{ i + j = k \\ i,j \geq 2 \\ i, j \; \text{even}}} (2 C_aC_b + C_a^2 + C_b^2)\mathrm{Cat}(i-1) \mathrm{Cat}(j-1) C_{\mathrm{rad}}^{i+j-2}\\
    &\leq \underbrace{\left( 2n(C_a + C_b) + \frac{2C_aC_b + C_a^2 + C_b^2}{C_{\rm{rad}}}\right)}_{:= C_D} C_{\mathrm{rad}}^{k-1} \mathrm{Cat}(k-1),
    \end{align*}
    which holds for any $2 \leq k \leq m-1$. Defining 
    \begin{equation}\label{eq:auxC-def}
		\begin{aligned}
			C_D&:=2n(C_a+C_b)+\tfrac1{C_{\mathrm{rad}}}(C_a+C_b)^2,\\
			C_{N_1}&:=2nC_a+\tfrac1{C_{\mathrm{rad}}}\big(nC_a^2+(n- 2)C_b^2+2(n- 1)C_aC_b\big),\\
			C_{N_2}&:=(4n- 6)C_a+2(n- 1)C_b+\tfrac1{C_{\mathrm{rad}}}\big((n- 2)C_a^2+nC_b^2+2(n- 1)C_aC_b\big),\\
			C_Q&:=\tfrac{2(n-1)C_u+2nC_v}{2n-1}+\frac{1}{(2n-1)C_{\mathrm{rad}}} ( n(n-1) (C_u^2+C_v^2) + 2(n-1)^2 C_u C_v ) + \frac{2(n+1)}{C_{\mathrm{rad}}}, \\
            C_{N_{1}^{\mathrm{des}}} &:= n C_a^2 + (n-2)C_b^2 + 2(n-1)C_aC_b,
		\end{aligned}
	\end{equation}
    and applying the same procedure, we obtain that 
    \begin{equation} \label{eq:catbounds_NDQ}
    |x_k| \leq C_x \mathrm{Cat}(k + \sigma_x) C_{\mathrm{rad}}^{k + \sigma_x}, \qquad \mbox{ for } x\in \{D, N_1, N_2, Q \},
    \end{equation}
    where $\sigma_D = \sigma_{N_1} = \sigma_{N_2} = -1$ and $\sigma_Q = +1$. Bound \eqref{eq:catbounds_NDQ} holds for any $2 \leq k \leq m-1$ if $x\in \{ D, N_1, N_2 \}$ and for any $0 \leq k \leq m-3$ if $x = Q$. For $N_{1, k}^{\mathrm{des}}$, we obtain 
    \begin{equation*} \label{eq:catbounds_N1des}
    |N_{1, k}^{\mathrm{des}}| \leq C_{N_1^{\mathrm{des}}} \mathrm{Cat}(k-1) C_{\mathrm{rad}}^{k-2}, \qquad \mbox{ for } 2 \leq k \leq m+1.
    \end{equation*}
    Let us stress that since $N_{1, k}^{\mathrm{des}}$ does not involve $b_{k+1}$ or $a_{k+1}$ our bound holds up to $k = m+1$.

We can directly bound $b_{m+1}$ by going to its equation \eqref{eq:rec-b}, and combining the hypotheses in~\eqref{eq:catbounds} with the Catalan estimate \eqref{eq:catconv}, as we did previously. Concretely, recalling $R_{b, k}$ from \eqref{eq:rec-b}, we have
	\begin{align*} \label{eq:palermo}
	|R_{b, m}| \leq 2 C_b C_v C_{\mathrm{rad}}^{m-1} \mathrm{Cat}(m) + (m-1) C_Q C_b C_{\mathrm{rad}}^{m-1} \mathrm{Cat}(m).
	\end{align*}
    Now, since $b_{m+1} = \frac{n-1}{m-1} R_{b, m}$ due to \eqref{eq:rec-b}, we can conclude the desired bound \eqref{eq:catbounds} for $b_{m+1}$ provided that 
	\begin{equation}\label{eq:close-b}
    (n-1)(C_v + C_Q) \leq C_{\mathrm{rad}}.
	\end{equation}   
	
The strategy to get bounds on $a_{m+1}$ and $u_m$ is analogous to that of $b_{m+1}$, with the difference that since we use equations \eqref{eq:rec-aU}, we need to bound $R_{a, m}$ and $R_{u, m}$ simultaneously. Analogously to the case of $b_{m+1}$, we obtain
    \begin{equation*} \label{eq:calabria}
    |R_{a, m}| \leq C_a (2C_u + (m-1)C_Q) C_{\mathrm{rad}}^{m-1} \mathrm{Cat}(m).
    \end{equation*}
    
    For $|R_{u, m}|$, using $Q_{-2} = \frac{1}{n-1}$ and using bounds \eqref{eq:catbounds}--\eqref{eq:catconv} in \eqref{eq:rec-aU}, we obtain
    \begin{align} \begin{split} \label{eq:bari}
    | R_{u, m} | &\leq 
    (m-2)\left( C_D C_u + (n-1) \frac{C_D C_Q C_u}{C_{\mathrm{rad}}} + 2(n-1)(n-2) C_Q C_u\right)\mathrm{Cat}(m) C_{\rm{rad}}^{m-1} \\
    &\quad + \Big( (n-1)C_D C_u +  \frac{n-2}{n-1}C_{N_1^{\mathrm{des}}} + 2(2n-1)(n+1) \frac{C_{N_1} }{C_{\mathrm{rad}}} + C_Q C_{N_1} \Big) \mathrm{Cat}(m) C_{\rm{rad}}^{m-1}
     \\
    &\leq 
    \mathrm{Cat}(m) C_{\mathrm{rad}}^{m-1} (m-1) \Big[ C_u
    \left( \frac{n+1}{2}C_D + (n-1)\frac{C_DC_Q}{C_{\mathrm{rad}}} + 2(n-1)(n-2)C_Q \right) \\
    &\quad + \frac{n-2}{2(n-1)}C_{N_1^{\mathrm{des}}} + \frac{2(2n-1)(n+1) C_{N_1}}{2C_{\mathrm{rad}}}  + \frac12 C_Q C_{N_1}\Big] = \mathrm{Cat}(m) C_{\mathrm{rad}}^{m-1} (m-1) C_{R_u},
    \end{split} \end{align}
    where in the last line we define $C_{R_u}$ to be the term inside brackets, which just depends on the other $C_x$ and $n$. We also used the bounds $1\leq \frac{m-1}{2}$ and $m-2 \leq m-1$ to obtain a factor $m-1$ outside. The two top entries of $M_{m}^{-1}$ (from \eqref{eq:rec-aU}) are bounded in absolute value by $\frac{n-1}{m-1}$ and $\frac{1}{m-1} \frac{n-1}{(n-2)(n+4)}$. The two bottom entries of $M_m^{-1}$ are bounded in absolute value by $\frac{1}{(m-1)} \cdot \frac{n}{n+3}$ and $\frac{1}{2(m-1)(n-2)}$. Thus,
    \begin{align} \label{eq:trentino_alto}
    |a_{m+1}| &\leq \left( (n-1)(C_u + C_Q)C_a + C_{R_u} \frac{n-1}{(n-2)(n+4)}\right)  C_{\mathrm{rad}}^{m-1} \mathrm{Cat}(m), \\
    |u_m| &\leq \left( \frac{n}{n+3} C_a(C_u + C_Q) + \frac{1}{2(n-2)} C_{R_u} \right)  C_{\mathrm{rad}}^{m-1} \mathrm{Cat}(m).
    \end{align}
    
    Therefore, we can conclude the desired bounds for $|a_{m+1}|$ and $|u_m|$ provided that
\begin{align}
&(C_u + C_Q)C_a +  \frac{C_{R_u}}{(n-2)(n+4)} \leq \frac{C_a C_{\mathrm{rad}}}{n-1}, \enspace \frac{n}{n+3} C_a(C_u + C_Q) + \frac{C_{R_u}}{2(n-2)}  \leq C_u C_{\mathrm{rad}}, \enspace \mbox{ where } \nonumber\\
&\resizebox{0.92\linewidth}{!}{$\displaystyle
C_{R_u}= C_u\left( \tfrac{n+1}{2}C_D + (n-1)C_Q \left( \tfrac{C_D}{C_{\mathrm{rad}}} + 2(n-2) \right) \right) + \tfrac{(n-2)C_{N_1^{\mathrm{des}}}}{2(n-1)} + C_{N_1} \left( \tfrac{(2n-1)(n+1) }{C_{\mathrm{rad}}} + \tfrac{C_Q}{2} \right)
$}. \label{eq:close-au}
\end{align}

	Finally, we obtain the bounds for $v_m$. We start working on \eqref{eq:rec-v}, recalling that $v_{-1} = \frac{-1}{n-1}$, $Q_{-2} = \frac{1}{n-1}$ and $D_0 = 2(n-2)$ from \eqref{eq:first_coefs_NDQ}. Then, using the bounds \eqref{eq:catbounds}, and also \eqref{eq:catconv} to treat any sum, we have that
    \begin{align*}
    |R_{v, m}| &\leq \mathrm{Cat}(m) C_{\mathrm{rad}}^m \Big( \frac{n}{n-1}C_D  + \frac{2(2n-1)(n+1) C_{N_2}}{C_{\mathrm{rad}}^2} + \frac{n-2}{n-1}C_{N_2}  + \frac{(n-1)(m-2) C_D C_Q C_v }{C_{\mathrm{rad}}^2}  \\
    & \qquad + \frac{C_Q C_{N_2} + C_DC_Q + (m+n-3) C_D C_v + 2(n-1)(n-2) (m-2) C_Q C_v}{C_{\mathrm{rad}}}  \Big)
    \end{align*}
    where the first term corresponds to the two $D_{m+1}$ terms, the second to $N_{2, k-1}$ and the third one to the two terms with $N_{2, k+1}$ in \eqref{eq:rec-v}. The fourth and fifth terms correspond to the triple sum and all double sums combined, respectively. Therefore, using $m+n-1 \geq 5$, and $v_{m} = \frac{1}{2(n-2)(m+n-1)}R_{v, m}$, we can conclude the bound $|v_m| \leq C_v \mathrm{Cat}(m)C_{\mathrm{rad}}^m $ as long as 
    \begin{align} \begin{split} \label{eq:close-v}
\frac{1}{2(n-2)}\Big( \frac{n C_D}{5(n-1)} &+ \frac{2(2n-1)(n+1) C_{N_2}}{5C_{\mathrm{rad}}^2} + (n-1) \frac{C_D C_Q C_v}{C_{\mathrm{rad}}^2}  + \frac{ C_Q C_{N_2} + C_D C_Q}{5C_{\mathrm{rad}}}   \\
& + \frac{C_DC_v + 2(n-1)(n-2)C_QC_v}{C_{\mathrm{rad}}} \Big) + \frac{C_{N_2}}{10(n-1)} \leq C_v,
    \end{split} \end{align}
concluding the proof of the lemma.
\end{proof}

We can now state the main result of this subsection, which will be derived from Lemma \ref{lem:inductive_RPn} by induction. 

\begin{proposition}\label{prop:RPseries}
	Fix $n\ge 3$. Consider $(b_2, u_1) \in \mathbb{R}_+ \times \mathbb R$ and let $(a_k,b_k,u_k,v_k)_k$ be the coefficients of the $H$-expansions as in Lemma~\ref{lem:recRPn}. Then, the series given in \eqref{eq:seriesRP} converge uniformly on $(-\infty, H_-]$ for a certain $H_-<0$, and the resulting $(a, b, u, v)$ give a smooth solution to \eqref{eq:abUV}.

    Moreover, we have the following explicit quantitative estimate. Fix $r > 0$ and $m \in \mathbb{Z}$ with $m \geq 3$ odd and assume a choice of constants $C_a, C_b, C_u, C_v > 0$, and $C_{\mathrm{rad}} > 1$ such that the hypotheses of Lemma \ref{lem:inductive_RPn} are satisfied for any initial coefficients $(b_2', u_1') \in B_r(b_2) \times B_r(u_1) \subset \mathbb C^2$. 
    
    Then, for $x \in \{ a, b, u, v \}$ we have that 
    \begin{align} \begin{split} \label{eq:taylor_residual_rp}
    \left| x(H) - \sum_{k \leq m} x_k H^{-k} \right| &\leq \frac{4^{m+1} C_x C_{\mathrm{rad}}^{m+1}}{|H|^{m} (|H| - 4 C_{\mathrm{rad}})}  \\
    \left| \partial_y x(H) - \sum_{k \leq m} (\partial_y x_k) H^{-k} \right| &\leq \frac{4^{m+1} C_x C_{\mathrm{rad}}^{m+1}}{r |H|^{m} (|H| - 4 C_{\mathrm{rad}})}  \quad \mbox{ for } y\in \{ b_2, u_1 \}
    \end{split} \end{align}
    for any $H < -4C_{\mathrm{rad}}$. In particular, we see that we can take any $H_- < -4 C_{\mathrm{rad}}$ in the first part of the statement.
		
\end{proposition}

\begin{proof}
We start by proving the quantitative estimate, which will end up implying the qualitative part of the proposition. For any fixed $(b_2', u_1')\in B_r(b_2) \times B_r(u_1)$ we construct the series given by Lemma~\ref{lem:recRPn}. By Lemma \ref{lem:inductive_RPn} we can propagate by strong induction the bounds \eqref{eq:catbounds} so we obtain that \eqref{eq:catbounds} holds also for any $k$ odd, with constants $C_a, C_b, C_u, C_v, C_{\mathrm{rad}}$ independent of $(b_2', u_1')\in B_r(b_2)\times B_r(u_1)$. Using that $\mathrm{Cat}(k) \leq 4^k$, and bounding we obtain for each for $x\in\{a,b,u,v\}$,
\begin{equation*}
\sum_{k\geq m+1} C_x C_{\mathrm{rad}}^k |H|^{-k}\mathrm{Cat}(k) \leq C_x \sum_{k = m+1}^\infty \left( \frac{ 4 C_{\mathrm{rad}} }{|H|} \right)^k   = \frac{4^{m+1} C_x C_{\mathrm{rad}}^{m+1}}{|H|^{m+1}} \frac{1}{1 - 4 C_{\mathrm{rad}} / |H|} 
\end{equation*}
we directly obtain our first bound in \eqref{eq:taylor_residual_rp} and that the series \eqref{eq:seriesRP} are absolutely convergent on $(-\infty, H_-]$ for any $H_- < -4C_{\mathrm{rad}}$. In particular, all the left-hand sides in \eqref{eq:rp-ode} can also be expanded in convergent series on $(-\infty, H_-]$. Since we found $(a_k, b_k, u_k, v_k)$ via \eqref{eq:rec-b}--\eqref{eq:rec-v} by precisely imposing that those coefficients are zero, we obtain that all the left-hand sides of \eqref{eq:rp-ode} are zero, meaning that $(a, b, u, v)$ solve \eqref{eq:abUV}.

Regarding the derivative bounds, we start noting from \eqref{eq:a2u1} and \eqref{eq:rec-b}--\eqref{eq:rec-v} that all coefficients $(a_k, b_k, u_k, v_k)$ are polynomials in $(b_2', u_1')$. This implies that the series \eqref{eq:seriesRP} are actually holomorphic in $b_2'$ and $u_1'$. That is, if we track the full dependence of $x \in \{ a, b, u, v \}$ in terms of $H \in (-\infty, H_-]$, $b_2'\in B_r(b_2)$ and $u_1'\in B_r(u_1)$, then $x(H, b_2', u_1')$ is holomorphic on $b_2'$ (for every $H, u_1'$) or holomorphic in $u_1'$ (for every $H, b_2'$). The same is true for the series residual of $\sum_{k \geq m+1} x_k H^{-k}$. By Cauchy bounds on the derivatives of holomorphic functions, we obtain the second estimate of \eqref{eq:taylor_residual_rp}.

Finally, in order to prove the qualitative statement, it suffices to show that for any $(b_2, u_1) \in \mathbb R_+ \times \mathbb R$ we can find $r > 0$ and $C_{a}, C_b, C_u, C_v, C_{\mathrm{rad}}$ such that the hypotheses of Lemma \ref{lem:inductive_RPn} are satisfied for any $(b_2', u_1') \in B_r(b_2) \times B_r(u_1)$ and $m = 3$. Take $r = b_2/2$ so that $b_2' > b_2/2$ and then take $C_a, C_b, C_u, C_v > 0$ satisfying the bounds \eqref{eq:catbounds} for $k \leq m-2 = 1$. If needed, enlarge $C_v$ such that $\frac{n C_D + (n-2)C_{N_2}}{10(n-1)(n-2)}\leq \frac12 C_v$ (we can always do this since we see in \eqref{eq:auxC-def} that $C_D$ and $C_{N_2}$ are uniformly bounded for $C_{\mathrm{rad}} > 1$, with a bound just dependent on $n, C_a, C_b$). Then, taking $C_{\mathrm{rad}}$ sufficiently large in terms of $C_a, C_b, C_u, C_v$, it is clear that we satisfy \eqref{eq:close-b}, \eqref{eq:close-au}, \eqref{eq:close-v} (in the case of \eqref{eq:close-v}, using also $\frac{n C_D + (n-2)C_{N_2}}{10(n-1)(n-2)}\leq \frac12 C_v$). Therefore, as we found constants satisfying the hypotheses of Lemma~\ref{lem:inductive_RPn}, we conclude the proof of the proposition.
\end{proof}

\subsection{The expansion around the complex quadric}\label{subsec:expquadric}

Here, we work analogously to Subsection \ref{subsec:expRP}, but with the system \eqref{eq:abUVquadric}, which is the system that we employ close to $H = + \infty$ (corresponding to the complex quadric singular orbit). First of all, we clear the denominators in \eqref{eq:abUVquadric}, obtaining
\begin{equation} \label{eq:masterEQ_expquadric}
\begin{aligned}
Q \frac{d\alpha}{dH} &= 2(\alpha \mu  + \beta \eta), & (n-1) D Q \frac{d\mu}{dH} &= (n-1) DH \mu + (H^2 - Q + (2n-1) \lambda ) N_+, \\
Q \frac{d\beta}{dH} &= 2(\alpha \eta + \beta \mu),  &  (n-1) D Q \frac{d\eta}{dH} &= (n-1) DH \eta + (H^2 - Q + (2n-1) \lambda ) N_-
\end{aligned}
\end{equation}
and we recall \eqref{eq:NDQquadric} for the formulas of $N_\pm, D, Q$ in terms of $\alpha, \beta, \mu, \eta$. Recalling also the $H$-expansions for $\alpha, \beta, \mu, \eta$ given in \eqref{eq:seriesquadric}, and expanding $X = \sum_{k} X_k H^{-k}$ for $X \in \{ N_\pm, D, Q  \}$, we have
\begin{align} \begin{split}\label{eq:NDQk_quadric}
Q_k &= \frac{2(n-1)}{2n-1} \left( 2\mu_{-1}\mu_{k+1} +  \sum_{\substack{i+j=k\\ i,j\geq 1}} \mu_i \mu_j \right)+ 2(n-1) \left( 2\eta_0 \eta_k + \sum_{\substack{i+j=k\\ i,j\geq 2}} \eta_i \eta_j \right) , \\
D_k &= 4(n-1)\alpha_k - 4 \sum_{\substack{i+j=k\\ i,j\geq 1}} \beta_i \beta_j - 8 \beta_{-1} \beta_{k+1}, \\
N_{+, k} &= 2(n-1)\alpha_k + 4(n-1) \sum_{\substack{i+j=k\\ i,j\geq 1}} \beta_i \beta_j + 8(n-1) \beta_{-1} \beta_{k+1}, \\
N_{-, k} &= 4 \sum_{\substack{i+j=k\\ i\geq 0,\, j\geq 1}} \alpha_i \beta_j + 4\alpha_{-2}\beta_{k+2} + 4 \beta_{-1} \alpha_{k+1} - 2(n-1)\beta_k, 
\end{split} \end{align}
where $k \geq 2$ and even on the $Q_k, D_k, N_{+, k}$ equation, and $k \geq -1$ is odd on the equation of $N_{-, k}$. The lower order terms of the expansion can be obtained analogously, but the formulas do not follow the same pattern, and they are given by:
\begin{align}  \begin{split} \label{eq:first_coefs_NDQ_quadric}
N_{+,-2} &= 2(n-1)\alpha_{-2} + 4(n-1)\beta_{-1}^2 , \qquad N_{-, -3} = 4\alpha_{-2}\beta_{-1},  \\
D_{-2} &= 4\big((n-1)\alpha_{-2} - \beta_{-1}^2\big), \qquad Q_{-2} = 1, \qquad Q_0 = \lambda + \frac{4(n-1)}{2n-1}\,\mu_1 + 2(n-1)\,\eta_0^2, \\
D_0 &= 4(n-1)\alpha_0 - 8\beta_{-1}\beta_1 - 1, \qquad
N_{+, 0} = 2(n-1)\alpha_0 + 8(n-1)\beta_{-1}\beta_1 - n.
\end{split} \end{align}

First of all, we recall from \eqref{eq:seriesquadric} that we impose $\mu_{-1} = 1$, and we let $(\alpha_{-2}, \beta_{-1})$ be our free inputs determining the $H$-expansions. The remaining low-order coefficients $\eta_0$ and $\mu_1$ are then obtained from the $H^{2}$ term in the $\frac{d\beta}{dH}$ and $\frac{d\mu}{dH}$ equations in \eqref{eq:masterEQ_expquadric}, respectively. Together with $\mu_{-1}$, these are given by:
\begin{equation} \label{eq:firstcoef-quad}
\mu_{-1} = 1, \qquad \eta_0 = -\frac{\beta_{-1}}{2\alpha_{-2}}, \qquad \mu_1 = (2n-1)(n+1) - \frac{(2n-1)(n-1)}{8\alpha_{-2}}.
\end{equation}

Now, we assume that we know $\alpha_k$, $\beta_{k+1}$, $\eta_{k+2}$, $\mu_{k+3}$ for all $-2 \leq k < m$, with $k, m$ even. In particular, via formulas \eqref{eq:NDQk_quadric}, we can compute $D_k, N_{+, k}, Q_k$ for even $k<m$, as well as $Q_m$, and $N_{-, k}$ for any $k<m-1$ odd. Then, developing the first equation of \eqref{eq:masterEQ_expquadric} in series and equating the $H^{-(m-1)}$ terms, we obtain
\begin{align} \begin{split}
- (Q_{-2} m + 2\mu_{-1}) \alpha_{m}  &=  \sum_{\substack{i+j = m-2 \\ 0\leq i \leq m \\ 2 \leq j \leq m-2}} Q_i j \alpha_j - 2Q_m \alpha_{-2} + 2 \alpha_{-2} \mu_{m+1} + 2 \sum_{\substack{ i + j = m-1 \\ 0 \leq i \leq m-2 \\ 1 \leq  j \leq m-1 }} \alpha_i \mu_j \\
&\quad + 2  \beta_{-1} \eta_m + 2 \sum_{\substack{ i+j = m-1 \\ 1 \leq i \leq m-1 \\ 0 \leq j \leq m }} \beta_i \eta_j  =: R_{\alpha, m}. \label{eq:alpha_k}
\end{split} \end{align}
We observe that all the terms in $R_{\alpha, m}$ are known and therefore using \eqref{eq:first_coefs_NDQ_quadric}--\eqref{eq:firstcoef-quad}, we can compute $\alpha_m$ as $\alpha_m = \frac{-R_{\alpha, m}}{m+2}$. Next, we proceed to look at the equation for $d \beta / dH$ in \eqref{eq:masterEQ_expquadric}. Taking the $H^{-m}$ term, we obtain that
\begin{align*}
-(m+3)\,\beta_{m+1} &= 2\alpha_{-2}\,\eta_{m+2} + R_{\beta, m},
\end{align*}
where
\begin{equation}\label{eq:cq-Rbeta}
R_{\beta, m} := -\beta_{-1}Q_m +   \sum_{\substack{i+j=m-1\\ 0\le i\le m-2\\ 1\le j\le m-1}} j Q_i \beta_j   + 2 \eta_0 \alpha_m + 2\sum_{\substack{i+j=m\\ 0\le i\le m-2\\ 2\le j\le m}} \alpha_i\, \eta_j  + 2\beta_{-1}\mu_{m+1} +  2\sum_{\substack{i+j=m\\ 1\le i\le m-1\\ 1\le j\le m-1}} \beta_i \mu_j 
\end{equation}
In this case, we cannot solve for $\beta_{m+1}$ since $\eta_{m+2}$ is not known yet; however, we can compute $R_{\beta, m}$ and express
\begin{equation} \label{eq:cq-beta}
\beta_{m+1} = -\frac{2\alpha_{-2}}{m+3}\, \eta_{m+2} - \frac{1}{m+3}\, R_{\beta, m}.
\end{equation}

Now, we proceed to look at the $H^{-(m-1)}$ term in the $d\eta / dH$ equation in \eqref{eq:masterEQ_expquadric} and the $H^{-m}$ in $d\mu / dH$ one. First notice that since $Q_{-2} = 1$ from \eqref{eq:first_coefs_NDQ_quadric}, we have that in the parenthesis $(H^2 - Q + (2n-1)\lambda)$ in \eqref{eq:masterEQ_expquadric} the $H^2$ cancels with the leading term of $Q$. Similarly, we notice a cancellation at leading order $(n-1) DQ \frac{d\mu}{dH} - (n-1) DH \mu$ when $Q\sim H^2$ and $\mu \sim H$. We obtain that
\begin{equation} \label{eq:circa}
c_{\circ \eta} \eta_{m+2} + c_{\circ \mu} \mu_{m+3} + c_{\circ \beta} \beta_{m+1} = R_{\circ, m}, \qquad \mbox{ for } \circ \in \{ \eta, \mu \}
\end{equation}
where
\begin{align} \begin{split}
c_{\eta \eta} &= -4(n-1)\big( (m+3)(n-1)\alpha_{-2}-(m+1)\beta_{-1}^2\big), \qquad 
c_{\eta \mu} = \frac{16(n-1)\alpha_{-2}\beta_{-1}}{2n-1}, \\
c_{\eta\beta} &=-\frac{2(n-1)\big((n-1)\alpha_{-2}+\beta_{-1}^2\big)}{\alpha_{-2}},\qquad 
c_{\mu\beta}=-\frac{2(n-1)^2(2n-1)\beta_{-1}}{\alpha_{-2}} \\
c_{\mu \mu} &= - 4(n-1)\big((m+2)(n-1)\alpha_{-2}-(m+4)\beta_{-1}^2\big), \qquad
c_{\mu\eta} = -4(n-1)^2(2n-1)\beta_{-1}
\end{split} \end{align}
and
\begin{align} \begin{split} \label{eq:cq-Romegatau}
R_{\eta, m} =&  (n-1) \sum_{\substack{j+k=m\\  0\le j\le m-2 \\ 2\le k\le m}} (D_{-2}Q_j + Q_{-2} D_j ) k \eta_k  +  (n-1)\sum_{\substack{i+j+k=m-2\\ 0\le i\le m-2, 0\le j\le m-2\\ 2\le k\le m-2}} k\, D_i Q_j \eta_k \\
&\quad + (n-1) D_m^{\mathrm{des}} \eta_0 + (n-1)\sum_{\substack{i+j=m\\ 0\le i\le m-2 \\ 2\leq j \leq m}} D_i  \eta_{j} 
   \\
  &\quad + ((2n-1)\lambda - Q_0) N_{-,m-1}^{\rm des} - Q_{m+2}^{\mathrm{des}} N_{-, -3} - \sum_{\substack{i+j=m-1\\ 2\le i\le m \\ -1\leq j \leq m-3}} Q_i N_{-,j}, 
  \end{split} \end{align}
\begin{align}
R_{\mu, m} =& (n-1)\big[ D_{-2}Q_m\mu_1 + \sum_{\substack{i+k=m+1\\ 0\le i\le m-2 \\ -1\le k\le m+1}} k (D_{-2} Q_i + D_i Q_{-2}) \mu_k + \sum_{\substack{i+j+k=m-1\\ 0\le i\le m-2,\ 0\le j\le m\\ -1\le k\le m+1}} k D_i Q_j \mu_k \big] \nonumber\\
&\quad + (n-1) \left[ -D_{-2} Q_{m+2}^{\mathrm{des}} \mu_{-1} + D_{m}^{\mathrm{des}} (Q_{-2}\mu_1 - Q_0 \mu_{-1})\right]  +(n-1)D_m^{\mathrm{des}} \mu_{1} \label{eq:cq-Rtau}  \\
  & \quad +  (n-1)\sum_{\substack{i+j=m+1\\ 0\le i\le m-2 \\ 3\le j\le m+1}} D_i \mu_{j} + \big((2n-1)\lambda - Q_0\big)\, N_{+,m}^{\mathrm{des}} - N_{+, -2} Q_{m+2}^{\mathrm{des}} - \sum_{\substack{i+j=m\\ 2\le i\le m \\ 0\le j\le m-2}} Q_i N_{+,j}.\nonumber 
\end{align}
Here $D_m^{\rm des}$, $N_{+,m}^{\rm des}$, $N_{-,m-1}^{\rm des}$, $Q_{m+2}^{\rm des}$ denote these coefficients with their top-order term removed, that is, the term carrying the not-yet-determined $\beta_{m+1},\eta_{m+2},\mu_{m+3}$:
\begin{equation}\label{eq:des-quad}
\begin{aligned}
D_m^{\rm des} &:= D_m + 8\beta_{-1}\beta_{m+1}, &\qquad N_{+,m}^{\rm des} &:= N_{+,m} - 8(n-1)\beta_{-1}\beta_{m+1},\\
N_{-,m-1}^{\rm des} &:= N_{-,m-1} - 4\alpha_{-2}\beta_{m+1}, &\qquad Q_{m+2}^{\rm des} &:= Q_{m+2} - \tfrac{4(n-1)}{2n-1}\mu_{m+3} - 4(n-1)\eta_0\eta_{m+2},
\end{aligned}
\end{equation}
for $m \geq 0$. With this, $R_{\eta,m}$ and $R_{\mu,m}$ involve only $\alpha_m$ and coefficients of lower order; and the removed terms are exactly those retained on the left-hand side of \eqref{eq:circa} through $c_{\circ\beta}$, $c_{\circ\eta}$, $c_{\circ\mu}$, that is, the ones involving $\beta_{m+1}$, $\eta_{m+2}$, or $\mu_{m+3}$. Thus, substituting \eqref{eq:cq-beta} in \eqref{eq:circa}, we get
\begin{equation}
\begin{pmatrix}
c_{\eta \eta}-\frac{2\alpha_{-2}c_{\eta \beta}}{m+3} & c_{\eta \mu} \\
 c_{\mu \eta}-\frac{2\alpha_{-2}c_{\mu \beta}}{m+3} & c_{\mu \mu}
\end{pmatrix}
\cdot
\begin{pmatrix}
\eta_{m+2} \\
\mu_{m+3}
\end{pmatrix}
=
\begin{pmatrix}
R_{\eta, m} + \frac{c_{\eta \beta}}{m+3} R_{\beta, m} \\
R_{\mu, m} + \frac{c_{\mu \beta}}{m+3} R_{\beta, m} 
\end{pmatrix}
\end{equation}
The determinant is 
\begin{align} \begin{split}
\Delta &:= \left( c_{\eta \eta}-\frac{2\alpha_{-2}c_{\eta \beta}}{m+3} \right) c_{\mu \mu} - \left( c_{\mu \eta}-\frac{2\alpha_{-2}c_{\mu \beta}}{m+3} \right) c_{\eta \mu} \\
&= \frac{16\,(m+2)^2(m+4)(n-1)^2\big((n-1)\alpha_{-2}-\beta_{-1}^2\big)^2}{m+3}
\end{split} \end{align}
Therefore, the matrix is invertible for $m > -2$ and $(n-1)\alpha_{-2} \neq \beta_{-1}^2$, and we obtain
\begin{equation} \label{eq:omega_tau}
\begin{pmatrix}
\eta_{m+2} \\ \mu_{m+3}
\end{pmatrix}
= \frac{1}{\Delta}
\begin{pmatrix}
c_{\mu \mu} & - c_{\eta \mu} \\
-c_{\mu \eta} + \frac{2\alpha_{-2}c_{\mu \beta}}{m+3} & c_{\eta \eta}-\frac{2\alpha_{-2}c_{\eta \beta}}{m+3} 
\end{pmatrix}
\cdot 
\begin{pmatrix}
R_{\eta, m} + \frac{c_{\eta \beta}}{m+3} R_{\beta, m} \\
R_{\mu, m} + \frac{c_{\mu \beta}}{m+3} R_{\beta, m} 
\end{pmatrix}
\end{equation}

\begin{lemma}
\label{lem:recQn-1} 
Fix $(\alpha_{-2}, \beta_{-1}) \in \mathbb C^2 $ with $\alpha_{-2}\neq 0$ and $(n-1)\alpha_{-2} \neq \beta_{-1}^2$ (equivalently, $D_{-2}\neq 0$). We define $(\alpha_k, \beta_k, \eta_k, \mu_k)$ from \eqref{eq:firstcoef-quad}, \eqref{eq:alpha_k}, \eqref{eq:cq-beta}, \eqref{eq:omega_tau}, with $N_{\pm, k}, D_k, Q_k$ defined from \eqref{eq:NDQk_quadric}.

Taking $x = \sum_{k} x_k H^{-k}$ as a formal power series, for $x \in \{ \alpha, \beta, \eta, \mu, N_{\pm}, D, Q \}$, then the ODE system \eqref{eq:masterEQ_expquadric} and the definitions \eqref{eq:NDQquadric} hold as equalities of formal power series.
\end{lemma}

 \subsection{Catalan estimates and convergence of the $H$-expansion at $\s Q_{n-1}$} \label{subsec:catQC}

Here, we proceed analogously to Subsection \ref{subsec:catRPn} but for the expansion at the complex quadric $\s Q_{n-1}$. That is, we want to show that the formal series from Lemma \ref{lem:recQn-1} are actually convergent in a neighbourhood of $H = +\infty$ and obtain explicit bounds on the remainder. The main ingredient is the following lemma, which propagates bounds through the recurrences \eqref{eq:alpha_k}, \eqref{eq:cq-beta}, \eqref{eq:omega_tau}, \eqref{eq:NDQk_quadric}.

\begin{lemma}\label{lem:inductive_CQ}
	Fix an integer $n\ge 3$, let $(\alpha_{-2}, \beta_{-1}) \in\C^2$ such that $\alpha_{-2} \neq 0$ and $(n-1)\alpha_{-2} \neq \beta_{-1}^2$, and consider the formal series expansions from Lemma \ref{lem:recQn-1}. 
    
    Let $C_\alpha, C_\beta, C_\eta, C_\mu>0$ and $C_{\mathrm{rad}}>1$, and define the auxiliary constants $C_D,C_{N_+},C_{N_-},C_Q$, as well as the constants $C_{D^{\mathrm{des}}}$, $C_{N_\pm^{\mathrm{des}}}$, $C_{Q^{\mathrm{des}}}$ and $C_{R_\beta}, C_{R_\eta}, C_{R_\mu}$, as in the proof below. 
    
    Assume that the closure inequalities \eqref{eq:close_alpha} and \eqref{eq:close-omega}--\eqref{eq:close-beta} hold. Assume also the bounds
	\begin{equation}\label{eq:catbounds-qc}
		\begin{gathered}
			|\alpha_{k}|\le C_\alpha\,\mathrm{Cat}(k)\,C_{\mathrm{rad}}^{k},\qquad
			|\beta_{k+1}|\le C_\beta \mathrm{Cat}(k) C_{\mathrm{rad}}^{k},\\
			|\eta_{k+2}|\le C_\eta  \mathrm{Cat}(k) C_{\mathrm{rad}}^{k}, \qquad
			|\mu_{k+3}|\le C_\mu \mathrm{Cat}(k) C_{\mathrm{rad}}^{k}
		\end{gathered}
	\end{equation}
	for every even $k$ with $0 \le k\leq m -2$, for some $m \geq 2$ even. Then, the bounds \eqref{eq:catbounds-qc} also hold at $k=m$.
\end{lemma}

\begin{proof}
The strategy is very similar to that of Lemma \ref{lem:inductive_RPn}. In particular, we recall the first part of the proof of Lemma \ref{lem:inductive_RPn}. We recall \eqref{eq:catconv}, which we will use in our case with $\sigma_\alpha = 0$, $\sigma_\beta = -1$, $\sigma_\eta = -2$ and $\sigma_\mu = -3$. We also fix $\sigma_{D} = 0$, $\sigma_{N_+} = 0$, $\sigma_{N_-} = 1$, and $\sigma_Q = -2$, as will become clear in the next steps.

First, we study $Q_k$ for $k \geq 2$ from \eqref{eq:NDQk_quadric}, using our induction bounds on $\mu, \eta$ together with \eqref{eq:catconv}. We obtain
\begin{equation} \label{eq:Q_bound}
|Q_k| \leq C_Q \mathrm{Cat}(k-2)C_{\mathrm{rad}}^{k-2}
\end{equation}
for all $2\leq k \leq m$ and
\begin{align*}\begin{split}
C_Q := \frac{2(n-1)}{2n-1} \left( 2 C_\mu + |\mu_1|^2 + 2|\mu_1| \frac{C_\mu}{C_{\mathrm{rad}}^2} + \frac{C_\mu^2}{C_{\mathrm{rad}}^4} \right) + 2(n-1) \left( 2| \eta_0 | C_\eta + \frac{C_\eta^2}{C_{\mathrm{rad}}^2} \right).
\end{split} \end{align*}

Equipped with \eqref{eq:Q_bound}, and using again our induction bounds \eqref{eq:catbounds-qc} together with \eqref{eq:catconv}, we can bound $|R_{\alpha, m}|$, for $m \geq 2$, from \eqref{eq:alpha_k} as
\begin{align*}
|R_{\alpha, m}| &\leq C_{\mathrm{rad}}^{m-2} \mathrm{Cat}(m-2) \Big( (m-2) |Q_0| C_\alpha + \frac{(m-2)C_QC_\alpha}{C_{\mathrm{rad}}^2} + 2 C_Q | \alpha_{-2} | + 2|\alpha_{-2}| C_\mu + 2 |\mu_1| C_\alpha  \\
&\qquad + \frac{2C_\mu C_\alpha}{C_{\mathrm{rad}}^2} + 2 | \beta_{-1} | C_\eta + 2 \frac{C_\beta C_\eta}{C_{\mathrm{rad}}^2} + 2 | \eta_0 | C_\beta \Big)
\end{align*}
and therefore since $\alpha_m = \frac{-1}{m+2} R_{\alpha, m}$, we can close the induction bound provided that 
\begin{equation} \label{eq:close_alpha}
|Q_0| + \frac{C_Q}{C_{\mathrm{rad}}^2} + \frac{2(C_Q + C_\mu)|\alpha_{-2}|}{C_\alpha} + 2|\mu_1| + \frac{2C_\mu}{C_{\mathrm{rad}}^2} + \frac{2|\beta_{-1}| C_\eta + 2|\eta_0| C_\beta + 2C_\beta C_\eta / C_{\mathrm{rad}}^2}{C_\alpha} \leq C_{\mathrm{rad}}^2.
\end{equation}

 Now, we focus on $N_{\pm, k}^{\mathrm{des}}$ and $D_k^\mathrm{des}$ from \eqref{eq:NDQk_quadric}--\eqref{eq:des-quad}. Using our bounds \eqref{eq:catbounds-qc} (including the $\alpha_m$ bound from Step 2), as well as \eqref{eq:catconv}, we obtain
\begin{align} \begin{split} \label{eq:des_bounds}
|D_k^{\mathrm{des}}| \leq C_{D^{\mathrm{des}}} \mathrm{Cat}(k) C_{\mathrm{rad}}^k, \quad |N_{+, k}^{\mathrm{des}}| \leq C_{N_+^{\mathrm{des}}} \mathrm{Cat}(k) C_{\mathrm{rad}}^k, \quad |N_{-, k-1}^{\mathrm{des}}| \leq C_{N_-^{\mathrm{des}}} \mathrm{Cat}(k) C_{\mathrm{rad}}^k.
\end{split} \end{align}
for any $2\leq k \leq m$ even, where
\begin{align*}
C_{D^{\mathrm{des}}} &:= 4(n-1)C_\alpha + 4 \frac{C_\beta^2}{C_{\mathrm{rad}}^2}, \qquad C_{N_+^{\mathrm{des}}} := 2(n-1)C_\alpha + 4(n-1) \frac{C_\beta^2}{C_{\mathrm{rad}}^2}, \\
C_{N^{\mathrm{des}}_-} &:= 4 \frac{C_\alpha C_\beta}{C_{\mathrm{rad}}^2} + 4 |\beta_{-1}| C_\alpha + 2(n-1)\frac{C_\beta}{C_{\mathrm{rad}}^2}.
\end{align*}

Regarding $Q_{m+2}^{\mathrm{des}}$, we analogously obtain
\begin{equation}
|Q_{m+2}^{\mathrm{des}}| \leq \underbrace{\left( \frac{2(n-1)}{2n-1} \left( 2|\mu_1|C_\mu + \frac{C_\mu^2}{C_{\mathrm{rad}}^2} \right) + 2(n-1)C_\eta^2 \right)}_{=:\,C_{Q^{\mathrm{des}}}} \mathrm{Cat}(m-1) C_{\mathrm{rad}}^{m-2}.
\end{equation}

For any $0 \leq k \leq m-2$ even, we obtain
\begin{equation*}
|D_k| \leq C_D \mathrm{Cat}(k)C_{\mathrm{rad}}^{k}, \quad |N_{+, k}| \leq C_{N_+} \mathrm{Cat}(k) C_{\mathrm{rad}}^{k}, \quad |N_{-, k-1}| \leq C_{N_-} \mathrm{Cat}(k) C_{\mathrm{rad}}^{k},
\end{equation*}
where
\begin{align*}
C_D &:= \max \left\{ |D_0|, C_{D^{\mathrm{des}}} + 8|\beta_{-1}| C_\beta \right\} \\
C_{N_+} &:= \max \left\{ |N_{+, 0}|, C_{N_+^{\mathrm{des}}} + 8(n-1)|\beta_{-1}| C_\beta \right\}, \\
C_{N_-} &:= \max \left\{ |N_{-, -1}|, C_{N_-^{\mathrm{des}}} + 4|\alpha_{-2}| C_\beta \right\}. 
\end{align*}

Applying \eqref{eq:catconv} and our bounds from the previous steps on \eqref{eq:cq-Rbeta}, \eqref{eq:cq-Romegatau} and \eqref{eq:cq-Rtau}, and using $4\leq m+2$, we obtain
\begin{equation}
|R_{x, m}| \leq (m+2) C_{R_x} \mathrm{Cat}(m) C_{\mathrm{rad}}^m, \qquad \mbox{ for } x\in \{ \beta, \eta, \mu \}
\end{equation}
with
\begin{align}
C_{R_\beta} &:= \frac{|\beta_{-1}| C_Q}{4C_{\mathrm{rad}}^2} + |Q_0| \frac{C_\beta}{C_{\mathrm{rad}}^2} + \frac{C_QC_\beta}{C_{\mathrm{rad}}^4} +  \frac12|\eta_0| C_\alpha + \frac{C_\alpha C_\eta }{2C_{\mathrm{rad}}^2}
+ 2 \frac{|\beta_{-1}| C_\mu}{2C_{\mathrm{rad}}^2} + 2\frac{|\mu_{1}| C_\beta }{2C_{\mathrm{rad}}^2} + 2\frac{C_\beta C_\mu}{2C_{\mathrm{rad}}^4}, \nonumber\\
C_{R_\eta} &:= \frac{(n-1)}{C_{\mathrm{rad}}^2} \Big( C_\eta \left(|D_{-2}||Q_0| + \frac{|D_{-2}|C_Q }{C_{\mathrm{rad}}^2} + C_D\right) + \frac{|Q_0|C_D C_\eta}{C_{\mathrm{rad}}^2} + \frac{C_D C_Q C_\eta}{C_{\mathrm{rad}}^4} \Big) + \frac{n-1}{4} C_{D^{\mathrm{des}}} | \eta_0 |\nonumber\\
&  \quad + \frac{(n-1)C_{D}}{4} \cdot \frac{C_\eta}{C_{\mathrm{rad}}^2} + \frac{(2n-1)\lambda + |Q_0|}{4} C_{N_-^{\mathrm{des}}} + |N_{-, -3}| \frac{C_{Q^{\mathrm{des}}}}{4C_{\mathrm{rad}}^2} + \frac{C_Q C_{N_-}}{4C_{\mathrm{rad}}^2}, \nonumber\\
C_{R_\mu} &:= \frac{(n-1)}{C_{\mathrm{rad}}^2} \Big( C_\mu \Big(|D_{-2}||Q_0| + C_D + \frac{|D_{-2}|C_Q}{C_{\mathrm{rad}}^2}\Big) + C_D \Big(|Q_0||\mu_1| +C_Q+ \frac{|Q_0| C_\mu + |\mu_1| C_Q}{C_{\mathrm{rad}}^2}\Big)  \nonumber\\
& \quad  + \frac{C_D C_Q C_\mu}{C_{\mathrm{rad}}^4} \Big) + \frac{n-1}{4}\big(2|\mu_1| + |Q_0|\big) C_{D^{\mathrm{des}}} + \frac{(2n-1)\lambda + |Q_0|}{4} C_{N_+^{\mathrm{des}}} \nonumber\\
&\quad  + \frac{(n-1)|D_{-2}|\big(|\mu_1|C_Q + C_{Q^{\mathrm{des}}}\big) + |N_{+,-2}| C_{Q^{\mathrm{des}}} + C_Q C_{N_+} + (n-1) C_D C_\mu}{4\,C_{\mathrm{rad}}^2}. \label{eq:CRx-cq}
\end{align}

We now solve for $\eta_{m+2}, \mu_{m+3}$ via \eqref{eq:omega_tau}. Since $(n-1)\alpha_{-2} - \beta_{-1}^2 = D_{-2}/4$ and $\frac{m+4}{m+3} \geq 1$, the determinant satisfies
\begin{equation*}
|\Delta| \,=\, \frac{16(m+2)^2(m+4)(n-1)^2}{m+3}\,\Big|\frac{D_{-2}}{4}\Big|^2 \,\geq\, (m+2)^2 (n-1)^2 |D_{-2}|^2.
\end{equation*}
Moreover, using $m+3 \leq 2(m+2)$ and $m+4 \leq 2(m+2)$, the diagonal coefficients satisfy
\begin{equation*}
|c_{\mu\mu}|,\; |c_{\eta\eta}| \,\leq\, 8(n-1)(m+2)\big( (n-1)|\alpha_{-2}| + |\beta_{-1}|^2 \big),
\end{equation*}
while for the off-diagonal entries of the matrix in \eqref{eq:omega_tau} we note the identity $2\alpha_{-2} c_{\mu\beta} = c_{\mu\eta}$, so that
\begin{equation*}
\Big| -c_{\mu\eta} + \frac{2\alpha_{-2} c_{\mu\beta}}{m+3} \Big| \,=\, \frac{m+2}{m+3}\,|c_{\mu\eta}| \,\leq\, |c_{\mu\eta}|,
\end{equation*}
and, using $2|\alpha_{-2}||c_{\eta\beta}| = 4(n-1)\big|(n-1)\alpha_{-2} + \beta_{-1}^2\big|$ together with $4 \leq m+2$,
\begin{equation*}
\Big| c_{\eta\eta} - \frac{2\alpha_{-2} c_{\eta\beta}}{m+3} \Big| \,\leq\, 9(n-1)(m+2)\big( (n-1)|\alpha_{-2}| + |\beta_{-1}|^2 \big).
\end{equation*}
Finally, by Step 4 and $\frac{m+2}{m+3}\leq 1$, the right-hand side of \eqref{eq:omega_tau} satisfies
\begin{equation*}
\Big| R_{x, m} + \frac{c_{x\beta}}{m+3} R_{\beta, m} \Big| \,\leq\, (m+2)\big( C_{R_x} + |c_{x\beta}|\, C_{R_\beta} \big)\, \mathrm{Cat}(m)\, C_{\mathrm{rad}}^m, \qquad x \in \{\eta, \mu\}.
\end{equation*}
Plugging these into \eqref{eq:omega_tau}, one factor $(m+2)$ of $|\Delta|$ cancels the one above, the diagonal entries consume the other, and the off-diagonal contributions retain a spare factor $\frac{1}{m+2} \leq \frac14$ (resp.\ $\frac{1}{m+3} \leq \frac15$).  Therefore, we can close the induction bounds on $\eta_{m+2}$ and $\mu_{m+3}$ if
\begin{align} \label{eq:close-omega}
\frac{8\big((n-1)|\alpha_{-2}| + |\beta_{-1}|^2\big)\big(C_{R_\eta} + |c_{\eta\beta}|\, C_{R_\beta}\big)}{(n-1)\,|D_{-2}|^2} + \frac{|c_{\eta\mu}|\,\big(C_{R_\mu} + |c_{\mu\beta}|\, C_{R_\beta}\big)}{4(n-1)^2\,|D_{-2}|^2} \;\leq\; C_\eta, \\
 \label{eq:close-tau}
\frac{|c_{\mu\eta}|\,\big(C_{R_\eta} + |c_{\eta\beta}|\, C_{R_\beta}\big)}{5(n-1)^2\,|D_{-2}|^2} + \frac{9\big((n-1)|\alpha_{-2}| + |\beta_{-1}|^2\big)\big(C_{R_\mu} + |c_{\mu\beta}|\, C_{R_\beta}\big)}{(n-1)\,|D_{-2}|^2} \;\leq\; C_\mu.
\end{align}

and then we can deduce the desired bound for $\beta_{m+1}$ provided
\begin{equation} \label{eq:close-beta}
\frac{2|\alpha_{-2}|\, C_\eta}{5} + C_{R_\beta} \;\leq\; C_\beta,
\end{equation}
as we wanted to prove.
\end{proof}

\begin{proposition}\label{prop:QCseries}
	Fix $n\ge 3$. Consider $(\alpha_{-2}, \beta_{-1}) \in \mathbb R^2$ with $\alpha_{-2} \neq 0$ and $(n-1)\alpha_{-2} \neq \beta_{-1}^2$, and let $(\alpha_k, \beta_k, \eta_k, \mu_k)$ be the coefficients of the $H$-expansions as in Lemma~\ref{lem:recQn-1}. Then, the series given in \eqref{eq:seriesquadric} converge uniformly on $[H_+, +\infty)$ for a certain $H_+>0$, and the resulting $(\alpha, \beta, \eta, \mu)$ give a smooth solution to \eqref{eq:abUVquadric}.

    Moreover, the following explicit quantitative estimate holds. Fix $r > 0$ and $m \in \mathbb{Z}$ with $m \geq 2$ even and assume a choice of constants $C_\alpha, C_\beta, C_\eta, C_\mu > 0$, and $C_{\mathrm{rad}} > 1$ such that the hypotheses of Lemma \ref{lem:inductive_CQ} are satisfied for any initial coefficients $(\alpha_{-2}', \beta_{-1}') \in B_r(\alpha_{-2}) \times B_r(\beta_{-1}) \subset \mathbb C^2$. Then, for $x \in \{ \alpha, \beta, \eta, \mu \}$ we have that 
    \begin{align} \begin{split} \label{eq:taylor_residual_q}
    \left| x(H) - \sum_{k \leq m} x_k H^{-k} \right| &\leq \frac{4^{m+1} C_x C_{\mathrm{rad}}^{m+1}}{|H|^{m} (|H| - 4 C_{\mathrm{rad}})}  \\
    \left| \partial_y x(H) - \sum_{k \leq m} (\partial_y x_k) H^{-k} \right| &\leq \frac{4^{m+1} C_x C_{\mathrm{rad}}^{m+1}}{r |H|^{m} (|H| - 4 C_{\mathrm{rad}})}  \quad \mbox{ for } y\in \{ \alpha_{-2}, \beta_{-1} \}
    \end{split} \end{align}
    for any $H > 4C_{\mathrm{rad}}$. In particular, we see that we can take any $H_+ > 4 C_{\mathrm{rad}}$ in the first part of the statement.
\end{proposition}
\begin{proof}
The proof of the quantitative statement follows in a completely analogous way to Proposition \ref{prop:RPseries} from Subsection \ref{subsec:catRPn}. Thus, let us focus on the qualitative statement, for which we simply need to show that there exists $r > 0$ and $C_\alpha, C_\beta, C_\eta, C_\mu > 0$, and $C_{\mathrm{rad}} > 1$ such that the hypotheses of Lemma \ref{lem:inductive_CQ} are satisfied for any $(\alpha_{-2}', \beta_{-1}') \in B_r(\alpha_{-2}) \times B_r (\beta_{-1}) \subset \mathbb C^2$. For that, we take $r$ sufficiently small so that $|(n-1)\alpha_{-2}' - (\beta_{-1}')^2|$ is lower bounded, and then we choose the constants as $ 1 \ll C_\alpha \ll C_\eta, C_\mu \ll C_\beta \ll C_{\mathrm{rad}}$, where $A\ll B$ means that $B$ is taken sufficiently large depending on $A$. It is obvious that \eqref{eq:close_alpha} holds since $C_{\mathrm{rad}}$ is sufficiently large depending on any other constant. Now, we notice from \eqref{eq:CRx-cq} that $C_{R_x} \lesssim 1 + C_\alpha$ for $x \in \{ \beta, \eta, \mu \}$ since the only terms of $C_{R_\beta}$, $C_{R_\eta}$ or $C_{R_\mu}$ without a factor of $C_{\mathrm{rad}}^{-1}$ are multiples of $C_\alpha$. Hence, we see that \eqref{eq:close-omega}--\eqref{eq:close-tau} also hold due to $C_\alpha \ll C_\eta, C_\mu$. Finally, we obtain that \eqref{eq:close-beta} holds from $C_\eta + C_\alpha \lesssim C_\beta$. 
\end{proof}

\section{The matching problem and Krawczyk's theorem}
\label{sec:krawczyk}

\subsection{The matching function}\label{subsec:matching}

In this section we use the local solutions constructed at $H=\pm\infty$ in Section~\ref{sec:Hexpansions} to reduce the existence of an $\s{SO}_{n+1}$-invariant Einstein metric on $\C\s P^n$ to the existence of a zero of a finite-dimensional matching function $\mathcal G\colon\R^4\to\R^4$.

\begin{definition}[Matching function]\label{def:matching}
	Let $(b_2,u_1,\alpha_{-2},\beta_{-1})\in\R^4$ with $b_2>0$, $\alpha_{-2} > 0$ and $(n-1)\alpha_{-2} > \beta_{-1}^2$ (so that $D_{-2} > 0$). We consider the local solution $(a, b, u, v)$ defined in $(-\infty, H_-]$ which agrees with the $H$-expansions generated by $(b_2, u_1)$ via Proposition \ref{prop:RPseries}. Similarly, we consider the solution $(\alpha, \beta,  \mu,\eta)$ agreeing with the local solution generated by $(\alpha_{-2}, \beta_{-1})$ via Proposition \ref{prop:QCseries}. We assume that both the $(a, b, u, v)$ and $(\alpha, \beta, \mu, \eta)$ solutions satisfy $D > \epsilon > 0$ for some uniform $\epsilon$. In that case, we say that our matching function $\mathcal G\colon \mathbb R^4 \to \mathbb R^4$ is \textit{well-defined} at the point $(b_2, u_1, \alpha_{-2}, \beta_{-1})$, and we define it by
    \begin{equation}\label{eq:matchingG}
		\mathcal G(b_2,u_1,\alpha_{-2},\beta_{-1}):=(a,b,u,v)(0)-(\alpha+\beta,\,\alpha-\beta,\,\mu+\eta,\,\mu-\eta)(0).
	\end{equation}
\end{definition}

\begin{remark} 
The hypothesis that $D > \epsilon > 0$ in order for $\mathcal{G}$ to be well-defined is essential, and will be checked in our proof. However, let us notice that as long as $D > \epsilon > 0$, the local solutions to \eqref{eq:abUV} and \eqref{eq:abUVquadric} can be continued up to $H = 0$, justifying that $\mathcal{G}$ is well-defined.

Assume a local solution of \eqref{eq:abUV} (equivalently, \eqref{eq:abUVquadric}) with $\mathcal G$ well-defined (thus, $D > \epsilon > 0$). We notice that $Q \geq \frac{1}{2n-1}(1 + H^2 + u^2 + v^2)$, in particular, both denominators of \eqref{eq:abUV} are lower bounded.  Gr\"onwall inequality ensures that $a$ and $b$ do not blow up in finite time, since $Q \geq \frac{1}{2n-1}(1 + H^2 + u^2 + v^2)$ so that $|da/dH| \leq C |a|$ and $|db/dH| \leq C|b|$ from \eqref{eq:abUV}, for some absolute constant $C > 0$. Once we know $a$ and $b$ remain bounded, so do $N_1$ and $N_2$ (see~\eqref{eq:abuvquantities}). Again, using our $Q$ lower bound together with $D > \epsilon$, and Gr\"onwall inequality, by~\eqref{eq:abUV}, we deduce that $u$ and $v$ do not blow up in finite time. Therefore the ODE solutions can be continued up to $H = 0$, which justifies evaluating at $H = 0$ in the definition of $\mathcal G$ in \eqref{eq:matchingG}.

\end{remark}

The intuition for the previous definition is clear. By standard ODE uniqueness theory (e.g.\ the Picard--Lindel\"of theorem), whenever the matching function vanishes, $(a,b,u,v)$ and $(\alpha+\beta,\alpha-\beta,\mu+\eta,\mu-\eta)$ represent the same local solution at $H=0$ (recall from Remark~\ref{rem:quadricvars} that the latter is a solution of \eqref{eq:abUV}). That is, we can extend $(a,b,u,v)$ to $H>0$ by simply taking $(a,b,u,v)=(\alpha+\beta,\alpha-\beta,\mu+\eta,\mu-\eta)$, and we obtain a global solution of \eqref{eq:abUV} on $\R$. 

Then, we can check that the corresponding solution satisfies the hypotheses of Proposition \ref{prop:equivalence}. First, we have that $a(H)$ and $b(H)$ are positive for $-H$ sufficiently large since $a = 1 + O(H^{-2})$ and $b = b_2 H^{-2} + O(H^{-4})$ with $b_2 > 0$. If $a$ (resp.\ $b$) takes the value $0$ for some $H \in \mathbb R$, then Gr\"onwall's inequality on $a' = 2au$ (resp.\ $b' = 2v b$) would imply that $a$ (resp.\ $b$) is constantly zero. Since that is not the case for $-H$ sufficiently large, we conclude $a, b > 0$ globally. By an analogous computation to \eqref{eq:preH}:
\[
  \frac{d}{dH} \left( \frac{H^2 - Q + (2n-1)\lambda}{D}  \right) =\frac{2 (H-n(u+v))}{(2n-1) Q}  \left( \frac{H^2 - Q + (2n-1)\lambda}{D}  \right) .
\]
Again, Gr\"onwall's inequality applied to $
\frac{\mathrm{scal}(g_t)}{D}=\frac{H^2 - Q + (2n-1)\lambda}{D}$, implies that the sign of $\frac{\mathrm{scal}(g_t)}{D}$ is constant. In particular, it cannot reach zero, so it suffices to check positivity at a single point. Since $D>0$ and, by~\eqref{eq:Qcoord}, $H^2 - Q + (2n-1)\lambda = \tfrac{n-2}{n-1}H^{2} + O(1)>0$ as $H\to-\infty$, we conclude $\frac{\mathrm{scal}(g_t)}{D}>0$ there, and hence everywhere.

Therefore, all the sign hypotheses of Proposition \ref{prop:equivalence} hold, and that yields $(h,f_1,f_2)$ generating an Einstein metric on the regular part $(0,T)\times\s G/\s H$. Moreover, due to Proposition~\ref{prop:smoothH}, we know that it closes up smoothly at both singular orbits. Smoothness at any $t\in(0,T)$ follows from classical ODE theory applied to \eqref{eq:E1}--\eqref{eq:E4}, since none of $h,f_1,f_2$ vanishes at any $t\in(0,T)$. All of this yields the following proposition.

\begin{proposition}\label{prop:gluing}
	Let $(b_2,u_1,\alpha_{-2},\beta_{-1})\in\R^4$ with $b_2 > 0$, $\alpha_{-2} > 0$, and $(n-1)\alpha_{-2} > \beta_{-1}^2$. Assume that the matching function $\mathcal G$ is well defined at $(b_2,u_1,\alpha_{-2},\beta_{-1})$, and that
	\[
	\mathcal G(b_2,u_1,\alpha_{-2},\beta_{-1})=0.
	\]
	Then the functions $(h,f_1,f_2)$ provided by Proposition \ref{prop:equivalence} define, through \eqref{eq:ansatz}, a smooth $\s{SO}_{n+1}$-invariant Einstein metric on $\C\s P^n$ with Einstein constant $\lambda=2(n+1)$.
\end{proposition}

The tool that we will employ to find a zero of $\mathcal G$ is the following fixed point theorem, which is known as Krawczyk's fixed-point theorem (or more commonly, Krawczyk's method). This is one formalization of the standard Newton method, but with a rigorous guarantee that a root exists in a certain box. Although this is classical (see e.g.~\cite[Theorem 8.2]{MKC09}), we include a proof for completeness.

\begin{theorem}[Krawczyk]\label{th:krawczyk}
	Let $X=\prod_{i=1}^{m}[r_i^-,r_i^+]\subset\R^m$ be a compact box, let $\mathcal G\colon U  \to\R^m$ be a $C^1$-map on an open neighbourhood $U \supset X$. Fix $x_0\in X$ and let $\mathcal D \subset \R^{m \times m}$ be an \emph{interval matrix} enclosing the Jacobian of $\mathcal G$ on $X$, that is, a set of the form
$$
    \mathcal D = \big\{ M \in \mathcal{M}_{m\times m}(\R) : M_{ij} \in [d_{ij}^-, d_{ij}^+] \text{ for all } i,j \big\}
$$
such that $\partial_j \mathcal G_i(x)\in [d_{ij}^-,d_{ij}^+]$ for every $x\in X$ and all $i,j$. Fix $A\in\s{GL}_{m}(\R)$, and define the set
	\begin{equation}\label{eq:krawczykset}
		\mathcal K(x_0, A, X):=x_0- A \mathcal G(x_0)+\big\{(\Id- A M)(x-x_0):\ M\in\mathcal D,\ x\in X\big\}.
	\end{equation}
	If $\mathcal K(x_0, A, X)\subset\Int X$, then $\mathcal G$ has a zero in $X$.
\end{theorem}

\begin{proof}
Define $\Phi\colon X\to\R^m$, $\Phi(x):=x- A\mathcal G(x)$ and write $\Phi=(\Phi_1,\ldots, \Phi_m)$. The Jacobian of $\Phi$ at any point of $X$ is of the form $\Id - AM$ for some $M \in \mathcal D$. By the Mean Value Theorem along the segment from $x_0$ to $x$, for each $i\in\{1,\ldots,m\}$, there exists $\widetilde{x}_i$ such that $$\Phi_i(x) = \Phi_i(x_0) + \nabla \Phi_i (\widetilde{x}_i)(x-x_0) = \left[ x_0 - A\mathcal G (x_0) + (\Id - AM)(x-x_0)\right]_i, \qquad \mbox{for some } M \in \mathcal D.$$
 We observe that this is the $i$-th component of a point in $\mathcal K (x_0, A, X)$, which is contained in $X$ by hypothesis. Hence $\Phi_i(x) \in [r_i^-, r_i^+]$, and thus $\Phi\colon X \to X$. Since $\Phi$ is continuous and $X$ is a compact convex subset of $\R^m$, Brouwer's fixed point theorem yields $x_\star \in X$ with $\Phi(x_\star) = x_\star$, which implies $A\mathcal G (x_\star) = 0$, and since $A$ is invertible, $\mathcal G (x_\star) = 0$.
\end{proof}

\begin{remark}
Although the theorem holds irrespective of $x_0$ or $A$, in practice, in order to satisfy the hypothesis of $\mathcal K(x_0, A, X) \subset \mathrm{Int}(X) $ it is convenient to choose $x_0$ to be an approximate root of $\mathcal G$ and $A$ an approximate inverse of $\partial_j \mathcal G_i(x_0)$. Then, we can take a box $X$ centered around $x_0$ of side, for example, $2\varepsilon$ with $\varepsilon := |A\mathcal G(x_0)|$, which is a very small box due to $\mathcal{G}(x_0) \approx 0$. We see from \eqref{eq:krawczykset} that this choice will satisfy the hypothesis as long as $|(\mathrm{Id} - AM) (x-x_0)| < \varepsilon$. Since $|x-x_0| \lesssim \varepsilon$ and $A$ is an approximate inverse of $\partial_j \mathcal G_i(x_0)$, the containment $\mathcal K(x_0, A, X) \subset \mathrm{Int}(X)$ will hold easily with those choices.
\end{remark}

Therefore, the last remaining ingredient is the following:
\begin{lemma}
\label{lem:matchingzero}
Let $n \in \{ 3, 4, 5, 6, 7 \}$ and consider the values of $x_0 = (b_2, u_1, \alpha_{-2}, \beta_{-1})$ given in Table~\ref{table:x0_values} for each value of $n$. Define  
$$X = b_2^I \times u_1^I \times \alpha_{-2}^I \times \beta_{-1}^I$$
where $x^I = [x-\delta, x+ \delta]$ for each $x\in \{ b_2, u_1, \alpha_{-2}, \beta_{-1} \}$ (with $\delta$ given in Table \ref{table:x0_values}). The matching function $\mathcal G$ is then well-defined in an open neighborhood of $X$. Moreover, there exists a matrix $A \in \s{GL}_4(\mathbb R)$ such that 
\begin{equation} \label{eq:final_eq}
\| A \mathcal G (x_0) \|_\infty + \delta \| \mathrm{Id} - A D \mathcal G (x) \|_{\infty \to \infty} < \delta, \qquad \forall x \in X,
\end{equation}
where $\|\cdot \|_{\infty\to \infty}$ denotes the matrix norm induced by $\|\cdot \|_{\infty}$, that is, $\|M\|_{\infty\to\infty}=\max_{i}\sum_{j}|M_{ij}|$.
\end{lemma}

\begin{proof}
	The proof is computer-assisted and the code can be found in the source of the arXiv submission of this paper as \texttt{CAP\_Einstein\_CPn\_SubmissionVersion.ipynb}. The main task is to be able to compute rigorous enclosures for $\mathcal G(x_0) $ and $D \mathcal G (x)$. In Cell~1 of the code, we compute the coefficients $(a_k, b_k, u_k, v_k)$ for $k \leq m$ according to Lemma \ref{lem:recRPn}. They are stored as polynomials in $b_2$ and $u_1$. Similarly, Cell~2 computes $(\alpha_k, \beta_k, \eta_k, \mu_k)$ using Lemma \ref{lem:recQn-1} (in that case, they are polynomials of $\alpha_{-2}$, $1/\alpha_{-2}$, $\beta_{-1}$ and $1/D_{-2}$). Cell~3 finds candidates for $C_a$, $C_b$, $C_u$, $C_v$ and $C_{\mathrm{rad}} $ and verifies they satisfy the hypotheses of Proposition \ref{prop:RPseries}. Cell~4 does the analogous task for the complex quadric singular orbit and Proposition \ref{prop:QCseries}. Then,  Cell~5 uses the results from Proposition \ref{prop:RPseries} and \ref{prop:QCseries} respectively to perform rigorous enclosures for $x (H)$ or $\partial_y x(H)$ where $x \in \{ a, b, u, v\}$ with $y \in \{ b_2, u_1 \}$ or $x\in \{ \alpha, \beta, \eta, \mu \}$ with $y \in \{ \alpha_{-2}, \beta_{-1} \}$. All the cases with $\partial_y$ require checking the hypotheses of the corresponding Proposition in complex balls $b_2 + B_r(0)$, $u_1 + B_r(0)$, $\alpha_{-2} + B_r(0)$ or $\beta_{-1} + B_r(0)$, for which we always take $r = 10^{-4}$.

    The approach described is not enough to compute the matching function. Since the hypotheses of Propositions \ref{prop:RPseries}, \ref{prop:QCseries} require $|H| \geq 4 C_{\mathrm{rad}}$ we can only evaluate $x(H)$ at explicit but large $H$ (for example, for $n  = 4$ we evaluate our series at $H = -800$ and $H = 1200$ respectively). In order to continue the solutions to $H = 0$, we notice that the ODEs \eqref{eq:rp-ode} and \eqref{eq:masterEQ_expquadric} are not singular, so we can use a standard ODE solver such as \texttt{capd::MpIOdeSolver} from CAPD. Cell 6 implements that; since CAPD is written in C++, we also implement a driver to run the code there. As we do that, we also check in this cell that the matching function $\mathcal G$ is well-defined on $X$ by checking that $D > \epsilon > 0$. 

    Combining our series from Cell~5 with the ODE solver from Cell~6, in Cell~7, we compute the matching function $\mathcal G$ and its Jacobian, which are algebraic combinations of $x(0)$ for $x\in \{a, b, u, v, \alpha, \beta, \eta, \mu \}$ or their derivatives with respect to $b_2, u_1, \alpha_{-2}, \beta_{-1}$. Then, Cell~8 implements \texttt{krawczyk\_check}, which accepts $n, \delta$ and other hyperparameters (orders $m$ and radius $r$ for our series expansions in Propositions \ref{prop:RPseries}, \ref{prop:QCseries}, values $H$ at which we evaluate each series, tolerance and order of our ODE solvers, $\ldots$) and returns a Boolean value that is true only when \eqref{eq:final_eq} holds. Then, we proceed to run the function \texttt{krawczyk\_check} for $n \in \{3, 4, 5, 6, 7 \}$ with suitable hyperparameters for each case, and obtain \texttt{True} for each of them, concluding the proof of \eqref{eq:final_eq}.

    Regarding the claim that $\mathcal G$ is well-defined on an open neighbourhood of $X$, this is checked during our computation. The ODE solver from Cell 6 obtains an enclosure for the solution, and checks using interval arithmetic that $D > 0$ on $[H_-, H_+]$ (hence, since it is continuous, is lower bounded). Regarding the intervals $(-\infty, H_-]$ or $[H_+, +\infty)$ we check it in Cell 5 using Proposition \ref{prop:RPseries} and \ref{prop:QCseries}.\qedhere    
\end{proof}

\subsection{Proof of Main~Theorem}\label{subsec:proofmain}

The only known Einstein metrics on complex projective spaces are homogeneous. These were classified by Ziller in \cite{Z82}: the Fubini-Study metric $g_{\mathrm{FS}}$, which is Kähler-Einstein, and Ziller's metric $g_Z$ on $\C\s P^{2n+1}$. We scale the Fubini-Study metric $g_{\mathrm{FS}}$ so that its minimum sectional curvature equals one. Its volume and Einstein constant are
\begin{equation*}
	\label{eq:volFS}
	\mathrm{Vol}(\C\s P^n, g_{\mathrm{FS}}) = \frac{\pi^n}{n!}, \quad \lambda=2(n+1).
\end{equation*} 

On the other hand, Ziller's metric $g_Z$ on $\C\s P^{2n+1}$ is obtained via a canonical variation of the Fubini-Study metric with respect to the twistor fibration $\s S^2 \rightarrow \C\s P^{2n+1}\rightarrow \H \s P^n$. Recall that given a Riemannian submersion $\pi\colon (M,g)\rightarrow (B,\check{g})$ with totally geodesic fiber $F$, a canonical variation of $g$ is a Riemannian metric of the form
\[g_{\kappa}=\kappa\, g_{\rvert \mathcal{V}\times \mathcal{V}} + g_{\rvert \mathcal{H}\times \mathcal{H}},\]
where $\kappa>0$ and $\mathcal{V}$ and $\mathcal{H}$ denote the vertical and horizontal distributions. Since the fibers are totally geodesic, they are pairwise isometric, and the volume of $g_{\kappa}$ is given by
\[\mathrm{Vol}(M,g_{\kappa})= \mathrm{Vol}(B,\check{g})\, \mathrm{Vol}(F,\kappa\, g_{\rvert \mathcal{V}\times \mathcal{V}})=\kappa^{\dim F/2}\,\mathrm{Vol}(M,g).\]
By~\cite{Z82}, Ziller's metric is obtained at $\kappa=\tfrac{1}{n+1}$, with volume and Einstein constant
\begin{equation*}
	\label{eq:Zillermetric}
	\mathrm{Vol}(\C\s P^{2n+1}, g_{\mathrm{Z}}) = \frac{\pi^{2n+1}}{(n+1)(2n+1)!}, \quad \lambda=\frac{4(n^2+3n+1)}{n+1}.
\end{equation*}

\begin{remark}\label{rem:Hermitian}
By Ziller's classification~\cite{Z82}, every homogeneous Einstein metric on $\C\s P^n$ is, up to homothety and isometry, either the Fubini-Study metric or, in odd complex dimension, $g_Z$. Hence, they are Hermitian with respect to the pullback of $J$ by a suitable isometry, though not necessarily with respect to the fixed standard complex structure $J$.
\end{remark}

Recall that a geodesic orbit  space is a Riemannian manifold $(M,g)$ such that every geodesic is the orbit of a $1$-parametric subgroup of $\Isom(M)$, see~\cite{BN} for more details on geodesic orbit spaces. Now we state a well-known necessary condition for a homogeneous metric to be a geodesic orbit metric that will be useful to discuss the inhomogeneity of the Einstein metrics we found.
\begin{lemma}
\label{lem:jacobigo}
Let $M$ be a complete geodesic orbit space and let $\gamma$ be a unit speed geodesic of $M$. Then the Jacobi operator $R_{\dot\gamma}$ has constant spectrum along $\gamma$.
\end{lemma}
\begin{proof}
Let $\gamma$ be a unit speed geodesic of $M$. As $M$ is a geodesic orbit space, we can assume that $\gamma(t)=\varphi_t(p)$, where $\{\varphi_t\}_{t\in \R}\leq \mathrm{Isom}(M)$ is a $1$-parametric subgroup of isometries with $\varphi_0=\Id$. Differentiating, we obtain $\dot\gamma(t)=d\varphi_t(p)(\dot\gamma(0))$. Notice that $d\varphi_t R_{\dot{\gamma}(0)}=R_{\dot{\gamma}(t)}d\varphi_t$, and as $d\varphi_t$ is a linear isometry, this implies that $R_{\dot{\gamma}(0)}$ and $R_{\dot{\gamma}(t)}$ have the same eigenvalues. Consequently, the spectrum of the Jacobi operator $R_{\dot\gamma}$ is constant along $\gamma$. 
\end{proof}

\begin{proof}[Proof of Main Theorem]
	Fix $n\in\{3,4,5,6,7\}$. By Lemma~\ref{lem:matchingzero}, the matching function has a zero $(b_2,u_1,\alpha_{-2},\beta_{-1})\in X_n$ satisfying the hypotheses of Proposition~\ref{prop:gluing}, which then provides a smooth $\s{SO}_{n+1}$-invariant Einstein metric on $\C\s P^n$ with Einstein constant $\lambda=2(n+1)>0$. 

    By~\cite{Ta99}, every $\s{Sp}_{n+1}$-invariant metric on $\C\s P^{2n+1}$ is geodesic orbit, and the same holds for the Fubini-Study metric on $\C\s P^{n}$ as it is symmetric. Thus, all homogeneous Einstein metrics on $\C\s P^n$ are geodesic orbit by the classification in~\cite{Z82}. Hence, in order to prove that the Einstein metric that we found is inhomogeneous it suffices to check that the eigenvalues of the Jacobi operator along the geodesic $\gamma$ are not constant. In particular, by Equation~\eqref{eq:jacobit} and Lemma~\ref{lem:jacobigo}, if we prove that $h''/h$ is not constant, we would be done. In order to show that, we will prove that $h''$ attains both positive and negative values on $(0, T)$ (and therefore, so does $h''/h$)

    First, we prove that there exists $t'' \in (0, T)$ such that $h''(t'') < 0$. Since $h(0) > 0$ and $h(T) = 0$, we obtain $t' \in (0, T)$ such that $h'(t') < 0$. Similarly, since $h'(0) = 0$ and $h'(t') < 0$, we obtain $t''\in (0, T)$ such that $h''(t'') < 0$.

 Now, we prove that $h''/h$ also takes positive values. We express the quotient $h''/h$ for our solution. From \eqref{eq:E4} we have that $$w:= \frac{h'}{h} = -\frac{H + (n-1)(u+v)}{2n-1} \qquad \mbox{ so that } \qquad \frac{h''}{h} = w'+w^2 = Q\partial_H w + w^2.$$
    Expanding the term $Q \partial_H w + w^2$, we obtain
    \begin{equation}
    \frac{h''}{h} = Q \frac{-1 - (n-1)(\partial_H u + \partial_H v)}{2n-1} + \frac{H^2 + (n-1)^2(u+v)^2 + 2(n-1)H(u+v)}{(2n-1)^2}, 
    \end{equation}
    Using our convergent expansions \eqref{eq:seriesRP}, which give $u + v = -\frac{H}{n-1} + (u_1 + v_1)H^{-1} + O(H^{-3})$ and $Q = \frac{H^2}{n-1} + O(1)$ as $t \to 0$ (that is, as $H \to -\infty$), we obtain
    \begin{equation}
    \lim_{t\to 0} \frac{h''}{h} = \frac{u_1 + v_1}{2n-1} = \frac{1}{n}\left( u_1 + \frac{b_2}{n-1} - 2(n+1) \right),
    \end{equation}
    where the last equality follows from the formula for $v_1$ in \eqref{eq:a2u1}. Since we know that $b_2 \in b^I, u_1 \in u^I$ for our metrics with $n \in \{3, 4, 5, 6, 7\}$ (with $u^I_1$ and $b^I_2$ given by balls of radius $\delta \leq 10^{-11}$ around the values in Table \ref{table:x0_values}), and the right-hand side is increasing in both $b_2$ and $u_1$, it is a direct computation to check that $\lim_{t \to 0} \frac{h''}{h} > 0$ in each of those cases. Thus, we see that $h''(t) / h(t) > 0$ for $t$ sufficiently small. We conclude that $h''/h$ takes both positive and negative values on $(0, T)$, hence, $h''/h$ is not constant, and therefore our metric is inhomogeneous.\qedhere
    
\end{proof}
\begin{remark}
\label{rem:sec>0}
Since the previous proof shows that $h''/h$ takes both positive and negative values, the metric~$g$ in the \hyperref[th:main]{Main Theorem} does not have non-negative sectional curvature by \eqref{eq:jacobit}.
\end{remark}

Now we prove that the Einstein metrics on \hyperref[th:main]{Main Theorem} are not Hermitian. One could be tempted to assert that proving inhomogeneity of the metric $g$ implies that $g$ is not Hermitian by the discussion in Remark~\ref{rem:Hermitian}. However, notice that Ziller's classification is up to isometry. Thus, for instance, one can take $g = \Phi^\ast g_{\mathrm{FS}}$ for an arbitrary diffeomorphism $\Phi$, then $g$ is certainly homogeneous and Einstein, but it is naturally Hermitian with respect to $\Phi^\ast J$, and thus, not necessarily $J$-Hermitian, that is, not Hermitian with respect to the fixed standard $J$.

Firstly, we note the following result which gives a necessary condition for an $\s{SO}_{n+1}$-invariant metric to be Hermitian.

\begin{lemma}\label{lem:ratio-hermitian}
	Let $g$ be an $\s{SO}_{n+1}$-invariant metric of the form \eqref{eq:ansatz} on 
	$\C\s P^n$, let $J$ be the complex structure of $\C\s P^n$, and let 
	$r\in[0,\pi/4]$ denote the Fubini-Study distance of an $\s{SO}_{n+1}$-orbit to $\R\s P^n$. If $g$ is Hermitian with respect to $J$, then
	\[
	\frac{f_1(r)}{f_2(r)}=\cot(r)\qquad\text{for }r\in(0,\pi/4).
	\]
\end{lemma}

\begin{proof}
Let us denote by $\pi\colon \s{S}^{2n+1}\rightarrow \C\s P^n$ the Hopf fibration and consider the horizontal lift $\widetilde\gamma(r)=\cos(r)e_1+\sin(r)ie_2\in\s S^{2n+1}$ of  $\gamma(r)=[\cos(r)e_1+\sin(r)ie_2]$, with~$r\in[0,\pi/4]$, which is a geodesic of $\C\s P^n$ with respect to the Fubini-Study metric. It can be checked that a Killing vector field $X^*$ induced by $X\in\g{so}_{n+1}$ satisfies
\begin{equation}
	\label{eq:pi_killing}
	X^*_{\gamma(r)}=d\pi_{\tilde{\gamma}(r)} (X\cdot \tilde{\gamma}(r)).
\end{equation}
Take the unit-$\mathcal{Q}$ generators 
$X=E_{1,3}\in\g{n}_1$ and $Y=E_{2,3}\in\g{n}_2$ from \eqref{eq:nsplitting}. Then by~\eqref{eq:pi_killing},
\begin{align*}
	X^{*}_{\gamma(r)}=d\pi_{\tilde{\gamma}(r)}( E_{1,3}\cdot\widetilde\gamma(r))=-\cos(r)d\pi_{\tilde{\gamma}(r)}(e_3), \enspace  	Y^{*}_{\gamma(r)}=d\pi_{\tilde{\gamma}(r)}( E_{2,3}\cdot\widetilde\gamma(r))=-\sin(r)d\pi_{\tilde{\gamma}(r)}(i e_3).
\end{align*}
Also, the complex structure $J$ of $\C\s P^n$ satisfies	$
	d\pi_{\tilde{\gamma}(r)}(iw)= Jd\pi_{\tilde{\gamma}(r)}(w)$ for each $w\in T_{\tilde{\gamma}(r)}\s{S}^{2n+1}$, and thus, if we fix $r\in(0,\pi/4)$,
\[
J X^{*}_{\gamma(r)}=-\cos(r) J d\pi_{\tilde{\gamma}(r)}(e_3)=-\cos(r)	d\pi_{\tilde{\gamma}(r)}(ie_3)=\cot(r)\big(-\sin(r)	d\pi_{\tilde{\gamma}(r)}(ie_3)\big)=\cot(r)Y^{*}_{\gamma(r)}.
\]	
However, by the ansatz \eqref{eq:ansatz}, $g(X^{*}_{\gamma(r)},X^{*}_{\gamma(r)})=f_1^2(r)$ and $g(Y^{*}_{\gamma(r)},Y^{*}_{\gamma(r)})=f_2^2(r)$, and if $g$ 
	is Hermitian with respect to $J$, we have
	\[
	f_1^2(r)=g(X^{*}_{\gamma(r)},X^{*}_{\gamma(r)})=g(JX^{*}_{\gamma(r)},JX^{*}_{\gamma(r)})=\cot^2(r)\,g(Y^{*}_{\gamma(r)},Y^{*}_{\gamma(r)})=\cot^2(r)\,f_2^2(r).
	\]
	Since $f_1(r),f_2(r)>0$ and $\cot(r)>0$ for  $r\in(0,\pi/4)$, we get $f_1(r)/f_2(r)=\cot(r)$. 
\end{proof}

\begin{proposition} \label{prop:notHermitian}
Let $g$ be the metric constructed on \hyperref[th:main]{Main Theorem}. Then $g$ is not $J$-Hermitian.
\end{proposition}
\begin{proof}
    We claim that if $g$ is $J$-Hermitian then $f_1'(T)\leq0$.  By Lemma~\ref{lem:ratio-hermitian},  $f_1(r)/f_2(r) = \cot (r)$. Then $r(t) = \mathrm{arccot} (f_1(t)/f_2(t))$, therefore $r'(t) = \frac{-1}{1 + (f_1(t)/f_2(t))^2} \left( f_1(t) / f_2(t) \right)'$. Notice that $r(t)$ is strictly monotone in $t$ because $r(t)$ is the Fubini-Study distance from $\R\s P^n$ along the geodesic (with respect to $g$) denoted by $\gamma$ which intersects each orbit once, see \S\ref{sec:so_n+1-action}. Then, $f_1(t)/f_2(t)$ is strictly monotone as well. Since $f_1(t)/f_2(t) \to +\infty $ as $t \to 0^+$, and $f_1(T)/f_2(T) =  1$  from \eqref{eq:smoothproj} and \eqref{eq:smoothquadric}, strict monotonicity of $f_1(t)/f_2(t)$ implies that $f_1(t) > f_2(t)$ for all $t \in (0, T)$. Since we also know from \eqref{eq:smoothquadric} that $f_1'(T) = -f_2'(T)$, this implies that $f_1'(T) \leq 0 \leq  f_2'(T)$.

Now we show that our metric $g$ satisfies $f_1'(T)>0$, contradicting the assumption that $g$ is $J$-Hermitian. Recall from~\eqref{eq:uvdef} that $f_1'/f_1 = u + h'/h$ and $f_2'/f_2 = v + h'/h$. Tracing $\mathcal S_t$ in~\eqref{eq:eigshape}, we obtain $(2n-1)h'/h = -H - (n-1)(u+v) = -H - 2(n-1)\mu$, where we have used~\eqref{eq:greekvariables}. Hence $f_1'/f_1 = \eta + \frac{\mu - H}{2n-1}$, and since $\mu_{-1} = 1$ (see~\eqref{eq:firstcoef-quad}), by Proposition~\ref{prop:smoothH} we have $\mu - H = \mu_{1}H^{-1} + O(H^{-3})$. Therefore, by \eqref{eq:firstcoef-quad}, $\lim_{t \to T^-} f_1'/f_1 = \eta_0 =-\beta_{-1}^\star/(2\alpha_{-2}^\star)$. By \eqref{eq:helsinki}, we deduce that $f_1(T)^2 = 4\alpha_{-2}^\star$. Thus, $f_1'(T) = -\beta_{-1}^\star / \sqrt{\alpha_{-2}^\star}$, which (since $\alpha_{-2}^\star > 0$) shares the sign of $-\beta_{-1}^\star$. By Lemma~\ref{lem:matchingzero}, $\beta_{-1}^\star\in[\beta_{-1}-\delta,\beta_{-1}+\delta]$ with $\delta$ and $\beta_{-1}$ as in Table~\ref{table:x0_values}. Since $\beta_{-1}<-\tfrac{1}{100}$ for every $n\in\{3,4,5,6,7\}$, we have $\beta_{-1}^\star<0$, so $f_1'(T)=-\beta_{-1}^\star/\sqrt{\alpha_{-2}^\star}>0$, contradicting $g$ being $J$-Hermitian.\qedhere
\end{proof}

\begin{remark}
\label{rem:nonNK}
The metrics in \hyperref[th:main]{Main Theorem} are not Hermitian, and thus they are not conformally Kähler. Moreover, they are not nearly Kähler when $n\ge4$. The complex projective space $\C\s P^n$ is a simply connected manifold that is not topologically a product; this can be seen by looking at its cohomology ring which is isomorphic to $\mathbb{Z}[x]/(x^{n+1})$. Recall that, on the one hand, by~\cite{Na02} an irreducible simply connected strictly nearly Kähler manifold, i.e.\ non-Kähler, belongs to one of the following three classes:
\begin{enumerate}
	\item[\textup{(a)}] a $6$-dimensional nearly Kähler manifold,
	\item[\textup{(b)}] homogeneous nearly Kähler spaces,
	\item[\textup{(c)}] twistor spaces over quaternionic Kähler manifolds.
\end{enumerate}
On the other hand, the LeBrun--Salamon conjecture, which states that every positive quaternionic Kähler manifold is symmetric, has been proved in real dimensions $8$, $12$, and $16$, see~\cite{PS91,LS94,HH02,BW22}. As the twistor spaces over these quaternionic Kähler manifolds are homogeneous, this implies that $\C\s P^{n}$ with $n\in\{4,5,6,7\}$ cannot carry an inhomogeneous nearly Kähler metric.

Recall that in dimension 6, the nearly Kähler condition is rather special and it implies Einstein. Indeed, in $\C\s P^3$, there is another non-integrable almost complex structure $\tilde{J}$ which can be obtained from the standard complex structure $J$ of $\C\s P^3$ by changing the sign of $J$ when applied to the vertical distribution of the twistor fibration $\s S^2\rightarrow\C\s P^3\rightarrow\H\s P^1$. Notice that cohomogeneity-one actions preserving the non-integrable almost complex  structure $\tilde{J}$ of $\C\s P^3$  were classified in~\cite{PK10}. In the case of $\C\s P^3$, there is only one such cohomogeneity-one action, given by $\s{SU}_2\times\s{SU}_2$, which has singular orbits diffeomorphic to $\s{S}^2$.  However, for $n=3$, the action we consider has as singular orbits $\R\s P^3$ and $\s{S}^2\times\s{S}^2$, see~\S\ref{sec:so_n+1-action}.  Consequently, the metrics in \hyperref[th:main]{Main Theorem} are not nearly Kähler with respect to $\tilde{J}$. It is interesting to point out that Foscolo and Haskins~\cite{FH17} conjectured that every cohomogeneity-one nearly Kähler metric on $\C\s P^3$ is homogeneous. 
\end{remark}

\newpage
\appendix
\section{Numerics and plots}\label{sec:appendix_A}

To carry out our proof, we need to find approximate zeros of $\mathcal{G}(x)$, as the ones in Table \ref{table:x0_values}. To do that, one can simply implement the function $\mathcal{G}(x)$ following the  series from Lemmas \ref{lem:recRPn}, \ref{lem:recQn-1} and solve the ODEs \eqref{eq:rp-ode}, \eqref{eq:masterEQ_expquadric} to compute $\mathcal{G}$. This is significantly simpler than the implementation done in Lemma \ref{lem:matchingzero}, since we are doing numerics, and no rigorous control is needed here (either on the series remainder or for the ODE solver). Then, we perform a classical Newton method to find the zeros of $\mathcal{G}(x)$, with different initialization parameters $x$. Some initialization parameters converge to zeros of $\mathcal{G}(x)$, and for $n \in \{ 3, 4, 5, 6, 7 \}$, we observe two different zeros (corresponding to the Fubini-Study metric and the metric from our \hyperref[th:main]{Main Theorem}). In contrast, we observe no new metrics for $n \geq 8$. Numerical exploration of the $n = 2$ case also confirms no new metrics there (proved by Broder's result \cite{Br26}). 

After finding those approximate $x_0$ with $\mathcal{G}(x_0) \approx 0$, we use those values of $(b_2, u_1, \alpha_{-2}, \beta_{-1})$ to generate the corresponding solutions $(a, b, u, v)$ to \eqref{eq:rp-ode} or $(\alpha, \beta, \eta, \mu)$ of \eqref{eq:masterEQ_expquadric}. Then, Proposition \ref{prop:equivalence} (and Remark \ref{rem:quadricvars}) allow us to compute $f_1, f_2, h$. We show the numerical values of $(b_2, u_1, \alpha_{-2}, \beta_{-1})$ in Table \ref{table:x0_values} and a plot of $f_1, f_2, h$ for the metrics of \hyperref[th:main]{Main Theorem} in Figure \ref{fig:plots}.

\begin{table}[ht]
  \centering
  \small
  \begin{tabular}{|c|l|l|l|l|l|}
    \hline
    $n$ & \multicolumn{1}{c|}{$b_2$} & \multicolumn{1}{c|}{$u_1$} & \multicolumn{1}{c|}{$\alpha_{-2}$} & \multicolumn{1}{c|}{$\beta_{-1}$} & $\delta$ \\
    \hline
    $3$ & $21.090615289608624$ & $4.2769888845669771$ & $0.0248433475704649$ & $-0.045831878718430$ & $10^{-14}$\\
    $4$ & $39.073179213574893$ & $6.0328404911983641$ & $0.0559943268417511$ & $-0.042848398677701$ & $10^{-11}$\\
    $5$ & $69.506625525589775$ & $8.4713521979792653$ & $0.0750978465264258$ & $-0.031246038856108$ & $10^{-12}$\\
    $6$ & $120.06632755687508$ & $12.289082277882205$ & $0.0868218947390250$ & $-0.020103748103872$ & $10^{-14}$\\
    $7$ & $229.69165123519819$ & $20.840270491124821$ & $0.0944735712466017$ & $-0.010001596495958$ & $10^{-14}$\\
    \hline
  \end{tabular}\vspace{0.2cm}
  \caption{Numerical values of the parameters $x_0=(b_2,u_1,\alpha_{-2},\beta_{-1})$ for each $n\in\{3,4,5,6,7\}$, as used in Lemma~\ref{lem:matchingzero}. We also include the value of $\delta$ we take to define $b_2^I, u_1^I, \alpha_{-2}^I$ and $\beta_{-1}^I$.}
  \label{table:x0_values}
\end{table}

Finally, let us recall the notion of Yamabe energy displayed in Table~\ref{table:yamabe}. Given a closed Riemannian manifold $(M, g)$ of dimension $m$, the Yamabe energy of the conformal class $[g]$ is 
\begin{equation*}
Y([g]):=\inf_{\tilde g\in[g]}\ \frac{\int_M \mathrm{scal}(\tilde g)\, d\mathrm{vol}_{\tilde g}}{\mathrm{Vol}(M,\tilde g)^{\frac{m-2}{m}}}.
\end{equation*}
Since $\C\s P^n$ with $n\ge 2$ is not diffeomorphic to the round sphere, the resolution of the Yamabe problem together with~\cite[Proposition~6.2]{Ob72} implies that if $g$ is Einstein, then
\begin{equation*}
Y([g]) = 2n\,\lambda\,\mathrm{Vol}(\C\s P^n, g)^{1/n},
\end{equation*}
which we evaluate using the volume formula \eqref{eq:volformula} applied to our numerical solutions.

\begin{table}[ht]
  \centering
  \small
  \begin{tabular}{|l|r|r|r|r|}
    \hline
    $n$ & $T$ & $Y([g])$ & $Y([g])/Y([g_{\mathrm{FS}}])$ & $Y([g_{\mathrm{Z}}])/Y([g_{\mathrm{FS}}])$ \\
    \hline
    $3$ & $1.5899$ & $64.601$  & $0.7785$ & $0.9921$ \\
    $4$ & $1.4951$ & $97.559$  & $0.8592$ & --- \\
    $5$ & $1.4575$ & $130.010$ & $0.8984$ & $0.9811$ \\
    $6$ & $1.4493$ & $162.420$ & $0.9213$ & --- \\
    $7$ & $1.4752$ & $194.899$ & $0.9361$ & $0.9741$ \\
    \hline
    Fubini-Study & $\pi/4$ & $\dfrac{4n(n+1)\pi}{(n!)^{1/n}}$ & $1$ & --- \\
    \hline
  \end{tabular}
  \vspace{0.2cm}
  \caption{Numerically computed values of $T$, the Yamabe energy
  $Y([g])$ of the Einstein metrics from \hyperref[th:main]{Main Theorem} and its ratio to the Yamabe energy of the Fubini-Study metric. The last column gives, in the odd dimensions where Ziller's metric $g_{\mathrm{Z}}$ exists, the ratio $Y([g_{\mathrm{Z}}])/Y([g_{\mathrm{FS}}])$. Interestingly the ratios between $g$ and $g_{\mathrm{FS}}$ increase with dimension and the ratios between $g_Z$ and $g_{\mathrm{FS}}$ decrease. }
  \label{table:yamabe}
\end{table}
\begin{figure}[h!]
\centering
\includegraphics[width=0.9\textwidth]{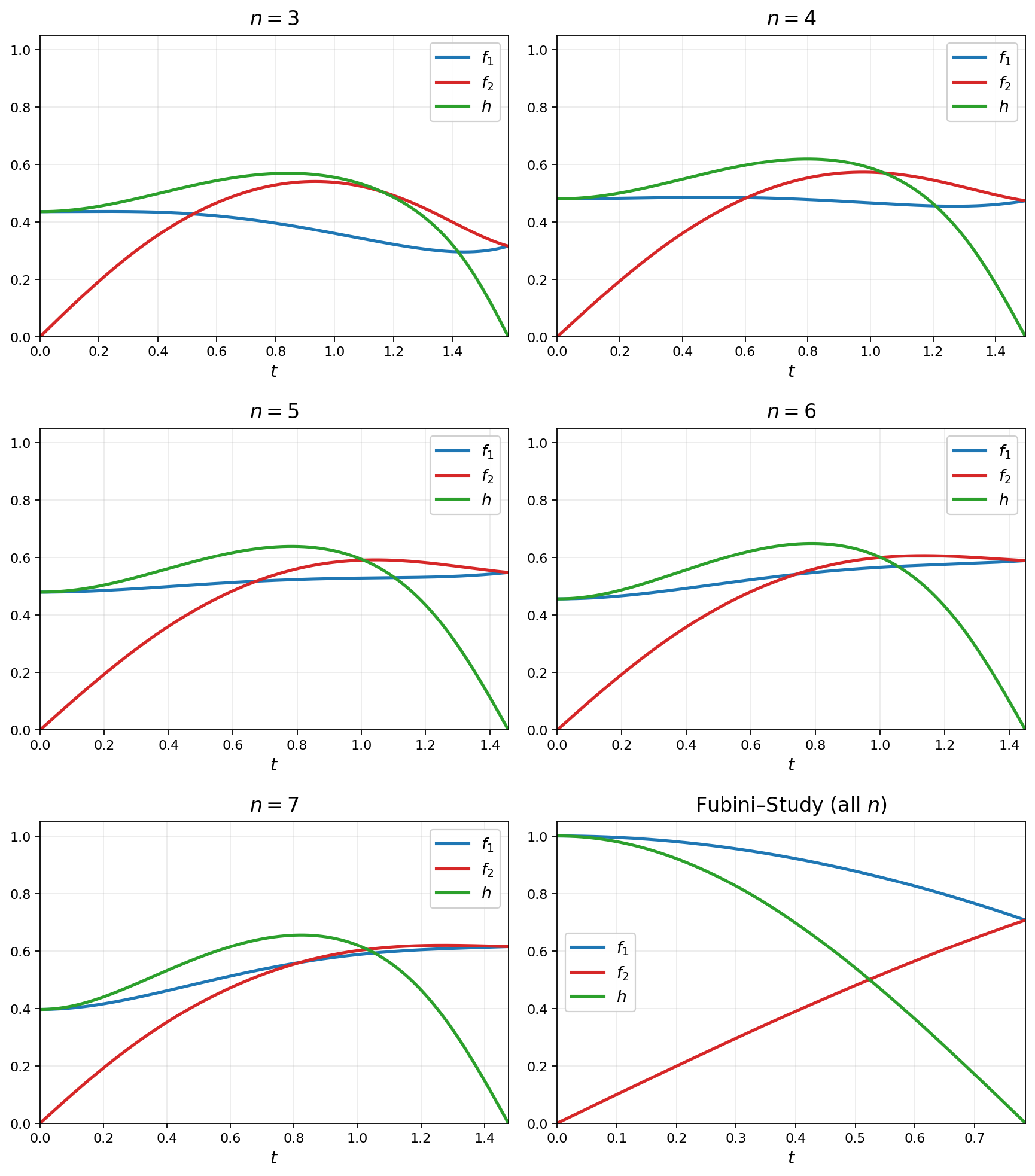}
\caption{We plot the functions $f_1, f_2, h$ corresponding to the new Einstein metrics from \hyperref[th:main]{Main Theorem} as well as for the Fubini-Study metric (for which $f_1, f_2, h$ are independent of $n$). The coordinate $t$ ranges between $0$ and $T$, with $T$ depending on the specific metric (and given in Table \ref{table:yamabe}). }
\label{fig:plots}
\end{figure}
 
\newpage
{\footnotesize  
\setlength{\bibsep}{0pt plus 0.3ex} 

}

\end{document}